\documentclass[11pt,a4paper]{article}
\usepackage{fullpage}
\usepackage{booktabs}
\usepackage{multirow}
\usepackage[utf8]{inputenc}
\usepackage{amsthm,amssymb}
\usepackage[leqno]{amsmath}
\usepackage{thmtools}
\usepackage{mathtools}
\usepackage{caption,subcaption}
\usepackage{alphalph}

\usepackage{enumerate}
\usepackage[sort]{cite}
\usepackage{titling,authblk}
\usepackage{xparse}
\usepackage{hyperref}
\hypersetup{
    colorlinks=true,       % false: boxed links; true: colored links
    linkcolor=blue,        % color of internal links (change box color with linkbordercolor)
    citecolor=red,         % color of links to bibliography
    filecolor=magenta,     % color of file links
    urlcolor=cyan,         % color of external links
    linktocpage=true
}

\makeatletter
 \newcommand{\linkdest}[1]{\Hy@raisedlink{\hypertarget{#1}{}}}
\makeatother

\makeatletter
\g@addto@macro\normalsize{%
  \setlength\abovedisplayskip{6.6pt}
  \setlength\belowdisplayskip{6.6pt}
  \setlength\abovedisplayshortskip{6.6pt}
  \setlength\belowdisplayshortskip{6.6pt}
}
\makeatother

\usepackage{float}
\usepackage{rotating}
\usepackage{array}
\usepackage{lscape}
\usepackage[linesnumbered,ruled,vlined]{algorithm2e}
\usepackage{enumitem}
\usepackage{tikz}

\newcommand\thisenumsymbol{}

\allowdisplaybreaks[4] % Allow pagebreaks in equations. 1 is "please", 4 is "do it!"

\theoremstyle{plain}
\newtheorem{thm}{Theorem}
\newtheorem{lemma}[thm]{Lemma}
\newtheorem{cor}[thm]{Corollary}
\newtheorem{claim}[thm]{Claim}

\theoremstyle{definition}

\usepackage{ifthen}
\newcommand{\p}[3]{\ifthenelse{\equal{#3}{}}{p_{[#1|#2]}}{p_{[#1|#2|#3]}}}
\newcommand{\lin}{\ell^{\mathrm{in}}}
\newcommand{\lout}{\ell^{\mathrm{out}}}
\newcommand{\uin}{u^{\mathrm{in}}}
\newcommand{\uout}{u^{\mathrm{out}}}
\newcommand{\MyComment}[1]{\begin{flushright} \ttfamily\textcolor{black!50}{//~#1} \end{flushright} \vspace{-8pt}}

\SetCommentSty{mycommfont}

\DeclareMathOperator{\ei}{\ell_{\infty}}
\DeclareMathOperator{\eo}{\ell_1}
\DeclareMathOperator*{\spa}{span}

\makeatletter
\let\oldnl\nl
\newcommand{\nonl}{\renewcommand{\nl}{\let\nl\oldnl}}
\makeatother

\makeatletter
\NewDocumentCommand{\makemathbox}{O{\width} m}{%
  \def\makemathbox@##1##2{\makebox[#1]{$##1##2$}}%
  \mathpalette\makemathbox@{#2}%
}
\makeatother

\SetKwFor{Whileblock}{while}{do}{}
\SetKwBlock{Begin}{}{}
\SetVlineSkip{1.5pt}

\newcommand{\cF}{\mathcal{F}}
\newcommand{\cI}{\mathcal{I}}
\newcommand{\cO}{\mathcal{O}}

\def\final{0}  % set this to 1 to get a comment-free version
\def\iflong{\iffalse}
\ifnum\final=0  %namely if we allow comments in the output
\newcommand{\krnote}[1]{{\color{red}[{\tiny \textbf{Krist{\'o}f:} \bf #1}]\marginpar{\color{red}*}}}
\newcommand{\kinote}[1]{{\color{teal}[{\tiny \textbf{Kitti:} \bf #1}]\marginpar{\color{teal}*}}}
\newcommand{\lnote}[1]{{\color{purple}[{\tiny \textbf{Lydia:} \bf #1}]\marginpar{\color{purple}*}}}
\else % in this case [final=1] we don't want any comments to show
\newcommand{\krnote}[1]{}
\newcommand{\kinote}[1]{}
\newcommand{\lnote}[1]{}
\fi

\title{Newton-type algorithms for inverse optimization II:\\
weighted span objective}
\date{}

\thanksmarkseries{alph}
\author{Krist{\'o}f B{\'e}rczi} 
\author{Lydia Mirabel Mendoza-Cadena}
\author{Kitti Varga}
\affil{{\footnotesize MTA-ELTE Momentum Matroid Optimization Research Group and ELKH-ELTE Egerv\'ary Research Group, Department of Operations Research, E\"otv\"os Lor\'and University, Budapest, Hungary. Email: \texttt{kristof.berczi@ttk.elte.hu, lmmendoza@proton.me, vkitti@math.bme.hu}.}}

\begin{document}
\maketitle

\begin{abstract}
In inverse optimization problems, the goal is to modify the costs in an underlying optimization problem in such a way that a given solution becomes optimal, while the difference between the new and the original cost functions, called the deviation vector, is minimized with respect to some objective function. The $\eo$- and $\ei$-norms are standard objectives used to measure the size of the deviation. Minimizing the $\eo$-norm is a natural way of keeping the total change of the cost function low, while the $\ei$-norm achieves the same goal coordinate-wise. Nevertheless, none of these objectives is suitable to provide a \emph{balanced} or \emph{fair} change of the costs.

In this paper, we initiate the study of a new objective that measures the difference between the largest and the smallest weighted coordinates of the deviation vector, called the \emph{weighted span}. We give a min-max characterization for the minimum weighted span of a feasible deviation vector, and provide a Newton-type algorithm for finding one that runs in strongly polynomial time in the case of unit weights. 

\medskip

    
\end{abstract}

%%%%%%%%%%%%%%%%%%%%%%%%%%%%%%%%
\section{Introduction}
\label{sec:intro}
%%%%%%%%%%%%%%%%%%%%%%%%%%%%%%%%

Informally, in an inverse optimization problem, we are given a feasible solution to an underlying optimization problem together with a linear objective function, and the goal is to modify the objective so that the input solution becomes optimal. Such problems were first considered by Burton and Toit~\cite{burton1992instance} in the context of inverse shortest paths, and found countless applications in various areas ever since. We refer the interested reader to \cite{richter2016inverse} for the basics of inverse optimization, to \cite{heuberger2004inverse,demange2014introduction} for surveys, and to \cite{berczi2023infty,ahmadian2018algorithms} for quick introductions.

There are several ways of measuring the difference between the original and the modified cost functions, the $\eo$- and $\ei$-norms being probably the most standard ones. Minimizing the $\eo$-norm of the deviation vector means that the overall change in the costs is small, while minimizing the $\ei$-norm results in a new cost function that is close to the original one coordinate-wise. However, these objectives do not provide any information about the relative magnitude of the changes on the elements compared to each other. Indeed, a deviation vector of $\ei$-norm $\delta$ might increase the cost on an element by $\delta$ while decrease it on another by $\delta$, thus resulting in a large relative difference between the two. Such a solution may not be satisfactory in situations when the goal is to modify the costs in a fair manner, and the magnitude of the individual changes is not relevant. 

To overcome these difficulties, we consider a new type of objective function that measures the difference between the largest and the smallest weighted coordinates of the deviation vector, called the \emph{weighted span}\footnote{The notion of span appears under several names in various branches of mathematics, such as \emph{range} in statistics, \emph{amplitude} in calculus, or \emph{deviation} in engineering.}. Although being rather similar at first glance, the $\ell_\infty$-norm and the span behave quite differently as the infinite norm measures how far the coordinates of the deviation vector are from zero, while the span measures how far the coordinates of the deviation vector are from each other. In particular, it might happen that one has to change the cost on each element by the same large number, resulting in a deviation vector with large $\ei$-norm but with span equal to zero. To the best of our knowledge, this objective was not considered before in the context of inverse optimization. 

The present work is the second member of a series of papers that aims at providing simple combinatorial algorithms for inverse optimization problems under different objectives. In the first part~\cite{berczi2023infty} of the series, the authors considered the weighted bottleneck Hamming distance and weighted $\ei$-norm objectives, and proposed an algorithm that determines an optimal deviation vector efficiently. The algorithm is based on the following Newton-type scheme: in each iteration, we check if the input solution is optimal, and if it is not, then we ``eliminate'' the current optimal solution by balancing the cost difference between them. 

Here we work out the details of an analogous algorithm for the weighted span objective. As it turns out, finding an optimal deviation vector under this objective is significantly more challenging than it was for the weighted $\ei$-norm. Intuitively, the complexity of the problem is caused by the fact that the underlying optimization problem may have feasible solutions of different sizes, and one has to balance very carefully between increasing the costs on certain elements while decreasing on others to obtain a feasible deviation vector, especially when the coordinates of the deviation vector are ought to fall within given lower and upper bounds.

%%%%%%%%%%%%%%%%
\paragraph*{Previous work.}
%%%%%%%%%%%%%%%%

In \emph{balanced optimization problems}, the goal is to find a feasible solution such that the difference between the maximum and the minimum weighted variable defining the solution is minimized. Martello, Pulleyblank, Toth, and de Werra~\cite{martello1984balanced}, Camerini, Maffioli, Martello, and Toth~\cite{camerini1986most}, and Ficker, Spieksma and Woeginger~\cite{ficker2018robust}  considered the balanced assignment and the balanced spanning tree problems. Duin and Volgenant~\cite{duin1991minimum} introduced a general solution scheme for minimum deviation problems that is also suited for balanced optimization, and analyzed the approach for spanning trees, paths and Steiner trees in graphs. Ahuja~\cite{ahuja1997balanced} proposed a parametric simplex method for the general balanced linear programming problem as well as a specialized version for the balanced network flow problem. Scutell{\`a}~\cite{scutella1998strongly} studied the balanced network flow problem in the case of unit weights, and showed that it can be solved by using an extension of Radzik's~\cite{radzik1993parametric} analysis of Newton's method for linear fractional combinatorial optimization problems. 

An analogous approach was proposed in the context of $\ei$-norm objective by Zhang and Liu~\cite{zhang2002general}, who described a model that generalizes numerous inverse combinatorial optimization problems when no bounds are given on the changes. They exhibited a Newton-type algorithm that determines an optimal deviation vector when the inverse optimization problem can be reformulated as a certain maximization problem using dominant sets. 

%%%%%%%%%%%%%%%%
\paragraph*{Problem formulation.}
%%%%%%%%%%%%%%%%

We denote the sets of \emph{real} and \emph{positive real} numbers by $\mathbb{R}$ and $\mathbb{R}_+$, respectively. For a positive integer $k$, we use $[k]\coloneqq \{1,\dots,k\}$. Let $S$ be a ground set of size~$n$. Given subsets $X,Y\subseteq S$, the \emph{symmetric difference} of $X$ and $Y$ is denoted by $X\triangle Y\coloneqq (X\setminus Y)\cup(Y\setminus X)$. For a weight function $w\in\mathbb{R}_+^S$, the total sum of its values over $X$ is denoted by $w(X)\coloneqq \sum \{ w(s) \mid s\in X \}$, where the sum over the empty set is always considered to be $0$. Furthermore, we define $\frac{1}{w}(X)\coloneqq \sum \{ \frac{1}{w(s)} \bigm| s\in X \}$, and set $\|w\|_{\text{-}1}\coloneqq \frac{1}{w}(S)$. When the weights are rational numbers, then the values can be re-scaled as to satisfy $1/w(s)$ being an integer for each $s\in S$. Throughout the paper, we assume that $w$ is given in such a form  without explicitly mentioning it, implying that $\frac{1}{w}(X)$ is a non-negative integer for every $X\subseteq S$. 

Let $S$ be a finite ground set, $\cF \subseteq 2^S$ be a collection of \emph{feasible solutions} for an underlying optimization problem, $F^* \in \cF$ be an \emph{input solution}, $c \in \mathbb{R}^S$ be a \emph{cost function}, $w\in \mathbb{R}_+^S$ be a positive \emph{weight function}, and $\ell\colon S\to\mathbb{R}\cup\{-\infty\}$ and $u \colon S\to\mathbb{R}\cup\{+\infty\}$ be lower and upper bounds, respectively, such that $\ell \leq u $. We assume that an oracle $\cO$ is also available that determines an optimal solution of the underlying optimization problem $(S, \cF, c')$ for any cost function $c'\in\mathbb{R}^S$.

In the \emph{constrained minimum-cost inverse optimization problem under the weighted span objective} $\big( S, \cF, F^*, c, \ell, u, \spa_w(\cdot) \big)$, we seek a \emph{deviation vector} $p \in \mathbb{R}^S$ such that 
\begin{enumerate}[label=(\alph*)]\itemsep0em
\item \label{it:a} $F^*$ is a minimum-cost member of $\cF$ with respect to $c-p$,
\item \label{it:b} $p$ is within the bounds $\ell \leq p\leq u$, and
\item \label{it:c} $ \spa_w(p)\coloneqq \max \left\{ w(s)\cdot p(s) \bigm| s \in S \right\} - \min \left\{ w(s)\cdot p(s) \bigm| s \in S \right\}$ is minimized.
\end{enumerate}
Due to the lower and upper bounds $\ell$ and $u$, it might happen that there exists no deviation vector $p$ satisfying the requirements. A deviation vector is called \emph{feasible} if it satisfies conditions~\ref{it:a} and~\ref{it:b}, and \emph{optimal} if, in addition, it attains the minimum in~\ref{it:c}. We denote the problem by $\big( S, \cF, F^*, c, -\infty, +\infty, \spa_w(\cdot) \big)$ when no bounds are given on the coordinates of $p$ at all, and call such a problem \emph{unconstrained}. 

As an extension, we also consider \emph{multiple underlying optimization} problems at the same time. In this setting, instead of a single cost function, we are given $k$ cost functions $c^1,\dots,c^k$ together with an input solution $F^*$, and our goal is to find a single vector $p$ with $\ell \leq p \leq u$ such that $F^*$ has minimum cost with respect to $c^j-p$ for all $j \in [k]$. In other words, condition~\ref{it:a} modifies to
\begin{enumerate}[label=(\alph*')]\itemsep0em
\item \label{it:a'} $F^*$ is a minimum-cost member of $\cF$ with respect to $c^j-p$ for $j\in[k]$.
\end{enumerate}
In case of multiple cost functions, we use $\{c^j\}_{j\in[k]}$ instead of $c$ when denoting the problem.

%%%%%%%%%%%%%%%%
\paragraph*{Our results.}
%%%%%%%%%%%%%%%%

In \cite{berczi2023infty}, the authors gave Newton-type algorithms that determine an optimal deviation vector for bound-constrained minimum-cost inverse optimization problems under the weighted bottleneck Hamming distance and weighted $\ei$-norm objectives. Here our main result is an algorithm for the weighted span objective that works along the same line. However, due to the different nature of the span, the algorithm and its analysis is significantly more complicated than the one for the $\ei$-norm. We provide a min-max characterization for the weighted span of an optimal deviation vector in the unconstrained setting, i.e.\ when $\ell\equiv-\infty$ and $u \equiv+\infty$. Then we give an algorithm for finding an optimal deviation vector that makes $O(n^2\cdot\|w\|_{\text{-}1}^6)$ calls to the oracle $\cO$. In particular, the algorithm runs in strongly polynomial time for unit weights if the oracle $\cO$ can be realized by a strongly polynomial algorithm. Finally, we briefly explain how to solve the problem when multiple cost functions are given instead of a single one.

\medskip

The rest of the paper is organized as follows. In Section~\ref{sec:specdev}, we show that it is enough to look for an optimal deviation vector having a special form. The algorithm for the case of a single cost function in the constrained setting is presented in Section~\ref{sec:algorithm}. Finally, we give a min-max characterization of the weighted span of an optimal deviation vector in the unconstrained setting as well as a sketch of the extension of the algorithm for multiple cost functions in Section~\ref{sec:minmax}.

%%%%%%%%%%%%%%%%%%%%%%%%%%%%%%%%
\section{Optimal deviation vectors} 
\label{sec:specdev}
%%%%%%%%%%%%%%%%%%%%%%%%%%%%%%%%

For the weighted $\ei$-norm objective, the authors~\cite{berczi2023infty} verified the existence of an optimal deviation vector that corresponds to decreasing the costs on the elements of $F^*$ and increasing the costs on the elements of $S\setminus F^*$ by the same value $\delta\geq 0$, scaled by the reciprocal of the weights and truncated according to the lower and upper bound constraints. 

We first show that an analogous statement holds for the weighted span as well, which serves as the basic idea of our algorithm. Assume for a moment that there are no bound constraints and that $w$ is the unit weight function. If the members of $\cF$ have the same size, then one can always decrease the costs on the elements of $F^*$ by some $\delta$ in such a way that $F^*$ becomes a minimum-cost member of $\cF$. As another extreme case, if $F^*$ is the unique minimum- or maximum-sized member of $\cF$, then one can always shift the costs by the same number $\Delta$ in such a way that $F^*$ becomes a minimum-cost member of $\cF$. The idea of our approach is to combine these two types of changes in the general case by tuning the parameters $\delta,\Delta$, while also taking the weights and the bound constraints into account.

Consider an instance $\big( S, \cF, F^*, c, \ell, u, \spa_w(\cdot) \big)$ of the constrained minimum-cost inverse optimization problem under the weighted span objective, where $w\in\mathbb{R}^S_+$ is a positive weight function. For any $\delta, \Delta \in \mathbb{R}$, let $\p{\delta, \Delta}{\ell, u}{w} \colon S\to\mathbb{R}$ be defined as 
\begin{equation*}
\p{\delta, \Delta}{\ell, u}{w}(s) \coloneqq 
\begin{cases}
(\delta + \Delta)/w(s) & \text{if $s \in F^*$ and $\ell(s) \le (\delta + \Delta)/w(s) \le u(s)$}, \\
\ell(s) & \text{if $s \in F^*$ and $(\delta + \Delta)/w(s) < \ell(s)$}, \\
u(s) & \text{if $s \in F^*$ and $u(s) < (\delta + \Delta)/w(s)$}, \\
\Delta/w(s) & \text{if $s \in S\setminus  F^*$ and $\ell(s) \le \Delta/w(s) \le u(s)$}, \\
\ell(s) & \text{if $s \in S\setminus  F^*$ and $\Delta/w(s) < \ell(s)$}, \\
u(s) & \text{if $s \in S\setminus  F^*$ and $u(s) < \Delta/w(s)$}.
\end{cases}
\end{equation*}
We simply write $\p{\delta, \Delta}{}{w}$ when $\ell\equiv-\infty$ and $u\equiv+\infty$. The following observation shows that there exists an optimal deviation vector of special form.

\begin{lemma} \label{lem:inv_min_cost_span_deltas}
Let $\big( S, \mathcal{F}, F^*, c, \ell, u, \spa_w (\cdot) \big)$ be a feasible minimum-cost inverse optimization problem and let $p$ be an optimal deviation vector. Then $\p{\delta, \Delta}{\ell, u}{w}$ is also an optimal deviation vector, where $\Delta \coloneqq \min \{ w(s) \cdot p(s) \mid s \in S \}$ and $\delta \coloneqq \max \{ w(s) \cdot p(s) \mid s \in S \} - \Delta$.
\end{lemma}
\begin{proof}
The lower and upper bounds $\ell \le \p{\delta, \Delta}{\ell, u}{w} \le u$ hold by definition, hence \ref{it:b} is satisfied.

Now we show that \ref{it:a} holds. The assumption $\ell\leq p\leq u$ and the definitions of $\Delta$ and $\delta$ imply that $\ell(s) \le p(s) \le (\delta + \Delta)/w(s)$ and $\Delta/w(s) \le p(s) \le u(s)$ hold for every $s\in S$. Then for an arbitrary solution $F\in\mathcal{F}$, we get
\begin{align*}
&(c - \p{\delta, \Delta}{\ell, u}{w})(F^*) - (c - \p{\delta, \Delta}{\ell, u}{w})(F) \\[2pt]
{}&{}~~= 
\left( c(F^*) - \sum_{s \in F^*} \p{\delta, \Delta}{\ell, u}{w}(s) \right) - \left( c(F) - \sum_{s \in F} \p{\delta, \Delta}{\ell, u}{w}(s) \right) \\
{}&{}~~= 
\left (c(F^*) - \sum_{\makemathbox[.5\width]{s \in F^* }} \min\left\{\tfrac{\delta + \Delta}{w(s)}, u(s)\right\}\right)- \left(c(F) - \sum_{\makemathbox[.5\width]{s \in F \cap F^*} } \min\left\{\tfrac{\delta + \Delta}{w(s)}, u(s)\right\} - \sum_{\makemathbox[.5\width]{s \in F\setminus F^*}}  \max\left\{\ell(s), \tfrac{\Delta}{w(s)}\right\}\right) \\[2pt]
&~~= 
c(F^*) - c(F) - \sum_{\makemathbox[.6\width]{s \in F^*\setminus F}} \min\left\{\tfrac{\delta + \Delta}{w(s)},\ u(s)\right\} + \sum_{\makemathbox[.6\width]{s \in F\setminus F^*}} \max\left\{\ell(s),\ \tfrac{\Delta}{w(s)}\right\} \\[2pt]
&~~\le 
c(F^*) - c(F) - \sum_{s \in F^*\setminus F} p(s) + \sum_{s \in F\setminus F^*} p(s) \\[2pt]
&~~= 
\big( c(F^*) - p(F^*) \big) - \big( c(F) - p(F) \big)\\[2pt]
&~~\le 
0,
\end{align*}
where the last inequality holds by the feasibility of $p$. 

Finally, to see that \ref{it:c} holds for $\p{\delta, \Delta}{\ell, u}{w}$, observe that $\spa_w(p) = {\max \{ w(s) \cdot p(s) \mid s \in S \}} \allowbreak - \min \{ w(s) \cdot p(s) \mid s \in S \} = \delta$ and $\spa_w (\p{\delta, \Delta}{\ell, u}{w}) \le (\delta + \Delta) - \Delta = \delta$. That is, $\p{\delta, \Delta}{\ell, u}{w}$ is also optimal, concluding the proof of the lemma.
\end{proof}

\begin{cor} \label{cor:inv_min_cost_span_opt_dev_vector}
Let $\big( S,\mathcal{F},F^*,c,\ell,u,\spa_w( \cdot ) \big)$ be a feasible minimum-cost inverse optimization problem. Then there exist $\delta, \Delta \in \mathbb{R}$ such that $\p{\delta, \Delta}{\ell, u}{w}$ is an optimal deviation vector with 
\begin{align*}
\min \left\{ w(s) \cdot \p{\delta, \Delta}{\ell, u}{w}(s) \bigm| s \in S \right\} &= \Delta,\ \text{and}\\
\max \left\{ w(s) \cdot \p{\delta, \Delta}{\ell, u}{w}(s) \bigm| s \in S \right\}& = \delta + \Delta.
\end{align*}
Moreover,
\begin{align*}
\Delta &\le \min \left\{ w(s) \cdot u(s) \bigm| s \in S \right\},\ \text{and}\\
\delta + \Delta &\ge \max \left\{ w(s) \cdot \ell(s) \bigm| s \in S \right\}.    
\end{align*}
\end{cor}
\begin{proof}
The first half is straightforward from Lemma~\ref{lem:inv_min_cost_span_deltas}. Since $\ell(s) \le \p{\delta, \Delta}{\ell, u}{w}(s) \le u(s)$ and $\Delta \le w(s) \cdot \p{\delta, \Delta}{\ell, u}{w}(s) \le \delta + \Delta$ hold for any $s \in S$, the second statement follows.
\end{proof}

%%%%%%%%%%%%%%%%%%%%%%%%%%%%%%%%
\section{Algorithm} 
\label{sec:algorithm}
%%%%%%%%%%%%%%%%%%%%%%%%%%%%%%%%

Consider an instance $\big( S, \cF, F^*, c, \ell, u, \spa_w(\cdot) \big)$ of the constrained minimum-cost inverse optimization problem under the weighted span objective. The aim of this section is to give an algorithm that either finds a feasible deviation vector with minimum weighted span, or recognizes that the problem is infeasible. Let us highlight the main ideas towards achieving this.
\smallskip

\noindent \textbf{Step 1. Deviation vector.} If the problem is feasible, then there exists an optimal deviation vector of the form $\p{\delta, \Delta}{\ell, u}{w}$ for some choice of $\delta,\Delta\in\mathbb{R}$ by Corollary~\ref{cor:inv_min_cost_span_opt_dev_vector}. Hence the problem reduces to identifying the values of $\Delta$ and $\delta+\Delta$.
\medskip

\noindent \textbf{Step 2. Guessing $\boldsymbol{\Delta}$ and $\boldsymbol{\delta+\Delta}$.} With the help of Corollary~\ref{cor:inv_min_cost_span_opt_dev_vector}, we identify two sets of intervals $\cI_\Delta$ and $\cI_{\delta+\Delta}$ such that the optimal $\Delta$ and $\delta+\Delta$ are contained in a member of $\cI_\Delta$ and $\cI_{\delta+\Delta}$, respectively. Since $|\cI_\Delta|+|\cI_{\delta+\Delta}|\leq n$, this enables us to reduce the original problem to $O(n^2)$ subproblems, in each of which $\Delta$ and $\delta+\Delta$ are restricted to lie within fixed intervals. The optimal solution of the original problem is then the best of the optimal solutions for these subproblems.
\medskip

\noindent \textbf{Step 3. Modifying the bound-constraints.} Fixing the intervals for $\Delta$ and $\delta+\Delta$ enables us to simplify the lower and upper bound-constraints on the coordinates of the deviation vector. This step changes neither the feasibility nor the set of optimal deviation vectors of the form $\p{\delta, \Delta}{\ell, u}{w}$ for the corresponding subproblem.
\medskip

\noindent \textbf{Step 4. Characterizing feasibility.} Once the bound-constraints are simplified, we can characterize the feasibility of the problem. This characterization is essential as, by recognizing infeasible instances, it serves as a stopping rule in the algorithm.    
\medskip

\noindent \textbf{Step 5. Solving the subproblems.} The idea of the algorithm is to eliminate bad sets in $\cF$ iteratively. Assume for a moment that no bounds are given on the deviation vector and that $w$ is the unit weight function. Let us call a set $F\in\cF$ \emph{bad} if it has smaller cost than the input solution $F^*$ with respect to the current cost function, \emph{small} if $|F|<|F^*|$ and \emph{large} if $|F|>|F^*|$. A small or large bad set can be eliminated by decreasing or increasing the cost on every element by the same value, respectively. Note that such a step does not change the span of the deviation vector. Therefore, it is not enough to concentrate on single bad sets, as otherwise it could happen that we jump back and forth between the same pair of small and large bad sets by alternately decreasing and increasing the costs. To avoid the algorithm changing the costs in a cyclic way, we keep track of the small and large bad sets that were found the latest. If in the next iteration we find a bad set $F$ having the same size as $F^*$, then we drop the latest small and large bad sets, if they exist, and eliminate $F$ on its own. However, if we find a small or large bad set $F$, then we eliminate it together with the latest large or small bad set, if exists, respectively. For arbitrary weights, the only difference is that the size of a set is measured in terms of $\frac{1}{w}$. 

Nevertheless, the presence of bound constraints makes the problem more complicated. The difficulty is partially due to the fact that even if we hit one of the bounds on some element, the cost of that element might change later on. This is in sharp contrast to the $\ei$-norm in~\cite{berczi2023infty}, where the cost of an element becomes fixed once it reaches one of the bounds. Fortunately, with the help of Step 3, we can overcome these difficulties.

%%%%%%%%%%%%%%%%
\subsection{Guessing \texorpdfstring{$\boldsymbol{\Delta}$}{} and \texorpdfstring{$\boldsymbol{\delta+\Delta}$}{}}
\label{sec:guess}
%%%%%%%%%%%%%%%%

Let us order the elements of the ground set $S = \{ s_1, \dots, s_n \}$ in such a way that $F^*= \{ s_1, \dots, s_{|F^*|} \}$. By the second half of Corollary~\ref{cor:inv_min_cost_span_opt_dev_vector}, we may assume that
\begin{equation*}
w(s_1) \cdot \ell(s_1) = \ldots = w \left( s_{|F^*|} \right) \cdot \ell \left( s_{|F^*|} \right) \ge w \left( s_{|F^*| + 1} \right) \cdot \ell \left( s_{|F^*| + 1} \right) \ge \ldots \ge w(s_n) \cdot \ell(s_n) \end{equation*}
and
\begin{equation*}
w(s_1) \cdot u(s_1) \ge \ldots \ge w \left( s_{|F^*|} \right) \cdot u \left( s_{|F^*|} \right) \ge w \left( s_{|F^*| + 1} \right) \cdot u \left( s_{|F^*| + 1} \right) = \ldots = w(s_n) \cdot u(s_n).
\end{equation*}
Indeed, the corollary says that $\delta+\Delta\geq\max\{w(s)\cdot\ell (s)\mid s\in S\}$, hence we may assume that $w(s)\cdot\ell(s)$ is exactly this maximum for the elements of $F^*$. Similarly, $\Delta\leq\min\{w(s)\cdot u(s)\mid s\in S\}$, hence we may assume that $w(s)\cdot u(s)$ is exactly this minimum for the elements of $S\setminus F^*$. Now order the elements of $S$ as follows: we start with the elements of $F^*$ in a decreasing order according to $w(s)\cdot u(s)$, then followed by the elements of $S\setminus F^*$ in a decreasing order according to $w(s)\cdot \ell(s)$. This gives an ordering as requested.

By the first half of Corollary~\ref{cor:inv_min_cost_span_opt_dev_vector}, we may assume that the value of $\Delta$ falls in one of the intervals 
\begin{align*}
\cI_\Delta
\coloneqq 
{}&{}\Big\{\,\left[w(s_{|F^*| + 1})\cdot\ell(s_{|F^*| + 1}),\,
w(s_{|F^*| + 1})\cdot u(s_{|F^*| + 1})\right]\,\Big\}\\
\cup\,
{}&{}\Big\{\,\left[w(s_{i+1})\cdot\ell(s_{i+1}),\,
w(s_i)\cdot \ell(s_i)\right]\,\bigm| i=|F^*|+1,\dots,n-1\,\Big\},
\end{align*}
and similarly, the value of $\delta+\Delta$ falls in one of the intervals
\begin{align*}
\cI_{\delta+\Delta}
\coloneqq 
{}&{}\Big\{\,\left[w(s_{|F^*|})\cdot\ell(s_{|F^*|1}),\,
w(s_{|F^*|})\cdot u(s_{|F^*|})\right]\,\Big\}\\
\cup\,
{}&{}\Big\{\,\left[w(s_{i})\cdot\ell(s_{i}),\,
w(s_{i-1})\cdot \ell(s_{i-1})\right]\,\bigm| i=2,\dots,|F^*|\,\Big\}.
\end{align*}
Therefore, we pair up the intervals of $\cI_\Delta$ and $\cI_{\delta+\Delta}$ in every possible way, and consider the problem of finding an optimal deviation vector $\p{\delta, \Delta}{\ell, u}{w}$ subject to the bounds on $\delta$ and $\delta+\Delta$. 

%%%%%%%%%%%%%%%%
\subsection{Modifying the bound-constraints}
\label{sec:bounds}
%%%%%%%%%%%%%%%%

By the definition of $\p{\delta, \Delta}{\ell, u}{w}$, if $\Delta \in \left[\, w(s_{i+1}) \cdot \ell(s_{i+1}),\, w(s_{i}) \cdot \ell(s_{i})\, \right]$ for some $i \in \{{|F^*| + 1}, \allowbreak \ldots, n \}$, then $\p{\delta, \Delta}{\ell, u}{w}(s_j) = \ell(s_j)$ holds for each $j \in \left\{ |F^*| + 1, |F^*| + 2, \ldots, i \right\}$. Similarly, if $\delta + \Delta \in \left[\, w(s_{i}) \cdot u(s_{i}),\, w(s_{i-1}) \cdot u(s_{i-1})\, \right]$ for some $i \in \{2,\dots,|F^*|\}$, then $\p{\delta, \Delta}{\ell, u}{w}(s_j) = u(s_j)$ holds for each $j \in \{ i+1, \ldots, |F^*| \}$. These observations together with Corollary~\ref{cor:inv_min_cost_span_opt_dev_vector} imply that it suffices to consider instances of the minimum-cost inverse optimization problem where the lower and upper bounds are of the form
\medskip

{\centering
\noindent\fbox{\begin{minipage}{0.97\textwidth}
\begin{center}
    \textbf{\linkdest{spec-lu}{SPEC-LU}}
\end{center}
\vspace{-0.4cm}
\begin{gather*}
\ell(s)\coloneqq
\begin{cases}
0 & \text{if $s \in S_0$,} \\
\lin/w(s) & \text{if $s \in F^*\setminus S_0$,} \\
\lout/w(s) & \text{otherwise,}
\end{cases} \hspace{1cm} 
u(s)\coloneqq 
\begin{cases}
0 & \text{if $s \in S_0$,} \\
\uin/w(s) & \text{if $s \in F^*\setminus  S_0$,} \\
\uout/w(s) & \text{otherwise}
\end{cases}
\end{gather*}
for some $S_0 \subseteq S$, $\lin, \lout \in \mathbb{R} \cup \{ - \infty \}$ and $\uin, \uout \in \mathbb{R} \cup \{ + \infty \}$ satisfying $\lin \le \uin$, $\lout \le \uout$, $\lin \ge 0$ if $S_0 \cap F^* \ne \emptyset$, $\uout \le 0$ if $S_0\setminus  F^* \ne \emptyset$,  
\begin{align*}
\max \left\{ w(s) \bigm| s \in S\setminus  F^* \right\} \cdot \lout &\le \min \left\{ w(s) \bigm| s \in F^* \right\} \cdot \lin,\ \text{and}\\
\max \left\{ w(s) \bigm| s \in S\setminus  F^* \right\} \cdot \uout &\le \min \left\{ w(s) \bigm| s \in F^* \right\} \cdot \uin.
\end{align*}
\end{minipage}}}
\medskip

\noindent We refer to the set of these properties as \hyperlink{spec-lu}{\normalfont(SPEC-LU)}. For ease of discussion, we define
\[\mu(X)\coloneqq \tfrac{1}{w}(X\setminus S_0)=\sum_{s \in X \setminus S_0} \frac{1}{w(s)} .\]
The function $\mu$ plays a key role in the rest of the proof.

%%%%%%%%%%%%%%%%%%%%%%%%%%%%%%%%%%%%%%%%%%%%%%%%%%%%%%%%%%%%%%%%%%%%%%%%%%%%%%%%%%%%%%%%%%%%%%%%

%%%%%%%%%%%%%%%%
\subsection{Characterizing feasibility}
\label{sec:feasibility}
%%%%%%%%%%%%%%%%

The feasibility of the modified problem is characterized by the following lemma.

\begin{lemma} \label{lem:inv_min_cost_span_n&s_cond_for_infeas}
Let $\big( S, \mathcal{F}, F^*, c, \ell, u, \spa_w (\cdot) \big)$ be a minimum-cost inverse optimization problem, where $\ell$ and $u$ satisfy \hyperlink{spec-lu}{\normalfont(SPEC-LU)}. Let 
\begin{align*}
m_1&\coloneqq \min \left\{ \frac{c(F) - c(F^*) + \uin \cdot \mu\left( F^* \setminus F \right)}{\mu\left( F \setminus  F^* \right)} \Biggm| F \in \mathcal{F}, \, F\setminus  S_0 \nsubseteq F^*\setminus  S_0 \right\},\\[2pt]
m_2&\coloneqq \max \left\{ \frac{c(F^*) - c(F) + \lout \cdot \mu\left( F\setminus F^*\right)}{\mu\left( F^*\setminus  F \right)} \Biggm|
F \in \mathcal{F}, F^*\setminus S_0 \nsubseteq F\setminus  S_0 \right\},\\[2pt]
m_3&\coloneqq \min \left\{ \frac{c(F) - c(F^*)}{\mu\left( F \setminus  F^* \right)} \Biggm| F \in \mathcal{F}, F\setminus  S_0 \nsubseteq F^*\setminus  S_0 \right\}, \\[2pt]
m_4&\coloneqq \max \left\{ \frac{c(F^*) - c(F)}{\mu\left( F^*\setminus  F \right)} \Biggm| F \in \mathcal{F}, F^*\setminus  S_0 \nsubseteq F\setminus  S_0 \right\}.
\end{align*}
\begin{enumerate}[label=(\alph*)]\itemsep0em
\item If $\uin \ne + \infty$ and $\lout \ne - \infty$, then the minimum-cost inverse optimization problem is feasible if and only if $\p{\uin - \lout, \lout}{\ell, u}{w}$ is a feasible deviation vector.

\item If $\uin \ne + \infty$ and $\lout = - \infty$, then the minimum-cost inverse optimization problem is feasible if and only if $\p{\uin - m'_1, m'_1}{\ell, u}{w}$ is a feasible deviation vector, where
\begin{equation*}
m'_1 \coloneqq  
\begin{cases}
m_1 & \text{if $m_1 \ne + \infty$,} \\
0 & \text{otherwise.}
\end{cases}
\end{equation*}

\item If $\uin = + \infty$ and $\lout \ne - \infty$, then the minimum-cost inverse optimization problem is feasible if and only if $\p{m'_2 - \lout, \lout}{\ell, u}{w}$ is a feasible deviation vector, where
\begin{equation*}
m'_2 \coloneqq  
\begin{cases}
m_2 & \text{if $m_2 \ne - \infty$,} \\
0 & \text{otherwise.}
\end{cases}
\end{equation*}

\item If $\uin = + \infty$ and $\lout = - \infty$, then the minimum-cost inverse optimization problem is feasible if and only if $\p{m'_4-m'_3, m'_3}{\ell, u}{w}$ is a feasible deviation vector, where
\begin{equation*}
m'_3  \coloneqq  \begin{cases}
m_3 & \text{if $m_3 \ne + \infty$,} \\
0 & \text{otherwise,}
\end{cases}\quad \text{and}\quad 
m'_4  \coloneqq  \begin{cases}
m_4 & \text{if $m_4 \ne - \infty$,} \\
0 & \text{otherwise.}
\end{cases}
\end{equation*}
\end{enumerate}
\end{lemma}
\begin{proof}
In all cases, the `if' direction is straightforward, hence we prove the `only if' direction. 

Assume first that $\uin \ne + \infty$ and $\lout \ne - \infty$, and let $p$ be a feasible deviation vector. For any $F \in \mathcal{F}$, we have
\begin{align*}
&( c - \p{\uin - \lout, \lout}{\ell, u}{w} )(F^*) - ( c - \p{\uin - \lout, \lout}{\ell, u}{w} )(F)\\[2pt]
{}&{}~~= 
\left( c(F^*) - \sum_{s \in F^*\setminus S_0} \frac{\uin}{w(s)} \right) - \left( c(F) - \sum_{s \in (F\setminus  S_0) \cap (F^*\setminus  S_0)} \frac{\uin}{w(s)} - \sum_{s \in (F\setminus  S_0)\setminus (F^*\setminus S_0)} \frac{\lout}{w(s)} \right) \\[2pt]
{}&{}~~= 
c(F^*) - c(F) - \sum_{s \in (F^*\setminus  S_0)\setminus  (F\setminus  S_0)} \frac{\uin}{w(s)} + \sum_{s \in (F\setminus  S_0)\setminus  (F^*\setminus  S_0)} \frac{\lout}{w(s)} \\[2pt]
{}&{}~~= 
c(F^*) - c(F) - \sum_{s \in (F^*\setminus  S_0)\setminus  (F\setminus  S_0)} u(s) + \sum_{s \in (F\setminus  S_0)\setminus  (F^*\setminus  S_0)} \ell(s) \\[2pt]
{}&{}~~\le 
c(F^*) - c(F) - \sum_{s \in (F^*\setminus  S_0)\setminus  (F\setminus  S_0)} p(s) + \sum_{s \in (F\setminus  S_0)\setminus  (F^*\setminus  S_0)} p(s) \\[2pt]
{}&{}~~= 
\big( c(F^*) - p(F^*) \big) - \big( c(F) - p(F) \big)\\[2pt]
{}&{}~~\le 
0,
\end{align*}
implying that $\p{\uin - \lout, \lout}{\ell, u}{w}$ is also feasible.

Now assume that $\uin \ne + \infty$ and $\lout = - \infty$, let again $p$ be a feasible deviation vector and $F\in\cF$ be arbitrary. If $\mu(F \setminus  F^*) \ne 0$, i.e.\ $F\setminus  S_0 \nsubseteq F^*\setminus  S_0$, then we get
\begin{align*}
&( c - \p{\uin - m'_1, m'_1}{\ell, u}{w} )(F^*) - ( c - \p{\uin - m'_1, m'_1}{\ell, u}{w} )(F) \\[2pt]
{}&{}~~= 
c(F^*) - c(F) - \sum_{s \in (F^*\setminus  S_0)\setminus  (F\setminus  S_0)} \frac{\uin}{w(s)} + \sum_{s \in (F\setminus  S_0)\setminus  (F^*\setminus  S_0)} \frac{m'_1}{w(s)} \\[2pt]
{}&{}~~= 
c(F^*) - c(F) - \uin \cdot \mu(F^*\setminus  F) + m'_1 \cdot \mu(F \setminus  F^*)\\[2pt]
{}&{}~~\leq
0,
\end{align*}
where the last inequality follows from the definition of $m'_1$. On the other hand, if $\mu(F \setminus  F^*) = 0$, i.e.\ $(F\setminus  S_0) \setminus (F^*\setminus  S_0) = \emptyset$, then we get 
\begin{align*}
&( c - \p{\uin - m'_1, m'_1}{\ell, u}{w} )(F^*) - ( c - \p{\uin - m'_1, m'_1}{\ell, u}{w} )(F) \\[2pt]
{}&{}~~= 
c(F^*) - c(F) - \sum_{s \in (F^*\setminus  S_0)\setminus  (F\setminus  S_0)} \frac{\uin}{w(s)} + \sum_{s \in (F\setminus  S_0)\setminus  (F^*\setminus  S_0)} \frac{m'_1}{w(s)} \\[2pt]
{}&{}~~= 
c(F^*) - c(F) - \sum_{s \in (F^*\setminus  S_0)\setminus  (F\setminus  S_0)} u(s) + 0 \\[2pt]
{}&{}~~\le 
c(F^*) - c(F) - \sum_{s \in (F^*\setminus  S_0)\setminus  (F\setminus  S_0)} p(s) + \sum_{s \in (F\setminus  S_0) \setminus (F^*\setminus  S_0)} p(s) \\[2pt]
%{}&{}~~= 
%\left( c(F^*) - p(F^*) \right) - \left( c(F) - p(F) \right)\\[2pt]
{}&{}~~= 
(c-p)(F^*) - (c-p)(F)\\[2pt]
{}&{}~~\le 
0,
\end{align*}
implying that $\p{\uin - m'_1, m'_1}{\ell, u}{w}$ is also feasible.

The remaining two cases can be verified analogously.
\end{proof}

%%%%%%%%%%%%%%%%%%%%%%%%%%%%%%%%
\subsection{Solving the subproblems} 
\label{sec:subproblem}
%%%%%%%%%%%%%%%%%%%%%%%%%%%%%%%%

Finally, we derive an algorithm for solving a subproblem obtained after modifying the bound-constraints. The algorithm is presented as Algorithm~\ref{algo:inv_min_cost_span}. By convention, undefined objects are denoted by $\ast$. We slightly modify the notion of small and large bad sets from the beginning of Section~\ref{sec:algorithm} by calling a set $F\in\cF$ small if $\mu(F)<\mu(F^*)$ and large if $\mu(F)>\mu(F^*)$. The high level description of the algorithm is as follows. At each iteration $i$, we distinguish five main cases depending on the bad sets eliminated in that iteration. In Case 1, we eliminate a bad set $F_i$ having the same size as $F^*$. In Case 2, we eliminate a small bad set $F_i$ alone. In Case 3, we eliminate a small bad set $F_i$ together with a large bad set $Z_i$ that was found before. In Case 4, we eliminate a large bad set $F_i$ alone. Finally, in Case 5, we eliminate a large bad set $F_i$ together with a small bad set $X_i$ that was found before. 

In all cases, first we determine the values of $\Delta$ (corresponding to the change on the elements of $S\setminus F^*$) and $\delta+\Delta$ (corresponding the change on the elements of $F^*$) needed to eliminate the bad sets in question as if no bound-constraints were given. In Cases~\hyperlink{1.1}{1.1}, \hyperlink{2.1.1}{2.1.1}, \hyperlink{3.1.1}{3.1.1}, \hyperlink{4.1.1}{4.1.1} and~\hyperlink{5.1.1}{5.1.1}, the resulting deviation vector do not violate the bound-constrains, hence we apply the changes directly. In the remaining cases however, we hit a bound-constraint either on all the elements of $F^*$, or on those of $S\setminus F^*$, or both (recall that the bound-constraints are assumed to satisfy \hyperlink{spec-lu}{(SPEC-LU)} at this point). In such situations, we set the changes on the problematic elements to the extreme, that is, to the lower or upper bound that was violated, recompute the necessary changes on the remaining elements to eliminate the bad sets in question, and set the deviation vector accordingly if the bound-constraints are met, otherwise conclude that the problem is infeasible. The algorithm considers all the possible scenarios, and shows how to handle these cases using the values computed by the functions $f_1, \ldots, f_{12}$ defined below.

The algorithm and its analysis is rather technical, and requires the discussion of several cases. Nevertheless, it is worth emphasizing that in each step we apply the natural modification to the deviation vector that is needed to eliminate the current bad set or sets. At this point, the reader may rightly ask whether it is indeed necessary to consider all these cases. The answer is unfortunately yes, as the proposed Newton-type scheme may run into any of them, see Figure~\ref{fig:cases}. 

For ease of discussion, we introduce several notation before stating the algorithm. Let
 \begin{align*}
  \mathcal{F}' & \coloneqq  \left\{ F' \in \mathcal{F} \, \middle| \, \mu(F') < \mu(F^*) \right\} , \\
  \widehat{\mathcal{F}}' & \coloneqq  \left\{ F' \in \mathcal{F} \, \middle| \, \mu(F') < \mu(F^*), \, F'\setminus S_0 \nsubseteq F^*\setminus S_0 \right\} , \\[2pt]
  \mathcal{F}'' & \coloneqq  \left\{ F'' \in \mathcal{F} \, \middle| \, \mu(F'') = \mu(F^*), \, F''\setminus S_0 \ne F^*\setminus S_0 \right\} , \\[2pt]
  \mathcal{F}''' & \coloneqq  \left\{ F''' \in \mathcal{F} \, \middle| \, \mu(F''') > \mu(F^*) \right\} , \\[2pt]
  \widehat{\mathcal{F}}''' & \coloneqq  \left\{ F''' \in \mathcal{F} \, \middle| \, \mu(F''') > \mu(F^*), \, F^*\setminus S_0 \nsubseteq F'''\setminus S_0 \right\}.
 \end{align*} 
In each iteration, the algorithm computes an optimal solution of the underlying optimization problem, and if the cost of the input solution $F^*$ is strictly larger than the current optimum, then it updates the costs using a value determined by one of the following functions: 
\begin{enumerate}[label=(f\arabic*)]\itemsep5pt
\item~~$f_1(c,F'')\coloneqq \displaystyle \frac{c(F^*) - c(F'')}{\mu\left( F^*\setminus F'' \right)}$,
\item ~~$f_2(c,d,D,F'')\coloneqq \displaystyle \uin - d - D - \frac{c(F^*) - c(F'')}{\mu\left( F^*\setminus F'' \right)}$,
\item ~~$f_3(c,F)\coloneqq \displaystyle \frac{c(F^*) - c(F)}{\mu(F^*) - \mu(F)}$,
\item ~~$f_4(c,d,D,F')\coloneqq \displaystyle \frac{c(F') - c(F^*) + (\uin - d - D) \cdot \mu\left( F^*\setminus F'\right)}{\mu\left( F'\setminus F^*\right)}$,
\item ~~$f_5(c,d,D,F')\coloneqq \displaystyle \frac{c(F^*) - c(F') - (\uin - d - D) \cdot \big( \mu(F^*) - \mu(F') \big)}{\mu\left( F'\setminus F^*\right)}$,
\item ~~$f_6(c,D,F')\coloneqq \displaystyle \frac{c(F^*) - c(F') - (\uout - D) \cdot \big( \mu(F^*) - \mu(F') \big)}{\mu\left( F^*\setminus F'\right)}$,
\item ~~$f_7(c,F',F''')\coloneqq\frac{\hphantom{x} \displaystyle \frac{c(F^*) - c(F')}{\vphantom{\Big|} \mu(F^*) - \mu(F')} - \frac{c(F^*) - c(F''')}{\vphantom{\Big|} \mu(F^*) - \mu(F''')}\hphantom{x}}{\hphantom{x} \displaystyle \frac{\vphantom{\Big|} \mu\left( F^*\setminus F'\right)}{\vphantom{\Big|} \mu(F^*) - \mu(F')} - \frac{\vphantom{\Big|} \mu\left( F^*\setminus F''' \right)}{\vphantom{\Big|} \mu(F^*)- \mu(F''')} \hphantom{x}}$,
\item ~~$f_8(c,F',F''')\coloneqq\displaystyle \frac{c(F^*) - c(F') - f_7(c,F',F''') \cdot \mu\left( F^*\setminus F'\right)}{\mu(F^*) - \mu(F')}$,
\item ~~$f_9(c,D,F''')\coloneqq\displaystyle \frac{c(F^*) - c(F''') - (\lout - D) \cdot \big( \mu(F^*) - \mu(F''') \big)}{\mu\left( F^*\setminus F'''\right)}$,
\item ~~$f_{10}(c,d,D,F''')\coloneqq\displaystyle \frac{c(F''') - c(F^*) + (\lin - d - D) \cdot \mu\left( F^*\setminus F'''\right)}{\mu\left( F'''\setminus F^* \right)}$,
\item ~~$f_{11}(c,d,D,F''')\coloneqq\displaystyle \frac{c(F^*) - c(F''') - (\lin - d - D) \cdot \big( \mu(F^*) - \mu(F''') \big)}{\mu\left( F'''\setminus F^* \right)}$,
\item ~~$f_{12}(c,F',F''')\coloneqq \displaystyle \frac{c(F^*) - c(F''') - f_7(c,F',F''') \cdot \mu\left( F^*\setminus F''' \right)}{\mu(F^*) - \mu(F''')}$.
\end{enumerate}

In the definitions above, we have $F\in\cF'\cup\cF'''$, $F'\in\cF'$, $F''\in\cF''$, $F'''\in\cF'''$, $c\in\mathbb{R}^S$ and $d,D\in\mathbb{R}$ almost always, except -- to avoid division by zero -- for $f_4$ and $f_5$ where $F'\in \widehat{\mathcal{F}}'$, and for $f_9$ where $F'''\in\widehat{\cF}'''$.

\begin{algorithm}[H] \label{algo:inv_min_cost_span}
\caption{Algorithm for the minimum-cost inverse optimization problem under the weighted span objective with bound-constraints satisfying \protect\hyperlink{spec-lu}{(SPEC-LU)}.}
\DontPrintSemicolon

\KwIn{A minimum-cost inverse optimization problem $\big( S, \mathcal{F}, F^*, c, \ell, u, \spa_w( \cdot ) \big)$ with bound constraints satisfying \hyperlink{spec-lu}{(SPEC-LU)}, and an oracle $\cO$ for the minimum-cost optimization problem $(S,\cF,c')$ with any cost function $c'$.}
\KwOut{An optimal deviation vector if the problem is feasible, otherwise \texttt{Infeasible}.}

\bigskip

$\delta_0 \gets \max \{ \lin - \uout, 0 \}$,\quad $\Delta_0 \gets \begin{cases}
                 \uout & \text{if $\uout \ne + \infty$}, \\
                 \lin & \text{if $\uout = + \infty$ and $\lin \ne - \infty$}, \\
                 0 & \text{otherwise}.
                \end{cases}$\;
$d_0 \gets \delta_0$, \quad $D_0 \gets \Delta_0$\;
$X_0 \gets \ast$, \quad $Y_0 \gets \ast$, \quad $Z_0 \gets \ast$\;
$c_0 \gets c - \p{d_0, D_0}{\ell, u}{w}$\;
$F_0 \gets \text{minimum $c_0$-cost member of $\mathcal{F}$ determined by $\mathcal{O}$}$\;
$i \gets 0$\;
\SetAlgoLined
\Whileblock{$c_i(F^*) > c_i(F_i)$}{
    \SetAlgoVlined
    \uIf{$\mu(F_i) = \mu(F^*)$}{
        $X_{i+1} \gets \ast$, \quad $Y_{i+1} \gets F_i$, \quad $Z_{i+1} \gets \ast$\;
        \uIf{$d_i + D_i + f_1(c_i, F_i) \le \uin$}{
            $\delta_{i+1} \gets f_1(c_i, F_i)$, \qquad $\Delta_{i+1} \gets 0$ \tcp*[r]{\linkdest{1.1}Case~1.1~~~~}
        }
        \Else{
            \uIf{$D_i + f_2(c_i, d_i, D_i, F_i) \ge \lout$}{
                $\delta_{i+1} \gets f_1(c_i, F_i)$, \qquad $\Delta_{i+1} \gets f_2(c_i, d_i, D_i, F_i)$ \tcp*[r]{\linkdest{1.2.1}Case~1.2.1~~}
            }
            \Else{
                \Return{\tt Infeasible} \tcp*[r]{\linkdest{1.2.2}Case~1.2.2~~}
            }
        }
    }
%}
%\nonl 
%\Begin{
%    \SetAlgoVlined
    \uIf{$\mu(F_i) < \mu(F^*)$}{
        \uIf{$Z_i = \ast$}{
            $X_{i+1} \gets F_i$, \quad $Y_{i+1} \gets Y_i$, \quad $Z_{i+1} \gets Z_i$\;
            \uIf{$D_i + f_3(c_i, F_i) \le \uout$}{
                \uIf{$d_i + D_i + f_3(c_i, F_i) \le \uin$}{
                    $\delta_{i+1} \gets 0$, \quad $\Delta_{i+1} \gets f_3(c_i, F_i)$ \tcp*[r]{\linkdest{2.1.1}Case~2.1.1~~}
                }
                \Else{
                    \uIf{$F_i\setminus S_0 \nsubseteq F^*\setminus S_0$ {\bf and} $D_i + f_4(c_i, d_i, D_i, F_i) \ge \lout$}{
                        $\delta_{i+1} \gets f_5(c_i, d_i, D_i, F_i)$, \quad $\Delta_{i+1} \gets f_4(c_i, d_i, D_i, F_i)$ \MyComment{\linkdest{2.1.2.1}Case~2.1.2.1}
                    }
                    \Else{
                        \Return{\tt Infeasible} \tcp*[r]{\linkdest{2.1.2.2}Case~2.1.2.2}
                    }
                }
            }
            \Else{
                \uIf{$d_i + f_6(c_i, D_i, F_i) \le \uin - \uout$}{
                    $\delta_{i+1} \gets f_6(c_i, D_i, F_i)$, \quad $\Delta_{i+1} \gets \uout - D_i$ \tcp*[r]{\linkdest{2.2.1}Case~2.2.1~~}
                }
                \Else{
                    \uIf{$F_i\setminus S_0 \nsubseteq F^*\setminus S_0$ {\bf and} $D_i + f_4(c_i, d_i, D_i, F_i) \ge \lout$}{
                        $\delta_{i+1} \gets f_5(c_i, d_i, D_i, F_i)$, \quad $\Delta_{i+1} \gets f_4(c_i, d_i, D_i, F_i)$ \MyComment{\linkdest{2.2.2.1}Case~2.2.2.1}
                    }
                    \Else{
                        \Return{\tt Infeasible} \tcp*[r]{\linkdest{2.2.2.2}Case~2.2.2.2}
                    }
                }
            }
        }
    }
}
\end{algorithm}

%\newpage

\begin{algorithm}
\DontPrintSemicolon

\setcounter{AlgoLine}{35}

\SetAlgoLined
\nonl \Begin{ \vspace{-14.5pt}
    \nonl \Begin{
        \SetAlgoVlined
        \Else{
            \uIf{$D_i + f_8(c_i, F_i, Z_i) \le \uout$}{
                \uIf{$d_i + D_i + f_7(c_i, F_i, Z_i) + f_8(c_i, F_i, Z_i) \le \uin$}{
                    $X_{i+1} \gets F_i$, \quad $Y_{i+1} \gets Y_i$, \quad $Z_{i+1} \gets Z_i$\;
                    $\delta_{i+1} \gets f_7(c_i, F_i, Z_i)$, \quad $\Delta_{i+1} \gets f_8(c_i, F_i, Z_i)$ \tcp*[r]{\linkdest{3.1.1}Case~3.1.1~~}
                }
                \Else{
                    $X_{i+1} \gets F_i$, \quad $Y_{i+1} \gets Y_i$, \quad $Z_{i+1} \gets \ast$\;
                    \uIf{$F_i\setminus S_0 \nsubseteq F^*\setminus S_0$ {\bf and} $D_i + f_4(c_i, d_i, D_i, F_i) \ge \lout$}{
                        $\delta_{i+1} \gets f_5(c_i, d_i, D_i, F_i)$, \quad $\Delta_{i+1} \gets f_4(c_i, d_i, D_i, F_i)$ \MyComment{\linkdest{3.1.2.1}Case~3.1.2.1}
                    }
                    \Else{
                        \Return{\tt Infeasible} \tcp*[r]{\linkdest{3.1.2.2}Case~3.1.2.2}
                    }
                }
            }
            \Else{
                $X_{i+1} \gets F_i$, \quad $Y_{i+1} \gets Y_i$, \quad $Z_{i+1} \gets \ast$\;
                \uIf{$d_i + f_6(c_i, D_i, F_i) \le \uin - \uout$}{
                    $\delta_{i+1} \gets f_6(c_i, D_i, F_i)$, \quad $\Delta_{i+1} \gets \uout - D_i$ \tcp*[r]{\linkdest{3.2.1}Case~3.2.1~~}
                }
                \Else{
                    \uIf{$F_i\setminus S_0 \nsubseteq F^*\setminus S_0$ {\bf and} $D_i + f_4(c_i, d_i, D_i, F_i) \ge \lout$}{
                        $\delta_{i+1} \gets f_5(c_i, d_i, D_i, F_i)$, \quad $\Delta_{i+1} \gets f_4(c_i, d_i, D_i, F_i)$ \MyComment{\linkdest{3.2.2.1}Case~3.2.2.1}
                    }
                    \Else{
                        \Return{\tt Infeasible} \tcp*[r]{\linkdest{3.2.2.2}Case~3.2.2.2}
                    }
                }
            }
        }
    }
%}
%\end{algorithm}

%\newpage

%\begin{algorithm}
%\DontPrintSemicolon

%\setcounter{AlgoLine}{56}
%\SetAlgoLined
%\nonl \Begin{
    \SetAlgoVlined
    \uIf{$\mu(F_i) > \mu(F^*)$}{
        \uIf{$X_i = \ast$}{
            $X_{i+1} \gets X_i$, \quad $Y_{i+1} \gets Y_i$, \quad $Z_{i+1} \gets F_i$\;
            \uIf{$d_i + D_i + f_3(c_i, F_i) \ge \lin$}{
                \uIf{$D_i + f_3(c_i, F_i) \ge \lout$}{
                    $\delta_{i+1} \gets 0$, \quad $\Delta_{i+1} \gets f_3(c_i, F_i)$ \tcp*[r]{\linkdest{4.1.1}Case~4.1.1~~}
                }
                \Else{
                    \uIf{$F^*\setminus S_0 \nsubseteq F_i\setminus S_0$ {\bf and} $d_i + f_9(c_i, D_i, F_i) \le \uin - \lout$}{
                        $\delta_{i+1} \gets f_9(c_i, D_i, F_i)$, \quad $\Delta_{i+1} \gets \lout - D_i$ \tcp*[r]{\linkdest{4.1.2.1}Case~4.1.2.1}
                    }
                    \Else{
                        \Return{\tt Infeasible} \tcp*[r]{\linkdest{4.1.2.2}Case~4.1.2.2}
                    }
                }
            }
            \Else{
                \uIf{$D_i + f_{10}(c_i, d_i, D_i, F_i) \ge \lout$}{
                    $\delta_{i+1} \gets f_{11}(c_i, d_i, D_i, F_i)$, \quad $\Delta_{i+1} \gets f_{10}(c_i, d_i, D_i, F_i)$ \MyComment{\linkdest{4.2.1}Case~4.2.1~~}
                }
                \Else{
                    \uIf{$F^*\setminus S_0 \nsubseteq F_i\setminus S_0$ {\bf and} $d_i + f_9(c_i, D_i, F_i) \le \uin - \lout$}{
                        $\delta_{i+1} \gets f_9(c_i, D_i, F_i)$, \quad $\Delta_{i+1} \gets \lout - D_i$ \tcp*[r]{\linkdest{4.2.2.1}Case~4.2.2.1}
                    }
                    \Else{
                        \Return{\tt Infeasible} \tcp*[r]{\linkdest{4.2.2.2}Case~4.2.2.2}
                    }
                }
            }
        }
    }
}
\end{algorithm}

\newpage

\begin{algorithm}
\DontPrintSemicolon

\setcounter{AlgoLine}{74}
\nonl \Begin{ \vspace{-14.5pt}
    \nonl \Begin{
        \Else{
            \uIf{$d_i + D_i + f_7(c_i, X_i, F_i) + f_{12}(c_i, X_i, F_i) \ge \lin$}{
                \uIf{$D_i + f_{12}(c_i, X_i, F_i) \ge \lout$}{
                    $X_{i+1} \gets X_i$, \quad $Y_{i+1} \gets Y_i$, \quad $Z_{i+1} \gets F_i$\;
                    $\delta_{i+1} \gets f_7(c_i, X_i, F_i)$, \quad $\Delta_{i+1} \gets f_{12}(c_i, X_i, F_i)$ \tcp*[r]{\linkdest{5.1.1}Case~5.1.1~~}
                }
                \Else{
                    $X_{i+1} \gets \ast$, \quad $Y_{i+1} \gets Y_i$, \quad $Z_{i+1} \gets F_i$\;
                    \uIf{$F^*\setminus S_0 \nsubseteq F_i\setminus S_0$ {\bf and} $d_i + f_9(c_i, D_i, F_i) \le \uin - \lout$}{
                        $\delta_{i+1} \gets f_9(c_i, D_i, F_i)$, \quad $\Delta_{i+1} \gets \lout - D_i$ \tcp*[r]{\linkdest{5.1.2.1}Case~5.1.2.1}
                    }
                    \Else{
                        \Return{\tt Infeasible} \tcp*[r]{\linkdest{5.1.2.2}Case~5.1.2.2}
                    }
                }
            }
            \Else{
                $X_{i+1} \gets \ast$, \quad $Y_{i+1} \gets Y_i$, \quad $Z_{i+1} \gets F_i$\;
                \uIf{$D_i + f_{10}(c_i, d_i, D_i, F_i) \ge \lout$}{
                    $\delta_{i+1} \gets f_{11}(c_i, d_i, D_i, F_i)$, \quad $\Delta_{i+1} \gets f_{10}(c_i, d_i, D_i, F_i)$ \MyComment{\linkdest{5.2.1}Case~5.2.1~~}
                }
                \Else{
                    \uIf{$F^*\setminus S_0 \nsubseteq F_i\setminus S_0$ {\bf and} $d_i + f_9(c_i, D_i, F_i) \le \uin - \lout$}{
                        $\delta_{i+1} \gets f_9(c_i, D_i, F_i)$, \quad $\Delta_{i+1} \gets \lout - D_i$ \tcp*[r]{\linkdest{5.2.2.1}Case~5.2.2.1}
                    }
                    \Else{
                        \Return{\tt Infeasible} \tcp*[r]{\linkdest{5.2.2.2}Case~5.2.2.2}
                    }
                }
            }
        }
    }
    $d_{i+1} \gets d_i + \delta_{i+1}$, \quad $D_{i+1} \gets D_i + \Delta_{i+1}$\;
    $c_{i+1} \gets c - \p{d_{i+1}, D_{i+1}}{\ell, u}{w}$\;
    $F_{i+1} \gets \text{minimum $c_{i+1}$-cost member of $\mathcal{F}$ determined by $\mathcal{O}$}$\;
    $i \gets i+1$\;
}
\Return{$\p{d_i, D_i}{\ell, u}{w}$}\;
\end{algorithm}

Figure~\ref{fig:cases} in Appendix~\hyperref[sec:appendixb]{B} provides toy examples for the different cases occurring in Algorithm~\ref{algo:inv_min_cost_span}, while Table~\ref{table:summary} in Appendix~\hyperref[sec:appendixc]{C} gives a summary of the cases which might be helpful when reading the proof. For proving the correctness and the running time of the algorithm, we need the following lemmas. The proofs of these statements follow from the definitions in a fairly straightforward way, hence those are deferred to Appendix~\hyperref[sec:appendixa]{A}. 

Our first lemma shows that if $F^*$ is not optimal with respect to the current cost function, then in the next step it either has the same cost as the current optimal solution with respect to the modified cost function, or the problem is infeasible.

\begin{lemma} \label{lemma:inv_min_cost_span_correction}
If $F^*$ is not a minimum $c_i$-cost member of $\mathcal{F}$, then either
$c_{i+1}(F^*) = c_{i+1}(F_i)$ or Algorithm~\ref{algo:inv_min_cost_span} declares the problem to be infeasible.
\end{lemma}

The following lemmas together imply an upper bound on the total number of iterations.

\begin{lemma} \label{lemma:step1}
Let $i_1<i_2$ be indices such that both steps $i_1$ and $i_2$ corresponds to Cases~\hyperlink{1.1}{1.1} or~\hyperlink{1.2.1}{1.2.1}. Then $\mu\left( Y_{i_1+1} \cap F^* \right) < \mu\left( Y_{i_2+1} \cap F^* \right)$.
\end{lemma}

\begin{lemma} \label{lemma:step2}
Let $(a_1,a_2,a_3,a_4)$ be a 4-tuple of integers satisfying $0\leq a_k\leq \|w\|_{\text{-}1}$ for any $k\in[4]$.
\begin{enumerate}[label=(\alph*)]\itemsep0em
\item \label{lobab:a} There is at most one index $i$ such that $F^*$ is not a minimum $c_i$-cost member of $\cF$, $\mu(F_i) < \mu(F^*)$, $Z_i \ne \ast$ and $\big (\mu(F_i),\mu(Z_i),\mu(F_i\cap F^*),\mu(Z_i \cap F^*)\big)=(a_1,a_2,a_3,a_4)$.
\item \label{lobab:b} There is at most one index $i$ such that $F^*$ is not a minimum $c_i$-cost member of $\cF$, $\mu(F_i) > \mu(F^*)$, $X_i \ne \ast$ and $\big( \mu(X_i),\mu(F_i),\mu(X_i\cap F^*),\mu(F_i \cap F^*)\big)=(a_1,a_2,a_3,a_4)$.
\end{enumerate}
\end{lemma}

\begin{lemma} \label{lemma:step3}
\mbox{}
\begin{enumerate}[label=(\alph*)]\itemsep0em
\item \label{notsame_new:a} If $\mu(F_i), \mu(F_{i+1})<\mu(F^*)$ and $Z_{i}, Z_{i+1} = \ast$, then at least one of $\mu(F_i)\neq\mu(F_{i+1})$ and $\mu\left( F_i \cap F^* \right) < \mu\left( F_{i+1} \cap F^* \right)$ holds.
\item \label{notsame_new:b} If $\mu(F_i), \mu(F_{i+1})>\mu(F^*)$ and $X_{i}, X_{i+1} = \ast$, then at least one of $\mu(F_i)\neq\mu(F_{i+1})$ and $\mu\left( F_i \cap F^* \right) < \mu\left( F_{i+1} \cap F^* \right)$ holds.
\end{enumerate}
\end{lemma}

The essence of the next lemma is that for a feasible instance satisfying \hyperlink{spec-lu}{\normalfont (SPEC-LU)}, there exists an optimal deviation vector $\p{\delta, \Delta}{\ell, u}{w}$ where $\delta$ and $\Delta$ can be bounded.  

\begin{lemma} \label{lemma:inv_min_cost_span_opt_dev_vector_v2}
 Let $\big( S, \mathcal{F}, F^*, c, \ell, u, \spa_w( \cdot ) \big)$ be a feasible minimum-cost inverse optimization problem where the bound-constraints satisfy \hyperlink{spec-lu}{\normalfont (SPEC-LU)}. Then there exist $\delta, \Delta \in \mathbb{R}$ with $\lout \le \Delta \le \uout$ and $\lin \le \delta + \Delta \le \uin$ such that $\p{\delta, \Delta}{\ell, u}{w}$ is an optimal deviation vector.
\end{lemma}

Finally, we show that analogous bounds hold for the values $d_i$ and $D_i$ computed throughout the algorithm.

\begin{lemma} \label{lemma:inv_min_cost_span_deltas_stay_in_the_intervals}
Either $\lin \le d_i + D_i \le \uin$ and $\lout \le D_i \le \uout$, or Algorithm~\ref{algo:inv_min_cost_span} declares the problem to be infeasible. Moreover, in Cases~\hyperlink{2.1.2.1}{2.1.2.1}, \hyperlink{2.2.2.1}{2.2.2.1}, \hyperlink{3.1.2.1}{3.1.2.1}, and~\hyperlink{3.2.2.1}{3.2.2.1}, we have $d_{i+1} + D_{i+1} = \uin$, and in Cases~\hyperlink{4.2.1}{4.2.1} and~\hyperlink{5.2.1}{5.2.1}, we have $d_{i+1} + D_{i+1} = \lin$.
\end{lemma}

Now we are ready to prove the correctness and discuss the running time of the algorithm.

\begin{thm} \label{thm:inv_min_cost_span_correctness_of_algo}
Algorithm~\ref{algo:inv_min_cost_span} determines an optimal deviation vector, if exists, for the minimum-cost inverse optimization problem $\big( S, \mathcal{F}, F^*, c, \ell, u, \spa_w( \cdot ) \big)$ with bound-constraints satisfying \hyperlink{spec-lu}{\normalfont (SPEC-LU)} using $O\big( \|w\|_{\text{-}1}^{6} \big)$ calls to $\cO$.
\end{thm}
\begin{proof}
We discuss the time complexity and the correctness of the algorithm separately.
\medskip

\noindent \emph{Time complexity.} Recall that $w\in\mathbb{R}_+^S$ is scaled so that $\frac{1}{w}(X)$ is an integer for each $X\subseteq S$. We show that the algorithm terminates after at most $O(\|w\|_{\text{-}1}^6)$ iterations of the while loop. By Lemma~\ref{lemma:step1}, there are at most $O(\|w\|_{\text{-}1})$ iterations corresponding to Cases~\hyperlink{1.1}{1.1} and~\hyperlink{1.2.1}{1.2.1}. By Lemma~\ref{lemma:step2}, there are at most $O(\|w\|_{\text{-}1}^4)$ iterations corresponding to Cases~\hyperlink{3.1.1}{3.1.1}, \hyperlink{3.1.2.1}{3.1.2.1}, \hyperlink{3.2.1}{3.2.1}, \hyperlink{3.2.2.1}{3.2.2.1}, \hyperlink{5.1.1}{5.1.1}, \hyperlink{5.1.2.1}{5.1.2.1}, \hyperlink{5.2.1}{5.2.1}, and~\hyperlink{5.2.2.1}{5.2.2.1}. Between two such iterations, there are at most $O(\|w\|_{\text{-}1}^2)$ iterations corresponding to the remaining cases by Lemma~\ref{lemma:step3}. Hence the total number of iterations is at most $O(\|w\|_{\text{-}1}^6)$. 
\medskip

\noindent \emph{Infeasibility.} By the above, the algorithm terminates after a finite number of iterations. First, we show that if the algorithm returns \texttt{Infeasible}, then it correctly recognizes the problem to be infeasible. Assume that the algorithm terminates in the $i$th iteration and declares the problem to be infeasible. We distinguish different scenarios depending on which case the last step belongs to.
\medskip

\noindent \textbf{Case~\hyperlink{1.2.2}{1.2.2}.} Note that $\uin \ne + \infty$ and $\lout \ne - \infty$. In addition,
\begin{align*} 
(\uin - \lout) \cdot \mu(F^* \setminus F_i)
{}&{}< 
\Big( \uin - \big( \uin - d_i - D_i - f_1(c_i, F_i) \big) \Big) \cdot \mu(F^* \setminus F_i) \\[2pt]
{}&{}= 
\big( d_i + f_1(c_i, F_i) \big) \cdot \mu(F^* \setminus F_i)\\[2pt]
{}&{}= 
d_i \cdot \mu(F^* \setminus F_i) + \big( c_i(F^*) - c_i(F_i) \big) \\[2pt]
{}&{}= 
c_i(F^*) - c_i(F_i) + d_i \cdot \mu(F^* \setminus F_i) + D_i \cdot 0 \\[2pt]
{}&{}= 
c_i(F^*) - c_i(F_i) + d_i \cdot \mu(F^* \setminus F_i) + D_i \cdot \big( \mu(F^*) - \mu(F_i) \big)\\[2pt]
{}&{}= 
c(F^*) - c(F_i).
\end{align*}
Thus,
\begin{align*}
&(c - \p{\uin - \lout, \lout}{\ell, u}{w})(F^*) - (c - \p{\uin - \lout, \lout}{\ell, u}{w})(F_i) \\[2pt]
{}&{}~~= 
c(F^*) - c(F_i) - (\uin - \lout) \cdot \mu(F^* \setminus F_i) - \lout \cdot \big( \mu(F^*) - \mu(F_i) \big) \\[2pt]
{}&{}~~= 
c(F^*) - c(F_i) - (\uin - \lout) \cdot \mu(F^* \setminus F_i) - \lout \cdot 0\\[2pt]
{}&{}~~> 
0 ,
\end{align*}
so by Lemma~\ref{lem:inv_min_cost_span_n&s_cond_for_infeas}, the problem is infeasible.
\medskip

\noindent \textbf{Case~\hyperlink{2.1.2.2}{2.1.2.2} when $\boldsymbol{F_i\setminus S_0 \subseteq F^*\setminus S_0}$.} Note that $\uin \ne + \infty$. If $\lout \ne - \infty$, then
\begin{align*}
&(c - \p{\uin - \lout, \lout}{\ell, u}{w})(F^*) - (c - \p{\uin - \lout, \lout}{\ell, u}{w})(F_i) \\[2pt]
{}&{}= 
c(F^*) - c(F_i) - \uin \cdot \mu(F^* \setminus F_i) + \lout \cdot \mu(F_i \setminus F^*) \\[2pt]
{}&{}> 
c(F^*) - c(F_i) - \big( d_i + D_i + f_3(c_i, F_i) \big) \cdot \mu(F^* \setminus F_i) + \lout \cdot 0 \\[2pt]
{}&{}=
c(F^*) - c(F_i) - (d_i + D_i) \cdot \mu(F^* \setminus F_i) - f_3(c_i, F_i) \cdot \big( \mu(F^* \setminus F_i) - 0 \big) + D_i \cdot 0 \\[2pt]
{}&{}= 
c(F^*) - c(F_i) - (d_i + D_i) \cdot \mu(F^* \setminus F_i) - f_3(c_i, F_i) \cdot \big( \mu(F^* \setminus F_i) - \mu(F_i \setminus F^*) \big)\\[2pt]
{}&{}~~+ D_i \cdot \mu(F_i \setminus F^*) \\[2pt]
{}&{}= 
\big( c(F^*) - c(F_i) - (d_i + D_i) \cdot \mu(F^* \setminus F_i) + D_i \cdot \mu(F_i \setminus F^*) \big)  - f_3(c_i, F_i) \cdot \big( \mu(F^*) - \mu(F_i) \big) \\[2pt]
{}&{}= 
c_i(F^*) - c_i(F_i) - f_3(c_i, F_i) \cdot \big( \mu(F^*) - \mu(F_i) \big)\\[2pt]
{}&{}= 
0 ,
\end{align*}
and if $\lout = - \infty$, then the same calculation applies for $\p{\uin - m, m}{\ell, u}{w}$. So by Lemma~\ref{lem:inv_min_cost_span_n&s_cond_for_infeas}, the problem is infeasible.
\medskip

\noindent \textbf{Cases~\hyperlink{2.1.2.2}{2.1.2.2}, \hyperlink{2.2.2.2}{2.2.2.2}, \hyperlink{3.1.2.2}{3.1.2.2} and~\hyperlink{3.2.2.2}{3.2.2.2} when $\boldsymbol{F_i\setminus S_0 \nsubseteq F^*\setminus S_0}$.} Note that $\uin \ne + \infty$ and $\lout \ne - \infty$. In addition,
\begin{align*}
&d_i + D_i + f_4(c_i, d_i, D_i, F_i) + f_5(c_i, d_i, D_i, F_i) \\[2pt]
{}&{}~~= 
d_i + D_i + \frac{c_i(F_i) - c_i(F^*) + (\uin - d_i - D_i) \cdot \mu(F^* \setminus F_i)}{\mu(F_i \setminus F^*)} \\[2pt]
{}&{}~~~~+ \frac{c_i(F^*) - c_i(F_i) - (\uin - d_i - D_i) \cdot \big( \mu(F^*) - \mu(F_i) \big)}{\mu(F_i \setminus F^*)} \\[2pt]
{}&{}~~= 
d_i + D_i + (\uin - d_i - D_i) \cdot \frac{\mu(F^* \setminus F_i) - \big( \mu(F^*) - \mu(F_i) \big)}{\mu(F_i \setminus F^*)} \\[2pt] 
{}&{}~~= 
\uin .
\end{align*}
Using that $D_i + f_4(c_i, d_i, D_i, F_i) < \lout$, we obtain
\begin{align*}
&(c - \p{\uin - \lout, \lout}{\ell, u}{w})(F^*) - (c - \p{\uin - \lout, \lout}{\ell, u}{w})(F_i) \\[2pt]
{}&{}~~= 
c(F^*) - c(F_i) - \uin \cdot \mu(F^* \setminus F_i) + \lout \cdot \mu(F_i \setminus F^*) \\[2pt]
{}&{}~~> 
c(F^*) - c(F_i) - \big( d_i + D_i + f_4(c_i, d_i, D_i, F_i) + f_5(c_i, d_i, D_i, F_i) \big) \cdot \mu(F^* \setminus F_i) \\[2pt]
{}&{}~~~~+ \big( D_i + f_4(w_i, d_i, D_i, F_i) \big) \cdot \mu(F_i \setminus F) \\[2pt]
{}&{}~~= 
\big( c(F^*) - c(F_i) - (d_i + D_i) \cdot \mu(F^* \setminus F_i) + D_i \cdot \mu(F_i \setminus F^*) \big) \\[2pt]
{}&{}~~~~- f_4(c_i, d_i, D_i, F_i) \cdot \big( \mu(F^* \setminus F_i) - \mu(F_i \setminus F^*) \big)- f_5(c_i, d_i, D_i, F_i) \cdot \mu(F^* \setminus F_i) \\[2pt]
{}&{}~~= 
c_i(F^*) - c_i(F_i) - f_4(c_i, d_i, D_i, F_i) \cdot \big( \mu(F^*) - \mu(F_i) \big)- f_5(c_i, d_i, D_i, F_i) \cdot \mu(F^* \setminus F_i) \\[2pt]
{}&{}~~= 
c_i(F^*) - c_i(F_i)- \frac{c_i(F_i) - c_i(F^*) + (\uin - d_i - D_i) \cdot \mu(F^* \setminus F_i)}{\mu(F_i \setminus F^*)} \cdot \big( \mu(F^*) - \mu(F_i) \big) \\[2pt]
{}&{}~~~~- \frac{c_i(F^*) - c_i(F_i) - (\uin - d_i - D_i) \cdot \big( \mu(F^*) - \mu(F_i) \big)}{\mu(F_i \setminus F^*)} \cdot \mu(F^* \setminus F_i) \\[2pt]
{}&{}~~= \big( c_i(F^*) - c_i(F_i) \big) \cdot \left( 1 + \frac{\mu(F^*) - \mu(F_i)}{\mu(F_i \setminus F^*)} - \frac{\mu(F^* \setminus F_i)}{\mu(F_i \setminus F^*)} \right)\\[2pt]
{}&{}~~= 
0 ,
\end{align*}
so by Lemma~\ref{lem:inv_min_cost_span_n&s_cond_for_infeas}, the problem is infeasible.
\medskip 

\noindent \textbf{Cases~\hyperlink{2.2.2.2}{2.2.2.2} and~\hyperlink{3.2.2.2}{3.2.2.2} when $\boldsymbol{F_i\setminus S_0 \subseteq F^*\setminus S_0}$.} Note that $\uin \ne + \infty$. In addition,
\begin{align*}
\uin
{}&{}< 
d_i + f_6(c_i, d_i, F_i) + \uout \\[2pt]
{}&{}= 
d_i + \frac{c_i(F^*) - c_i(F_i) - (\uout - D_i) \cdot \big( \mu(F^*) - \mu(F_i) \big)}{\mu(F^* \setminus F_i)} + \uout \\[2pt]
{}&{}= 
d_i + \frac{\left( f_3(c_i, F_i) - (\uout - D_i) \right) \cdot \big( \mu(F^*) - \mu(F_i) \big)}{\mu(F^* \setminus F_i)} + \uout \\[2pt]
{}&{}= 
d_i + \frac{\left( f_3(c_i, F_i) - \uout + D_i \right) \cdot \mu(F^* \setminus F_i)}{\mu(F^* \setminus F_i)} + \uout\\[2pt] 
{}&{}= 
d_i + D_i + f_3(c_i, F_i) .
\end{align*}
Therefore, the exact same calculation applies as in Case~\hyperlink{2.1.2.2}{2.1.2.2} when $F_i\setminus S_0 \subseteq F\setminus S_0$. So by Lemma~\ref{lem:inv_min_cost_span_n&s_cond_for_infeas}, the problem is infeasible.
\medskip

\noindent \textbf{Case~\hyperlink{3.1.2.2}{3.1.2.2} when $\boldsymbol{F_i\setminus S_0 \subseteq F^*\setminus S_0}$.} Note that $\uin \ne + \infty$. In addition,
\begin{align*}
\uin 
{}&{}< 
d_i + D_i + f_7(c_i, F_i, Z_i) + f_8(c_i, F_i, Z_i) \\[2pt]
{}&{}= 
d_i + D_i + f_7(c_i, F_i, Z_i) + \frac{c_i(F^*) - c_i(F_i) - f_7(c_i, F_i, Z_i) \cdot \mu(F^* \setminus F_i)}{\mu(F^*) - \mu(F_i)} \\[2pt]
{}&{}= 
d_i + D_i + f_7(c_i, F_i, Z_i) + \frac{c_i(F^*) - c_i(F_i) - f_7(c_i, F_i, Z_i) \cdot \big( \mu(F^*) - \mu(F_i) \big)}{\mu(F^*) - \mu(F_i)} \\[2pt]
{}&{}= 
d_i + D_i + \frac{c_i(F^*) - c_i(F_i)}{\mu(F^*) - \mu(F_i)}\\[2pt]
{}&{}= 
d_i + D_i + f_3(c_i, F_i) .
\end{align*}
Therefore, the exact same calculation applies as in Case~\hyperlink{2.1.2.2}{2.1.2.2} when $F_i\setminus S_0 \subseteq F\setminus S_0$. So by Lemma~\ref{lem:inv_min_cost_span_n&s_cond_for_infeas}, the problem is infeasible.
\medskip

\noindent \textbf{Case~\hyperlink{4.1.2.2}{4.1.2.2} when $\boldsymbol{F^*\setminus S_0 \subseteq F_i\setminus S_0}$.} Note that $\lout \ne - \infty$. If $\uin \ne - \infty$, then
\begin{align*}
&( c - \p{\uin - \lout, \lout}{\ell, u}{w} )(F^*) - ( c - \p{\uin - \lout, \lout}{\ell, u}{w} )(F_i) \\[2pt]
{}&{}~~= 
c(F^*) - c(F_i) - \uin \cdot \mu(F^* \setminus F_i) + \lout \cdot \mu(F_i \setminus F^*) \\[2pt]
{}&{}~~> 
c(F^*) - c(F_i) - \uin \cdot 0 + \big( D_i + f_3(c_i, F_i) \big) \cdot \mu(F_i \setminus F^*) \\[2pt]
{}&{}~~= 
c(F^*) - c(F_i) - (d_i + D_i) \cdot 0 + D_i \cdot \mu(F_i \setminus F^*)+ f_3(c_i, F_i) \cdot \mu(F_i \setminus F^*) \\[2pt]
{}&{}~~= 
\big( c(F^*) - c(F_i) - (d_i + D_i) \cdot \mu(F^* \setminus F_i) + D_i \cdot \mu(F_i \setminus F^*) \big)+ f_3(c_i, F_i) \cdot \big( \mu(F_i) - \mu(F^*) \big) \\[2pt]
{}&{}~~=
c_i(F^*) - c_i(F_i) + f_3(c_i, F_i) \cdot \big( \mu(F_i) - \mu(F^*) \big)\\[2pt]
{}&{}~~= 0 ,
\end{align*}
and if $\uin = - \infty$, then the same calculation applies for $\p{M - \lout, \lout}{\ell, u}{w}$. So by Lemma~\ref{lem:inv_min_cost_span_n&s_cond_for_infeas}, the problem is infeasible.
\medskip

\noindent \textbf{Cases~\hyperlink{4.1.2.2}{4.1.2.2}, \hyperlink{4.2.2.2}{4.2.2.2}, \hyperlink{5.1.2.2}{5.1.2.2} and~\hyperlink{5.2.2.2}{5.2.2.2} when $\boldsymbol{F^*\setminus S_0 \nsubseteq F_i\setminus S_0}$.} Note that $\lout \ne - \infty$ and $\uin \ne + \infty$. Using that $d_i + f_9(c_i, d_i) > \uin - \lout$, we obtain
\begin{align*}
&(c - \p{\uin - \lout, \lout}{\ell, u}{w})(F^*) - (c - \p{\uin - \lout, \lout}{\ell, u}{w})(F_i) \\[2pt]
{}&{}~~= 
c(F^*) - c(F_i) - (\uin - \lout) \cdot \mu(F^* \setminus F_i) - \lout \cdot \big( \mu(F^*) - \mu(F_i) \big) \\[2pt]
{}&{}~~> 
c(F^*) - c(F_i) - \big( d_i + f_9(w_i, D_i, F_i) \big) \cdot \mu(F^* \setminus F_i)- \lout \cdot \big( \mu(F^*) - \mu(F_i) \big) \\[2pt]
{}&{}~~= 
\left( c(F^*) - c(F_i) - d_i \cdot \mu(F^* \setminus F_i) - D_i \cdot \big( \mu(F^*) - \mu(F_i) \big) \right) \\[2pt]
{}&{}~~~~- f_9(c_i, D_i, F_i) \cdot \mu(F^* \setminus F_i) - (\lout - D_i) \cdot \big( \mu(F^*) - \mu(F_i) \big) \\[2pt]
{}&{}~~= 
c_i(F^*) - c_i(F_i) - (\lout - D_i) \cdot \big( \mu(F^*) - \mu(F_i) \big)- f_9(c_i, D_i, F_i) \cdot \mu(F^* \setminus F_i) \\[2pt]
{}&{}~~= 
c_i(F^*) - c_i(F_i) - (\lout - D_i) \cdot \big( \mu(F^*) - \mu(F_i) \big) \\[2pt]
{}&{}~~~~- \frac{c_i(F^*) - c_i(F_i) - (\lout - D_i) \cdot \big( \mu(F^*) - \mu(F_i) \big)}{\mu(F^* \setminus F_i)} \cdot \mu(F^* \setminus F_i)\\[2pt]
{}&{}~~=
0 ,
\end{align*}
so by Lemma~\ref{lem:inv_min_cost_span_n&s_cond_for_infeas}, the problem is infeasible.
\medskip

\noindent \textbf{Cases~\hyperlink{4.2.2.2}{4.2.2.2} and~\hyperlink{5.2.2.2}{5.2.2.2} when $\boldsymbol{F^*\setminus S_0 \subseteq F_i\setminus S_0}$.} Note that $\lout \ne - \infty$. In addition,
\begin{align*}
\lout 
{}&{}> 
D_i + f_{10}(c_i, d_i, D_i, F_i)\\[2pt]
{}&{}= 
D_i + \frac{c_i(F_i) - c_i(F^*) + (\lin - d_i - D_i) \cdot \mu\left( F^* \setminus F_i \right)}{\mu\left( F_i \setminus F^* \right)} \\[2pt]
{}&{}= 
D_i + \frac{c_i(F_i) - c_i(F^*) + (\lin - d_i - D_i) \cdot 0}{\mu\left( F_i \setminus F^* \right)}\\[2pt]
{}&{}= D_i + \frac{c_i(F_i) - c_i(F^*)}{\mu(F_i) - \mu(F^*)}\\[2pt]
{}&{}= 
D_i + f_3(c_i, F_i) .
\end{align*}
Therefore, the exact same calculation applies as in Case~\hyperlink{4.1.2.2}{4.1.2.2} when $F_i\setminus S_0 \subseteq F\setminus S_0$. So by Lemma~\ref{lem:inv_min_cost_span_n&s_cond_for_infeas}, the problem is infeasible.
\medskip

\noindent \textbf{Case~\hyperlink{5.1.2.2}{5.1.2.2} when $\boldsymbol{F^*\setminus S_0 \subseteq F_i\setminus S_0}$.} Note that $\lout \ne - \infty$ and
\begin{align*}
\lout 
{}&{}> 
D_i + f_{12}(c_i, X_i, F_i) \\[2pt]
{}&{}=
D_i + \frac{c_i(F^*) - c_i(F_i) - f_7(c_i, X_i, F_i) \cdot \mu\left( F^* \setminus F_i \right)}{\mu(F^*) - \mu(F_i)} \\[2pt]
{}&{}= 
D_i + \frac{c_i(F^*) - c_i(F_i) - f_7(c_i, X_i, F_i) \cdot 0}{\mu(F^*) - \mu(F_i)} = D_i + f_3(c_i, F_i) .
\end{align*}
Therefore, the exact same calculation applies as in Case~\hyperlink{4.1.2.2}{4.1.2.2} when $F_i\setminus S_0 \subseteq F\setminus S_0$. So by Lemma~\ref{lem:inv_min_cost_span_n&s_cond_for_infeas}, the problem is infeasible.
\medskip

\noindent \emph{Optimality.} Assume now that the algorithm terminates with returning a vector whose feasibility follows from the fact that the while loop ended. If $F^*$ is a minimum $c_0$-cost member of $\cF$, then we are clearly done. Otherwise, there exists an index $q$ such that $F^*$ is a minimum $c_{q+1}$-cost member of $\mathcal{F}$. Suppose to the contrary that $\p{d_{q+1}, D_{q+1}}{\ell, u}{w}$ is not optimal. By Lemma~\ref{lemma:inv_min_cost_span_opt_dev_vector_v2}, there exists $\delta, \Delta \in \mathbb{R}$ such that $\delta<d_{q+1}$, the deviation vector $\p{\delta, \Delta}{\ell, u}{w}$ is optimal, $\lin \le \delta + \Delta \le \uin$, and $\lout \le \Delta \le \uout$. By Lemma~\ref{lemma:inv_min_cost_span_deltas_stay_in_the_intervals}, we know that $\lin \le d_{q+1} + D_{q+1} \le \uin$ and $\lout \le D_{q+1} \le \uout$ hold. If all steps correspond to Cases~\hyperlink{2.1.1}{2.1.1} and~\hyperlink{4.1.1}{4.1.1}, then $\spa_w(\p{d_{q+1}, D_{q+1}}{\ell, u}{w})=d_{q+1}=0\leq\delta$, a contradiction. Otherwise, let $q'$ be the largest index for which $d_{q'+1} \ne 0$. Note that $d_{q+1}=d_{q'+1}$. We arrive to a contradiction using different arguments, depending on which case step $q'$ belongs to. \medskip

\noindent \textbf{Cases~\hyperlink{1.1}{1.1} and~\hyperlink{1.2.1}{1.2.1}.} 
By Lemma~\ref{lemma:inv_min_cost_span_correction}, we have
\begin{align*}
0
{}&{}\ge 
( c - \p{\delta, \Delta}{\ell, u}{w} )(F^*) - ( c - \p{\delta, \Delta}{\ell, u}{w} )(F_{q'}) \\[2pt]
{}&{}= 
\left( c(F^*) - \sum_{\makemathbox[.7\width]{s \in F^*\setminus S_0}} \frac{\delta}{w(s)} - \sum_{\makemathbox[.7\width]{s \in F^*\setminus S_0}} \frac{\Delta}{w(s)} \right)- \left( c(F_q) - \sum_{\makemathbox[.9\width]{s \in (F_q\setminus  S_0) \cap (F^*\setminus S_0)}} \frac{\delta}{w(s)} - \sum_{\makemathbox[.7\width]{s \in F_q\setminus  S_0}} \frac{\Delta}{w(s)} \right) \\[2pt]
{}&{}= 
c(F^*) - c(F_{q'}) - \delta \cdot \mu\left( F^*\setminus  F_{q'}  \right) - \Delta \cdot \left( \mu(F^*) - \mu(F_{q'}) \right) \\[2pt]
{}&{}= c(F^*) - c(F_{q'}) - \delta \cdot \mu\left( F^* \setminus  F_{q'}  \right) - \Delta \cdot 0 \\[2pt]
{}&{}> 
c(F^*) - c(F_{q'}) - d_{{q'}+1} \cdot \mu\left( F^* \setminus  F_{q'}  \right) - D_{q'} \cdot 0 \\[2pt]
{}&{}= 
c(F^*) - c(F_{q'}) - d_{{q'}+1} \cdot \mu\left( F^* \setminus  F_{q'}  \right) - D_{q'} \cdot\left( \mu(F^*) - \mu(F_{q'}) \right) \\[2pt]
{}&{}= 
c_{{q'}+1}(F^*) - c_{{q'}+1}(F_{q'})\\[2pt] 
{}&{}= 
0 ,
\end{align*}
a contradiction.
\medskip

\noindent\textbf{Cases~\hyperlink{2.1.2.1}{2.1.2.1}, \hyperlink{2.2.2.1}{2.2.2.1}, \hyperlink{3.1.2.1}{3.1.2.1}, and~\hyperlink{3.2.2.1}{3.2.2.1}.} Note that
\begin{align*}
d_{q'+1} + D_{q'+1}
{}&{}= 
d_{q'} + \delta_{q'+1} + D_{q'} + \Delta_{q'+1}\\[2pt]
{}&{}= 
d_{q'} + f_5(c_{q'}, d_{q'}, D_{q'}, F_{q'}) + D_{q'} + f_4(c_{q'}, d_{q'}, D_{q'}, F_{q'}) \\[2pt]
{}&{}= 
d_{q'} + \frac{c_{q'}(F^*) - c_{q'}(F_{q'}) - (\uin - d_{q'} - D_{q'}) \cdot \left( \mu(F^*) - \mu(F_{q'}) \right)}{\mu\left( (F_{q'}\setminus  S_0)\setminus  (F^*\setminus S_0) \right)} \\[2pt]
{}&{}~~+ D_{q'} + \frac{c_{q'}(F_{q'}) - c_{q'}(F^*) + (\uin - d_{q'} - D_{q'}) \cdot \mu\left( (F^*\setminus S_0)\setminus  (F_{q'}\setminus  S_0) \right)}{\mu\left( (F_{q'}\setminus  S_0)\setminus  (F^*\setminus S_0) \right)} \\[2pt]
{}&{}= 
d_{q'} + D_{q'} + (\uin - d_{q'} - D_{q'}) \cdot \left( - \frac{\left( \mu(F^*) - \mu(F_{q'}) \right)}{\mu\left( (F_{q'}\setminus  S_0)\setminus  (F^*\setminus S_0) \right)} + \frac{\mu\left( F^*\setminus F_{q'} \right)}{\mu\left( F_{q'} \setminus F^* \right)} \right) \\[2pt]
{}&{}= 
d_{q'} + D_{q'} + (\uin - d_{q'} - D_{q'})\\[2pt]
{}&{}= 
\uin .
\end{align*}
Therefore, by Lemma \ref{lemma:inv_min_cost_span_correction}, we have
\begin{align*}
0 
{}&{}\ge 
( c - \p{\delta, \Delta}{\ell, u}{w} )(F^*) - ( c - \p{\delta, \Delta}{\ell, u}{w} )(F_{q'}) \\[2pt]
{}&{}=
c(F^*) - c(F_{q'}) - (\delta + \Delta) \cdot \left( \mu(F^*) - \mu(F_{q'}) \right) - \delta \cdot \mu\left( F_{q'} \setminus F^* \right) \\[2pt]
{}&{}> 
c(F^*) - c(F_{q'}) - \uin \cdot \left( \mu(F^*) - \mu(F_{q'}) \right) - d_{q'+1} \cdot \mu\left( F_{q'} \setminus  F^* \right) \\[2pt]
{}&{}= 
c(F^*) - c(F_{q'}) - (d_{q'+1} + D_{q'+1}) \cdot \left( \mu(F^*)- \mu(F_{q'}) \right)- d_{q'+1} \cdot \mu\left( F_{q'} \setminus  F^*\right) \\[2pt]
{}&{}= 
c_{q'+1}(F^*) - c_{q'+1}(F_{q'})\\[2pt]
{}&{}= 
0,
\end{align*}
a contradiction.
\medskip

\noindent\textbf{Cases~\hyperlink{2.2.1}{2.2.1} and~\hyperlink{3.2.1}{3.2.1}.} Note that $D_{q'+1} = D_{q'} + \Delta_{q'+1} = D_{q'} + (\uout - D_q') = \uout$ holds. Therefore, by Lemma~\ref{lemma:inv_min_cost_span_correction}, we have
\begin{align*}
0
{}&{}\ge 
( c - \p{\delta, \Delta}{\ell, u}{w} )(F^*) - ( c - \p{\delta, \Delta}{\ell, u}{w})(F_{q'}) \\[2pt]
{}&{}= 
c(F^*) - c(F_{q'}) - \delta \cdot \mu( F^*\setminus F_{q'}) - \Delta \cdot \left( \mu(F^*) - \mu(F_{q'}) \right) \\[2pt]
{}&{}> 
c(F^*) - c(F_{q'}) - d_{q'+1} \cdot \mu(F^*\setminus F_{q'}) - \uout \cdot \left( \mu(F^*) - \mu(F_{q'}) \right) \\[2pt]
{}&{}= 
c_{q'+1}(F^*) - c_{q'+1}(F_{q'})\\[2pt]
{}&{}= 
0 ,
\end{align*}
a contradiction.
\medskip

\noindent \textbf{Case~\hyperlink{3.1.1}{3.1.1}.} We have
\begin{align*}
0 
{}&{}\ge 
( c - \p{\delta, \Delta}{\ell, u}{w} )(F^*) - ( c - \p{\delta, \Delta}{\ell, u}{w} )(F_{q'}) \\[2pt]
{}&{}= 
c(F^*) - c(F_{q'}) - \delta \cdot \mu\left( F^*\setminus F_{q'} \right) - \Delta \cdot \left( \mu(F^*) - \mu(F_{q'}) \right) \\[2pt]
{}&{}> 
c(F^*) - c(F_{q'}) - d_{q'+1} \cdot \mu(F^*\setminus F_{q'}) - \Delta \cdot \left( \mu(F^*) - \mu(F_{q'}) \right)
\end{align*}
and
\begin{align*}
0 
{}&{}\ge 
( c - p )(F^*) - ( c - p )(Z_{q'}) \\[2pt]
{}&{}= 
c(F^*) - c(Z_{q'}) - \delta \cdot \mu(F^*\setminus Z_{q'}) - \Delta \cdot \left( \mu(F^*) - \mu(Z_{q'}) \right) \\[2pt]
{}&{}> 
c(F^*) - c(Z_{q'}) - d_{q'+1} \cdot \mu\left( F^* \setminus Z_{q'} \right) - \Delta \cdot \left( \mu(F^*) - \mu(Z_{q'}) \right).
\end{align*}
These together imply
\begin{align*}
\frac{c(F^*) - c(F_{q'}) - d_{q'+1} \cdot \mu(F^*\setminus F_{q'})}{\mu(F^*) - \mu(F_{q'})} &< \Delta,\ \text{and} \\[2pt]
\frac{c(F^*) - w(Z_{q'}) - d_{q'+1} \cdot \mu(F^* \setminus Z_{q'})}{\mu(F^*) - \mu(Z_{q'})} &>\Delta.
\end{align*}
Therefore, we have 
\begin{align*}
d_{q'+1} 
{}&{}> 
\frac{\hphantom{x} \displaystyle \frac{c(F^*) - c(F_{q'})}{\mu(F^*) - \mu(F_{q'})} - \frac{c(F^*) - c(Z_{q'})}{\mu(F^*) - \mu(Z_{q'})} \hphantom{x}}{\hphantom{x} \displaystyle \frac{\mu(F^*\setminus F_{q'})}{\mu(F^*) - \mu(F_{q'})} - \frac{\mu(F^* \setminus Z_{q'})}{\mu(F^*) - \mu(Z_{q'})} \hphantom{x}} \\[7pt]
{}&{}= 
\frac{\hphantom{x} \displaystyle \frac{c_{q'}(F^*) - c_{q'}(F_{q'}) + d_{q'} \cdot \mu\left( F^*\setminus F_{q'} \right) + D_{q'} \cdot \left( \mu(F^*) - \mu(F_{q'}) \right)}{\mu(F^*) - \mu(F_{q'})}\hphantom{x} }{\displaystyle \frac{\mu(F^*\setminus F_{q'})}{\mu(F^*) - \mu(F_{q'})} - \frac{\mu(F^* \setminus Z_{q'})}{\mu(F^*) - \mu(Z_{q'})}} \\[7pt]
{}&{}~~ - \frac{\hphantom{x} \displaystyle \frac{c_{q'}(F^*) - c_{q'}(Z_{q'}) + d_{q'} \cdot \mu(F^* \setminus Z_{q'}) + D_{q'} \cdot \left( \mu(F^*) - \mu(Z_{q'}) \right)}{\mu(F^*) - \mu(Z_{q'})} \hphantom{x}}{\displaystyle \frac{\mu(F^*\setminus F_{q'})}{\mu(F^*) - \mu(F_{q'})} - \frac{\mu(F^* \setminus Z_{q'})}{\mu(F^*) - \mu(Z_{q'})}} \\[7pt]
{}&{}= 
\frac{\hphantom{x} \displaystyle \frac{c_{q'}(F^*) - c_{q'}(F_{q'})}{\mu(F^*) - \mu(F_{q'})} - \frac{c_{q'}(F^*) - c_{q'}(Z_{q'})}{\mu(F^*) - \mu(Z_{q'})} \hphantom{x}}{\displaystyle \frac{\mu( F^*\setminus F_{q'})}{\mu(F^*) - \mu(F_{q'})} - \frac{\mu(F^* \setminus Z_{q'})}{\mu(F^*) - \mu(Z_{q'})}} + d_{q'}.
\end{align*}
By the above, we obtain 
\begin{equation*} 
\delta_{q'+1} 
> 
\frac{\hphantom{x} \displaystyle \frac{c_{q'}(F^*) - c_{q'}(F_{q'})}{\mu(F^*) - \mu(F_{q'})} - \frac{c_{q'}(F^*) - c_{q'}(Z_{q'})}{\mu(F^*) - \mu(Z_{q'})} \hphantom{x}}{\hphantom{x} \displaystyle \frac{\mu(F^*\setminus F_{q'})}{\mu(F^*) - \mu(F_{q'})} - \frac{\mu(F^* \setminus Z_{q'})}{\mu(F^*) - \mu(Z_{q'})} \hphantom{x}}= 
f_7(c_{q'}, F_{q'}, Z_{q'})= 
\delta_{q'+1},
\end{equation*}
a contradiction.
\medskip

\noindent \textbf{Cases~\hyperlink{4.1.2.1}{4.1.2.1}, \hyperlink{4.2.2.1}{4.2.2.1}, \hyperlink{5.1.2.1}{5.1.2.1} and~\hyperlink{5.2.2.1}{5.2.2.1}.} Note that $D_{q'+1} = D_{q'} + \Delta_{q'+1} = D_{q'} + (\lout - D_{q'}) = \lout$ holds. Thus, by Lemma~\ref{lemma:inv_min_cost_span_correction},
\begin{align*}
0
{}&{}\ge 
( c - \p{\delta, \Delta}{\ell, u}{w} )(F^*) - ( c -\p{\delta, \Delta}{\ell, u}{w} )(F_{q'}) \\[2pt]
{}&{}= 
c(F^*) - c(F_{q'}) - \delta \cdot \mu(F^*\setminus F_{q'}) - \Delta \cdot \left( \mu(F^*) - \mu(F_{q'}) \right) \\[2pt]
{}&{}> 
c(F^*) - c(F_{q'}) - d_{q'+1} \cdot \mu(F^*\setminus F_{q'}) - \lout \cdot \left( \mu(F^*) - \mu(F_{q'}) \right) \\[2pt]
{}&{}= 
c(F^*) - c(F_{q'}) - d_{q'+1} \cdot \mu(F^*\setminus F_{q'}) - D_{q'+1} \cdot \left( \mu(F^*) - \mu(F_{q'}) \right) \\[2pt]
{}&{}= 
c_{q'+1}(F^*) - c_{q'+1}(F_q') \\[2pt]
{}&{}= 0 ,
\end{align*}
a contradiction.
\medskip

\noindent \textbf{Cases~\hyperlink{4.2.1}{4.2.1} and~\hyperlink{5.2.1}{5.2.1}.} Note that
\begin{align*}
d_{q'+1} + D_{q'+1}
{}&{}= d_{q'} + \delta_{q'+1} + D_{q'} + \Delta_{q'+1}\\[2pt] 
{}&{}= d_{q'} + f_{11}(c_{q'}, d_{q'}, D_{q'}, F_{q'}) + D_{q'} + f_{10}(c_{q'}, d_{q'}, D_{q'}, F_{q'}) \\[2pt]
{}&{}= 
d_{q'} + \frac{c_{q'}(F^*) - c_{q'}(F_{q'}) - (\lin - d_{q'} - D_{q'}) \cdot \left( \mu(F^*) - \mu(F_{q'}) \right)}{\mu(F_{q'} \setminus F^*)} \\[2pt]
{}&{}~~ + 
D_{q'} + \frac{c_{q'}(F_{q'}) - c_{q'}(F^*) + (\lin - d_{q'} - D_{q'}) \cdot \mu\left( F^*\setminus F_{q'} \right)}{\mu(F_{q'} \setminus F^*)} \\[2pt]
{}&{}= 
d_{q'} + D_{q'} + (\lin - d_{q'} - D_{q'}) \cdot \left( - \frac{\mu(F^*) - \mu(F_{q'})}{\mu(F_{q'} \setminus F^*)} + \frac{\mu(F^*\setminus F_{q'})}{\mu(F_{q'} \setminus F^*)} \right) \\[2pt]
{}&{}= 
d_{q'} + D_{q'} + (\lin - d_{q'} - D_{q'})\\[2pt]
{}&{}= 
\lin.
\end{align*}
Thus, by Lemma~\ref{lemma:inv_min_cost_span_correction},
\begin{align*}
0 
{}&{}\ge 
( c - \p{\delta, \Delta}{\ell, u}{w} )(F^*) - ( c - \p{\delta, \Delta}{\ell, u}{w} )(F_{q'}) \\[2pt]
{}&{}= 
c(F^*) - c(F_{q'}) - (\delta + \Delta) \cdot \left( \mu(F^*) - \mu(F_{q'}) \right) - \delta \cdot \mu(F_{q'} \setminus F^*) \\[2pt]
{}&{}> 
c(F^*) - c(F_{q'}) - \lin \cdot \left( \mu(F^*) - \mu(F_{q'}) \right) - d_{q'+1} \cdot \mu(F_{q'} \setminus F^*) \\[2pt]
{}&{}= 
c(F^*) - c(F_{q'}) - (d_{q'+1} + D_{q'+1}) \cdot \left( \mu(F^*) - \mu(F_{q'}) \right)- d_{q'+1} \cdot \mu(F_{q'} \setminus F^*) \\[2pt]
{}&{}= 
c_{q'+1}(F^*) - c_{q'+1}(F_{q'})\\[2pt]
{}&{}= 
0 ,
\end{align*}
a contradiction.
\medskip

\noindent \textbf{Case~\hyperlink{5.1.1}{5.1.1}.} Similarly as in Case~\hyperlink{3.1.1}{3.1.1}, we obtain
\begin{align*}
\frac{c(F^*) - c(X_{q'}) - d_{q'+1} \cdot \mu(F^* \setminus X_{q'})}{\mu(F^*) - \mu(X_{q'})} &< \Delta,\ \text{and} \\[2pt]
\frac{c(F^*) - c(F_{q'}) - d_{q'+1} \cdot \mu(F^*\setminus F_{q'})}{\mu(F^*) - \mu(F_{q'})} &>\Delta.
\end{align*}
Therefore, we get
\begin{align*}
\delta_{q'+1} 
> 
\frac{\hphantom{x} \displaystyle \frac{c_{q'}(F^*) - c_{q'}(X_{q'})}{\mu(F^*) - \mu(X_{q'})} - \frac{c_{q'}(F^*) - c_{q'}(F_{q'})}{\mu(F^*) - \mu(F_{q'})} \hphantom{x}}{\displaystyle \frac{\mu( F^* \setminus X_{q'})}{\mu(F^*) - \mu(X_{q'})} - \frac{\mu(F^*\setminus F_{q'})}{\mu(F^*) - \mu(F_{q'})}}
= 
f_7(c_{q'}, X_{q'}, F_{q'})
= 
\delta_{q'+1},
\end{align*}
a contradiction.
\end{proof}

As we showed in Section~\ref{sec:algorithm} that an arbitrary instance $\big( S, \cF, F^*, c, \ell, u, \spa_w(\cdot) \big)$ of the constrained min\-i\-mum-cost inverse optimization problem under the weighted span objective can be reduced to solving $O(n^2)$ subproblems, we get the following result.

\begin{cor}
There exists an algorithm that determines an optimal deviation vector, if exists, for the bound-constrained minimum-cost inverse optimization problem under the weighted span objective using $O(n^2\cdot\|w\|_{\text{-}1}^6)$ calls to $\cO$.
\end{cor} 

%%%%%%%%%%%%%%%%%%%%%%%%%%%%%%%%
\section{Min-max characterization and multiple costs} 
\label{sec:minmax}
%%%%%%%%%%%%%%%%%%%%%%%%%%%%%%%%

With the help of Corollary~\ref{cor:inv_min_cost_span_opt_dev_vector}, we give a min-max characterization for the weighted span of an optimal deviation vector in the unconstrained setting, even for the case of multiple cost functions. The theorem may seem complicated at first glance as it involves a rather complex formula. However, let us note that the fractions appearing on the maximum side of the characterization are natural lower bounds for the minimum span of a feasible deviation vector, and the complexity of the formula is simply due to the presence of weights and bound constraints.

\begin{thm} \label{thm:inv_min_cost_span_minmax}
Let $\big( S, \mathcal{F},F^*, \{c^j\}_{j\in[k]}, -\infty,  +\infty, \spa_w (\cdot) \big)$ be a feasible minimum-cost inverse optimization problem. Then
\begin{gather*}
\min \left\{\textstyle\spa_w(p) \bigm| \text{$p$ is a feasible deviation vector} \right\} \\
= \max \{ 0,\, \omega_1,\, \omega_2 \},
\end{gather*}
where
\begin{align*}
\omega_1{}&{}\coloneqq \max \left\{ \frac{c^j(F^*) - c^j(F'')}{\tfrac{1}{w}(F^*\setminus  F'')}\, \middle|\, j\in[k],\, F'' \in \mathcal{F}\setminus\{F^*\}, \, \tfrac{1}{w}(F'') = \tfrac{1}{w}(F^*) \right\},\ \text{and}\\[2pt]
\omega_2{}&{}\coloneqq 
\max \left\{\frac{\displaystyle \frac{c^{j_1}(F^*) - c^{j_1}(F')}{\tfrac{1}{w}(F^*) - \tfrac{1}{w}(F')} - \frac{c^{j_2}(F^*) - c^{j_2}(F''')}{\tfrac{1}{w}(F^*) - \tfrac{1}{w}(F''')}}{\displaystyle \frac{\tfrac{1}{w}(F^*\setminus  F')}{\tfrac{1}{w}(F^*) - \tfrac{1}{w}(F')} - \frac{\tfrac{1}{w}(F^*\setminus  F''')}{\tfrac{1}{w}(F^*) - \tfrac{1}{w}(F''')}} \, \middle|\, \substack{ \displaystyle j_1,j_2\in[k],\ F', F''' \in \mathcal{F}, \\[2pt]  \displaystyle \tfrac{1}{w}(F') < \tfrac{1}{w}(F^*) < \tfrac{1}{w}(F''')}\right\}.
\end{align*}
\end{thm}
\begin{proof}
Let $p$ be an optimal deviation vector. By Corollary~\ref{cor:inv_min_cost_span_opt_dev_vector}, we may assume that $p$ is of the form $\p{\delta, \Delta}{}{w}$ for some $\delta,\Delta\in\mathbb{R}$ such that $\min \{ w(s) \cdot p(s) \mid s \in S \} = \Delta$ and $\max \{ w(s) \cdot p(s) \mid s \in S \} = \delta + \Delta$. Note that $\delta=\max \{ w(s) \cdot p(s) \mid s \in S \} - \min \{ w(s) \cdot p(s) \mid s \in S \} \ge 0$ clearly holds. For ease of discussion, let $d \coloneqq \max \{ 0,\, \omega_1,\, \omega_2 \}$. Using this notation, we define
\begin{equation*}
D \coloneqq 
\begin{cases}
\displaystyle\max_{\substack{j\in[k] \\ F' \in \mathcal{F}\\ \frac{1}{w}(F') < \frac{1}{w}(F^*)}} \left\{ \displaystyle \frac{c^j(F^*) - c^j(F') - d \cdot \tfrac{1}{w}(F^*\setminus  F')}{\tfrac{1}{w}(F^*) - \tfrac{1}{w}(F')} \right\} & \text{if $\{ F' \in \mathcal{F} \mid \tfrac{1}{w}(F') < \tfrac{1}{w}(F^*) \} \ne \emptyset$},\\[2pt]
\displaystyle\min_{\substack{j\in[k] \\ F''' \in \mathcal{F}\\ \frac{1}{w}(F''') > \frac{1}{w}(F^*)}} \left\{ \displaystyle \frac{c^j(F^*) - c^j(F''') - d \cdot \tfrac{1}{w}(F^*\setminus  F''')}{\tfrac{1}{w}(F^*) - \tfrac{1}{w}(F''')} \right\} & \text{if $\{ F' \in \mathcal{F} \mid \tfrac{1}{w}(F') < \tfrac{1}{w}(F^*) \} = \emptyset$,}\\[-20pt]
& \text{$\{F''' \in\mathcal{F} \mid \tfrac{1}{w}(F''')>\tfrac{1}{w}(F^*) \} \neq \emptyset$},\\[10pt]
0 & \text{otherwise}.
\end{cases}
\end{equation*} 
Let $j\in[k]$ and $F \in \mathcal{F}\setminus\{F^*\}$ be an arbitrary solution. Since $\p{\delta, \Delta}{}{w}$ is feasible, we have
\begin{align*}
0 
{}&{}\ge 
(c^j - \p{\delta, \Delta}{}{w})(F^*) - (c^j - \p{\delta, \Delta}{}{w})(F) \\[2pt]
{}&{}= 
\left( c^j(F^*) - \delta \cdot \tfrac{1}{w}(F^*) - \Delta \cdot \tfrac{1}{w}(F^*) \right) - \left( c^j(F^*) - \delta \cdot \tfrac{1}{w}(F \cap F^*) - \Delta \cdot \tfrac{1}{w}(F) \right) \\[2pt]
{}&{}= 
c^j(F^*) - c^j(F) - \delta \cdot \tfrac{1}{w}(F^*\setminus  F) - \Delta \cdot \left( \tfrac{1}{w}(F^*) -\tfrac{1}{w}(F) \right).
\end{align*}
Thus for any $j\in[k]$ and $F'' \in \mathcal{F}$ with $F'' \ne F^*$ and $\tfrac{1}{w}(F'') = \tfrac{1}{w}(F^*)$, if such $F''$ exists,
\begin{equation*}
\delta \ge \frac{c^j(F^*) - c^j(F'')}{\tfrac{1}{w}(F^*\setminus  F'')},
\end{equation*}
for any $j\in[k]$ and $F' \in \mathcal{F}$ with $\tfrac{1}{w}(F') < \tfrac{1}{w}(F^*)$, if such $F'$ exists,
\begin{equation*}
\Delta \ge \frac{c^j(F^*) - c^j(F') - \delta \cdot \tfrac{1}{w}(F^*\setminus  F')}{\tfrac{1}{w}(F^*) - \tfrac{1}{w}(F')},
\end{equation*}
and for any $j\in[k]$ and $F''' \in \mathcal{F}$ with $\tfrac{1}{w}(F''') > \tfrac{1}{w}(F^*)$, if such $F'''$ exists,
\begin{equation*}
\Delta \le \frac{c^j(F^*) - c^j(F''') - \delta \cdot \tfrac{1}{w}(F^*\setminus  F''')}{\tfrac{1}{w}(F^*) - \tfrac{1}{w}(F''')}.
\end{equation*}
By the above, for any $j_1,j_2\in[k]$ and $F', F''' \in \mathcal{F}$ with $\tfrac{1}{w}(F') < \tfrac{1}{w}(F^*) < \tfrac{1}{w}(F''')$, if such $F'$ and $F'''$ exist, we have
\begin{equation*}
\frac{c^{j_1}(F^*) - c^{j_1}(F') - \delta \cdot \tfrac{1}{w}(F^*\setminus  F')}{\tfrac{1}{w}(F^*) - \tfrac{1}{w}(F')} \le \frac{c^{j_2}(F^*) - c^{j_2}(F''') - \delta \cdot \tfrac{1}{w}(F^*\setminus  F''')}{\tfrac{1}{w}(F^*) - \tfrac{1}{w}(F''')},
\end{equation*}
implying
\begin{equation*}
\delta \ge \frac{\displaystyle \frac{c^{j_1}(F^*) - c^{j_1}(F')}{\tfrac{1}{w}(F^*) - \tfrac{1}{w}(F')} - \frac{c^{j_2}(F^*) - c^{j_2}(F''')}{\tfrac{1}{w}(F^*) - \tfrac{1}{w}(F''')}}{\displaystyle \frac{\tfrac{1}{w}(F^*\setminus  F')}{\tfrac{1}{w}(F^*) - \tfrac{1}{w}(F')} - \frac{\tfrac{1}{w}(F^*\setminus  F''')}{\tfrac{1}{w}(F^*) - \tfrac{1}{w}(F''')}}.
\end{equation*}
Therefore $\delta \ge d$ holds. To prove $\delta \le d$, it suffices to show that $\p{d, D}{}{w}$ is a feasible deviation vector, that is, $F^*$ has minimum cost with respect to $c^j-\p{d, D}{}{w}$ for every $j\in[k]$. For any $j\in[k]$ and $F' \in \mathcal{F}$ with $\tfrac{1}{w}(F') < \tfrac{1}{w}(F^*)$, if such $F'$ exists,
\begin{align*}
&(c^j - \p{d, D}{}{w})(F^*) - (c^j - \p{d, D}{}{w})(F') \\[2pt]
{}&{}~~= 
c^j(F^*) - c^j(F') - d \cdot \tfrac{1}{w}(F^*\setminus  F') - D \cdot \left( \tfrac{1}{w}(F^*) - \tfrac{1}{w}(F') \right) \\[2pt]
{}&{}~~\le 
c^j(F^*) - c^j(F') - d \cdot \tfrac{1}{w}(F^*\setminus  F') - \frac{c^j(F^*) - c^j(F') - d \cdot \tfrac{1}{w}(F^*\setminus  F')}{\tfrac{1}{w}(F^*) - \tfrac{1}{w}(F')} \cdot \left( \tfrac{1}{w}(F^*) - \tfrac{1}{w}(F') \right)\\[2pt]
{}&{}~~= 
0.
\end{align*}
For any $j\in[k]$ and $F'' \in \mathcal{F}$ with $F'' \ne F$ and $\tfrac{1}{w}(F'') = \tfrac{1}{w}(F^*)$, if such $F''$ exists,
\begin{align*}
&(c^j - \p{d, D}{}{w})(F^*) - (c^j - \p{d, D}{}{w})(F'') \\[2pt]
{}&{}~~= 
c^j(F^*) - c^j(F'') - d \cdot \tfrac{1}{w}(F^*\setminus  F'') - D \cdot \left( \tfrac{1}{w}(F^*) - \tfrac{1}{w}(F'') \right) \\[2pt]
{}&{}~~= 
c^j(F^*) - c^j(F'') - d \cdot \tfrac{1}{w}(F^*\setminus  F'') - D \cdot 0 \\[2pt]
{}&{}~~\le 
c^j(F^*) - c^j(F'') - \frac{c^j(F^*) - c^j(F'')}{\tfrac{1}{w}(F^*\setminus  F'')} \cdot \tfrac{1}{w}(F^*\setminus  F'')\\[2pt]
{}&{}~~= 
0.
\end{align*}
Let $j\in[k]$ and $F''' \in \mathcal{F}$ with $\tfrac{1}{w}(F''') > \tfrac{1}{w}(F^*)$ be arbitrary, if such $F'''$ exists. First note that
\begin{equation*}
D \le \frac{c^j(F^*) - c^j(F''') - d \cdot \tfrac{1}{w}(F^*\setminus  F''')}{\tfrac{1}{w}(F^*) - \tfrac{1}{w}(F''')}
\end{equation*}
holds since otherwise there would exist $F' \in \mathcal{F}$ with $\tfrac{1}{w}(F') < \tfrac{1}{w}(F^*)$ such that
\begin{equation*}
\frac{c^j(F^*) - c^j(F') - d \cdot \tfrac{1}{w}(F^*\setminus  F')}{\tfrac{1}{w}(F^*) - \tfrac{1}{w}(F')} > \frac{c^j(F^*) - c^j(F''') - d \cdot \tfrac{1}{w}(F^*\setminus  F''')}{\tfrac{1}{w}(F^*) - \tfrac{1}{w}(F''')},
\end{equation*}
contradicting the definition of $d$. Thus we get
\begin{align*}
&(c^j - \p{d, D}{}{w})(F^*) - (c^j - \p{d, D}{}{w})(F''') \\[2pt]
{}&{}~~= c^j(F^*) - c^j(F''') - d \cdot \tfrac{1}{w}(F^*\setminus  F''') - D \cdot \left( \tfrac{1}{w}(F^*) - \tfrac{1}{w}(F''') \right) \\[2pt]
{}&{}~~\le c^j(F^*) - c^j(F''') - d \cdot \tfrac{1}{w}(F^*\setminus  F''')\\[2pt]
{}&{}~~~~- \frac{c^j(F^*) - c^j(F''') - d \cdot \tfrac{1}{w}(F^*\setminus  F''')}{\tfrac{1}{w}(F^*) - \tfrac{1}{w}(F''')} \cdot \left( \tfrac{1}{w}(F^*) - \tfrac{1}{w}(F''') \right) \\[2pt]
{}&{}~~= 
0.
\end{align*}
This shows that $\p{d, D}{}{w}$ is indeed feasible, concluding the proof of the theorem.
\end{proof}

In~\cite{berczi2023infty}, the authors showed that the algorithm for the weighted $\ei$-norm objective naturally extends to the case of multiple cost functions. Roughly, this is doable since it suffices to compute an optimal deviation vector (of a special form) for each cost function separately, and one of these vectors can be shown to be optimal for the multiple costs version as well. 

Unfortunately, a similar approach does not apply for the weighted span objective, see Figure~\ref{fig:multiple} for an example. To overcome this, instead of considering the cost functions separately, one can run Algorithm~\ref{algo:inv_min_cost_span} for them simultaneously. That is, we still keep track of the small and large bad sets that were found the latest, but now together with the cost function $c^j$ for which those were bad. In each iteration, we pick an index $j$ such that $F^*$ is not optimal with respect to the $j$th modified cost function, and update the deviation vector according to the same rules that were applied to the case of a single cost function. If the number of cost functions is $k$, then this adds an additional factor of $k$ to the running time of the algorithm, that is, it makes $O(k\cdot n^2 \cdot \|w\|_{\text{-}1}^6)$ calls to oracle $\cO$.

\begin{figure}[H]
 \centering

 \begin{tikzpicture}[scale=0.85]
 \tikzstyle{vertex}=[draw,circle,fill=black,minimum size=5,inner sep=0]
 
 \node at (3.5,1.5) {\footnotesize $\begin{array}{r@{}c@{}l} \lin &{} = {}& -\infty, \\ \uin &{} = {}& +\infty, \end{array}$};
 \node at (6.5,1.5) {\footnotesize $\begin{array}{r@{}c@{}l} \lout &{} = {}& -\infty, \\ \uout &{} = {}& +\infty \end{array}$};
 
 \node at (1,0.5) {\footnotesize $F^*$};
 \draw (2,0) ellipse (1.2 and 0.5);
 \node[vertex] at (1.3,0) {};
 \node[vertex] at (2.7,0) {};

 \draw (4.5,0) ellipse (0.5 and 0.5);
 \node[vertex] at (4.5,0) {};
 
 \draw (7.5,0) ellipse (1.7 and 0.5);
 \node[vertex] at (6.3,0) {};
 \node[vertex] at (7.5,0) {};
 \node[vertex] at (8.7,0) {};
 
 \node at (1.3,-1.5) {\footnotesize $\begin{array}{c} (1,1) \\ \downarrow \\ (0,0) \end{array}$};
 \node at (2.7,-1.5) {\footnotesize $\begin{array}{c} (1,1) \\ \downarrow \\ (0,0) \end{array}$};
 \node at (4.5,-1.5) {\footnotesize $\begin{array}{c} (0,5) \\ \downarrow \\ (0,5) \end{array}$};
 \node at (6.3,-1.5) {\footnotesize $\begin{array}{c} (2,0) \\ \downarrow \\ (2,0) \end{array}$};
 \node at (7.5,-1.5) {\footnotesize $\begin{array}{c} (2,0) \\ \downarrow \\ (2,0) \end{array}$};
 \node at (8.7,-1.5) {\footnotesize $\begin{array}{c} (2,0) \\ \downarrow \\ (2,0) \end{array}$};
 \end{tikzpicture}
 \caption{Example for two cost functions, where the original costs of an element $s\in S$ are presented as $(c^1(s),c^2(s))$ in the top row, and the sets denote the members of $\cF$. The optimal deviation vectors for $c^1$ and for $c^2$ are defined by $\delta^1 = 0,\ \Delta^1 = 2$ and by $\delta^2 = 0,\ \Delta^2 = -2$, respectively. However, when both cost functions are considered, the optimal deviation vector corresponds to $\delta = 1,\ \Delta = 0$, resulting in the modified costs shown in the bottom row.} \label{fig:multiple}
\end{figure}

%%%%%%%%%%%%%%%%%%%%%%%%
\section{Conclusions}
\label{sec:conclusions}
%%%%%%%%%%%%%%%%%%%%%%%%

In this paper, we introduced an objective for minimum-cost inverse optimization problems that measures the difference between the largest and the smallest weighted coordinates of the deviation vector, thus leading to a fair or balanced solution. We presented a purely combinatorial algorithm that efficiently determines an optimal deviation vector, assuming that an oracle for solving the underlying optimization problem is available. The running time of the algorithm is pseudo-polynomial due to the presence of weights, and finding a strongly polynomial algorithm remains an intriguing open problem.

%%%%%%%%%%%%%%%%%%%%%%%%
\paragraph{Acknowledgement.} The work was supported by the Lend\"ulet Programme of the Hungarian Academy of Sciences -- grant number LP2021-1/2021 and by the Hungarian National Research, Development and Innovation Office -- NKFIH, grant number FK128673.
%%%%%%%%%%%%%%%%%%%%%%%%

%%%%%%%%%%%%%%%%
\bibliographystyle{abbrv}
\bibliography{part2}
%%%%%%%%%%%%%%%%

%%%%%%%%%%%%%%%%%%%%%%%%%%%%%%%%%%%%%%%%%%%%%%%%%%%%%%%%%%%%%%%%%%%%%%%%%%%%%%%%%%%%%%%%%%%%%%%%

 \clearpage

\appendix
%%%%%%%%%%%%%%%%
\section*{Appendix}
\label{sec:appendix}
%%%%%%%%%%%%%%%%

%%%%%%%%%%%%%%%%
\section{Deferred proofs}
\label{sec:appendixa}
%%%%%%%%%%%%%%%%

%%%%%%%%%%%%%%%%
\subsection{Proof of Lemma~\ref{lemma:inv_min_cost_span_correction}}
%%%%%%%%%%%%%%%%

\begin{proof}[Proof of Lemma~\ref{lemma:inv_min_cost_span_correction}]
Assume that $F^*$ is not a minimum $c_i$-cost member of $\mathcal{F}$ and that Algorithm~\ref{algo:inv_min_cost_span} does not declare the problem to be infeasible in the $i$th step. We distinguish the different cases the $i$th step might belong to.
\medskip

\noindent \textbf{Cases~\hyperlink{1.1}{1.1} and~\hyperlink{1.2.1}{1.2.1}.} Since
\begin{equation*}
\delta_{i+1} = \frac{c_i(F^*) - c_i(F_i)}{\mu(F^* \setminus F_i)},
\end{equation*}
we have
\begin{align*}
&c_{i+1}(F^*) - c_{i+1}(F_i) \\[2pt]
{}&{}~~= 
c_i(F^*) - c_i(F_i) - \delta_{i+1} \cdot \mu(F^* \setminus F_i) - \Delta_{i+1} \cdot \big( \mu(F^*) - \mu(F_i) \big) \\[2pt]
{}&{}~~= 
c_i(F^*) - c_i(F_i) - \delta_{i+1} \cdot \mu(F^* \setminus F_i) - \Delta_{i+1} \cdot 0 \\[2pt]
{}&{}~~=
0.
\end{align*}

\noindent \textbf{Cases~\hyperlink{2.1.1}{2.1.1}, \hyperlink{2.2.1}{2.2.1}, \hyperlink{3.1.1}{3.1.1}, \hyperlink{3.2.1}{3.2.1}, \hyperlink{4.1.1}{4.1.1}, \hyperlink{4.1.2.1}{4.1.2.1}, \hyperlink{4.2.2.1}{4.2.2.1}, \hyperlink{5.1.1}{5.1.1}, \hyperlink{5.1.2.1}{5.1.2.1}, and~\hyperlink{5.2.2.1}{5.2.2.1}.} In Cases~\hyperlink{2.1.1}{2.1.1}, \hyperlink{3.1.1}{3.1.1}, \hyperlink{4.1.1}{4.1.1}, and~\hyperlink{5.1.1}{5.1.1}, we have
\begin{equation*}
\Delta_{i+1} = \frac{c_i(F^*) - c_i(F_i) - \delta_{i+1} \cdot \mu(F^* \setminus F_i)}{\mu(F^*) - \mu(F_i)},
\end{equation*}
while in Cases~\hyperlink{2.2.1}{2.2.1}, \hyperlink{3.2.1}{3.2.1}, \hyperlink{4.1.2.1}{4.1.2.1}, \hyperlink{4.2.2.1}{4.2.2.1}, \hyperlink{5.1.2.1}{5.1.2.1}, and~\hyperlink{5.2.2.1}{5.2.2.1}, we have
\begin{equation*}
\delta_{i+1} = \frac{c_i(F^*) - c_i(F_i) - \Delta_{i+1} \cdot \big( \mu(F^*) - \mu(F_i) \big)}{\mu(F^* \setminus F_i)}.
\end{equation*}
Therefore in all these cases, we obtain
\begin{align*}
&c_{i+1}(F^*) - c_{i+1}(F_i) \\[2pt]
{}&{}~~= 
c_i(F^*) - c_i(F_i) - \delta_{i+1} \cdot \mu(F^* \setminus F_i) - \Delta_{i+1} \cdot \big( \mu(F^*) - \mu(F_i) \big)\\[2pt]
{}&{}~~= 
0 .
\end{align*}

\noindent \textbf{Cases~\hyperlink{2.1.2.1}{2.1.2.1}, \hyperlink{2.2.2.1}{2.2.2.1}, \hyperlink{3.1.2.1}{3.1.2.1}, and~\hyperlink{3.2.2.1}{3.2.2.1}.} Since $\delta_{i+1} = f_5(c_i, d_i, D_i, F_i)$ and $\Delta_{i+1} = f_4(c_i, d_i, D_i, F_i)$, we have
\begin{align*}
&c_{i+1}(F^*) - c_{i+1}(F_i) \\[2pt]
{}&{}~~= 
c_i(F^*) - c_i(F_i) - \delta_{i+1} \cdot \mu(F^* \setminus F_i) - \Delta_{i+1} \cdot \big( \mu(F^*) - \mu(F_i) \big) \\[2pt]
{}&{}~~=
c_i(F^*) - c_i(F_i) - f_5(c_i, d_i, D_i, F_i) \cdot \mu(F^* \setminus F_i)- f_4(c_i, d_i, D_i, F_i) \cdot \big( \mu(F^*) - \mu(F_i) \big) \\[2pt]
{}&{}~~= 
\big( c_i(F^*) - c_i(F_i) \big) \cdot \left( 1 - \frac{\mu(F^* \setminus F_i)}{\mu(F_i \setminus F^*)} + \frac{\mu(F^*) - \mu(F_i)}{\mu(F_i \setminus F^*)} \right)\\[2pt]
{}&{}~~= 
0 .
\end{align*}

\noindent \textbf{Cases~\hyperlink{4.2.1}{4.2.1} and~\hyperlink{5.2.1}{5.2.1}.} Since $\delta_{i+1} = f_{11}(c_i, d_i, D_i, F_i)$ and $\Delta_{i+1} = f_{10}(c_i, d_i, D_i, F_i)$, we have
\begin{align*}
&c_{i+1}(F^*) - c_{i+1}(F_i) \\[2pt]
{}&{}~~= 
c_i(F^*) - c_i(F_i) - \delta_{i+1} \cdot \mu(F^* \setminus F_i) - \Delta_{i+1} \cdot \big( \mu(F^*) - \mu(F_i) \big) \\[2pt]
{}&{}~~= 
c_i(F^*) - c_i(F_i)- f_{11}(c_i, d_i, D_i, F_i) \cdot \mu(F^* \setminus F_i)- f_{10}(c_i, d_i, D_i, F_i) \cdot \big( \mu(F^*) - \mu(F_i) \big) \\[2pt]
{}&{}~~= 
\big( c_i(F^*) - c_i(F_i) \big) \cdot \left( 1 - \frac{\mu(F^* \setminus F_i)}{\mu(F_i \setminus F^*)} + \frac{\mu(F^*) - \mu(F_i)}{\mu(F_i \setminus F^*)} \right)\\[2pt]
{}&{}~~= 
0 .
\end{align*}
This concludes the proof of the lemma.
\end{proof}

%%%%%%%%%%%%%%%%
\subsection{Technical claims}
%%%%%%%%%%%%%%%%

Before moving the the proofs of the other lemmas, we need the following technical claims.

\begin{claim} \label{claim:inv_min_cost_span_reformulating_Delta_iplusone}
If $\mu(F_i) < \mu(F^*)$ and $Z_i \ne \ast$, then 
\begin{align*}
f_8(c_i, F_i, Z_i) 
{}&{}= 
\frac{c_i(F^*) - c_i(F_i) - f_7(c_i, F_i, Z_i) \cdot \mu(F^* \setminus F_i)}{\mu(F^*) - \mu(F_i)} \\[2pt]
{}&{}= 
\frac{c_i(F^*) - c_i(Z_i) - f_7(c_i, F_i, Z_i) \cdot \mu(F^* \setminus Z_i)}{\mu(F^*) - \mu(Z_i)}.
\end{align*}
If $\mu(F_i) > \mu(F^*)$ and $X_i \ne \ast$, then 
\begin{align*}
f_{12}(c_i, X_i, F_i) 
{}&{}= 
\frac{c_i(F^*) - c_i(F_i) - f_7(c_i, X_i, F_i) \cdot \mu(F^* \setminus F_i)}{\mu(F^*) - \mu(F_i)} \\[2pt]
{}&{}= 
\frac{c_i(F^*) - c_i(X_i) - f_7(c_i, X_i, F_i) \cdot \mu(F^* \setminus X_i)}{\mu(F^*) - \mu(X_i)}.
\end{align*}
\end{claim}
\begin{proof}
The statement directly follows from the definitions of the functions $f_7$, $f_8$, and $f_{12}$. 
\end{proof}

\begin{claim} \label{claim:inv_min_cost_span_correction_for_a_pair}
\mbox{}
\begin{enumerate}[label=(\alph*)]\itemsep0em
\item \label{v1:a} If $\mu(F_i) < \mu(F^*)$ and $Z_i, Z_{i+1} \ne \ast$, then either $c_{i+1}(F^*) = c_{i+1}(Z_{i+1})$ or Algorithm~\ref{algo:inv_min_cost_span} declares the problem to be infeasible. 
\item \label{v1:b} If $\mu(F_i) > \mu(F^*)$ and $X_i, X_{i+1} \ne \ast$, then either $c_{i+1}(F^*) = c_{i+1}(X_{i+1})$ or Algorithm~\ref{algo:inv_min_cost_span} declares the problem to be infeasible.
\end{enumerate}
\end{claim}
\begin{proof}
Assume that Algorithm~\ref{algo:inv_min_cost_span} does not declare the problem to be infeasible in the $i$th step. First, consider the case when $\mu(F_i) < \mu(F^*)$ and $Z_i, Z_{i+1} \ne \ast$. Then the $i$th step corresponds to Case~\hyperlink{3.1.1}{3.1.1}, and $Z_i = Z_{i+1}$. Thus, by Claim~\ref{claim:inv_min_cost_span_reformulating_Delta_iplusone}, we have
\begin{align*}
&c_{i+1}(F^*) - c_{i+1}(Z_{i+1}) \\[2pt]
{}&{}~~= 
c_i(F^*) - c_i(Z_i) - f_7(c_i, F_i, Z_i) \cdot \mu(F^* \setminus Z_i)- f_8(c_i, F_i, Z_i) \cdot \big( \mu(F^*) - \mu(Z_i) \big) \\[2pt]
{}&{}~~= 
0 .
\end{align*}

Now consider the case when $\mu(F_i) > \mu(F^*)$ and $X_i, X_{i+1} \ne \ast$. Then in the $i$th step corresponds to Case~\hyperlink{5.1.1}{5.1.1}, and $X_i = X_{i+1}$. Thus, by Claim~\ref{claim:inv_min_cost_span_reformulating_Delta_iplusone}, we have
\begin{align*}
&c_{i+1}(F^*) - c_{i+1}(X_{i+1}) \\[2pt]
{}&{}~~= 
c_i(F^*) - c_i(X_i) - f_7(c_i, X_i, F_i) \cdot \mu(F^* \setminus X_i)- f_{12}(c_i, X_i, F_i) \cdot \big( \mu(F^*) - \mu(X_i) \big) \\[2pt]
{}&{}~~= 
0 .
\end{align*}
This concludes the proof of the claim.
\end{proof}

\begin{claim} \label{claim:inv_min_cost_span_correction_for_a_pair_v2}
\mbox{}
\begin{enumerate}[label=(\alph*)]\itemsep0em
\item \label{v2:a} If $Z_i \ne \ast$, then either $c_i(F^*) = c_i(Z_i)$ or Algorithm~\ref{algo:inv_min_cost_span} declares the problem to be infeasible.
\item \label{v2:b} If $X_i \ne \ast$, then either $c_i(F^*) = c_i(X_i)$ or Algorithm~\ref{algo:inv_min_cost_span} declares the problem to be infeasible.
\end{enumerate}
\end{claim}
 \begin{proof}
Assume that Algorithm~\ref{algo:inv_min_cost_span} does not declare the problem to be infeasible in the $i$th step. First, consider the case when $Z_i \ne \ast$. Then in the $(i-1)$th step we were not in Cases~\hyperlink{1.1}{1.1}, \hyperlink{1.2.1}{1.2.1}, or~\hyperlink{1.2.2}{1.2.2}, thus $\mu(F_{i-1}) \ne \mu(F^*)$. If $\mu(F_{i-1}) > \mu(F^*)$, then $Z_i = F_{i-1}$ and the statement follows from Lemma~\ref{lemma:inv_min_cost_span_correction}. If $\mu(F_{i-1}) < \mu(F^*)$, then $Z_i = Z_{i-1}$ and the statement follows from Claim~\ref{claim:inv_min_cost_span_correction_for_a_pair}.

The case when $X_i \ne \ast$ can be proved analogously.
\end{proof}

\begin{claim} \label{claim:inv_min_cost_span_nonneg_of_delta}
If $F^*$ is not a minimum $c_i$-cost member of $\mathcal{F}$, then either $\delta_{i+1} \ge 0$ and equality holds if and only if at the $i$th step belongs to Case~\hyperlink{2.1.1}{2.1.1} or~\hyperlink{4.1.1}{4.1.1}, or Algorithm~\ref{algo:inv_min_cost_span} declares the problem to be infeasible. Furthermore, $\Delta_{i+1} > 0$ if the $i$th step belongs to Case~\hyperlink{2.1.1}{2.1.1}, and $\Delta_{i+1} < 0$ if the $i$th step belongs to Case~\hyperlink{4.1.1}{4.1.1}. In addition, $\delta_0 \ge 0$. 
\end{claim}
\begin{proof}
Clearly, $\delta_0 = \max \{ \lin - \uout, 0 \} \ge 0$. Assume that $F^*$ is not a minimum $c_i$-cost member of $\mathcal{F}$ and that Algorithm~\ref{algo:inv_min_cost_span} does not declare the problem to be infeasible in the $i$th step.  We distinguish the different cases the $i$th step might belong to.
\medskip

\noindent \textbf{Cases~\hyperlink{1.1}{1.1} and~\hyperlink{1.2.1}{1.2.1}.} We have
\[ \delta_{i+1} = f_1(c_i, F_i) = \frac{c_i(F^*) - c_i(F_i)}{\mu(F^* \setminus F_i)} > 0 . \]

\noindent \textbf{Cases~\hyperlink{2.1.1}{2.1.1} and~\hyperlink{4.1.1}{4.1.1}.} Clearly, $\delta_{i+1} = 0$. Furthermore, in Case~\hyperlink{2.1.1}{2.1.1},
\[ \Delta_{i+1} = f_3(c_i, F_i) = \frac{c_i(F^*) - c_i(F_i)}{\mu(F^*) - \mu(F_i)} > 0, \]
and in Case~\hyperlink{4.1.1}{4.1.1},
\[ \Delta_{i+1} = f_3(c_i, F_i) = \frac{c_i(F^*) - c_i(F_i)}{\mu(F^*) - \mu(F_i)} < 0 . \]

\noindent\textbf{Case~\hyperlink{2.1.2.1}{2.1.2.1}.} We have
\[ \delta_{i+1} = f_5(c_i, d_i, D_i, F_i) = \frac{\big( f_3(c_i, F_i) - (\uin - d_i - D_i) \big) \cdot \big( \mu(F^*) - \mu(F_i) \big)}{\mu(F_i \setminus F^*)} > 0 . \]

\noindent \textbf{Case~\hyperlink{2.2.1}{2.2.1}.} We have
\[ \delta_{i+1} = f_6(c_i, D_i, F_i) = \frac{ \big( f_3(c_i, F_i) - (\uout - D_i) \big) \cdot \big( \mu(F^*) - \mu(F_i) \big)}{\mu(F^* \setminus F_i)} > 0 . \]

\noindent \textbf{Case~\hyperlink{2.2.2.1}{2.2.2.1}.} Note that
\begin{align*}
&f_3(c_i, F_i) - (\uin - d_i - D_i) > f_3(c_i, F_i) - \big( f_6(c_i, D_i, F_i) + \uout - D_i \big) \\[2pt]
{}&{}~~= 
f_3(c_i, F_i) - \left( \frac{ \big( f_3(c_i, F_i) - (\uout - D_i) \big) \cdot \big( \mu(F^*) - \mu(F_i) \big)}{\mu(F^* \setminus F_i)} + \uout - D_i \right) \\[2pt]
{}&{}~~= 
f_3(c_i, F_i) \cdot \left( 1 - \frac{\mu(F^*) - \mu(F_i)}{\mu(F^* \setminus F_i)} \right) + (\uout - D_i) \cdot \left( \frac{\mu(F^*) - \mu(F_i)}{\mu(F^* \setminus F_i)} - 1 \right) \\[2pt]
{}&{}~~= 
\big( f_3(c_i, F_i) - (\uout - D_i) \big) \cdot \frac{\mu(F_i \setminus F^*)}{\mu(F^* \setminus F_i)}\\[2pt]
{}&{}~~> 
0 .
\end{align*}
Thus, we have
\[ \delta_{i+1} = f_5(c_i, d_i, D_i, F_i) = \frac{\big( f_3(c_i, F_i) - (\uin - d_i - D_i) \big) \cdot \big( \mu(F^*) - \mu(F_i) \big)}{\mu(F_i \setminus F^*)} > 0 . \]

\noindent \textbf{Case~\hyperlink{3.1.1}{3.1.1}.} By Claim~\ref{claim:inv_min_cost_span_correction_for_a_pair_v2}, we have
\[ \delta_{i+1} = f_7(c_i, F_i, Z_i) = \frac{\hphantom{x} \displaystyle \frac{c_i(F^*) - c_i(F_i)}{\mu(F^*) - \mu(F_i)} - \frac{c_i(F^*) - c_i(Z_i)}{\mu(F^*) - \mu(Z_i)} \hphantom{x}}{\displaystyle \frac{\mu\left( F^* \setminus F_i \right)}{\mu(F^*) - \mu(F_i)} - \frac{\mu\left( F^* \setminus Z_i \right)}{\mu(F^*) - \mu(Z_i)}} > 0 . \]

\noindent \textbf{Case~\hyperlink{3.1.2.1}{3.1.2.1}.} By Claim~\ref{claim:inv_min_cost_span_correction_for_a_pair_v2}, similarly as in Case~\hyperlink{3.1.1}{3.1.1}, $f_7(c_i, F_i, Z_i) > 0$ holds, implying
\begin{align*}
&f_3(c_i, F_i) - (\uin - d_i - D_i)\\[2pt] 
{}&{}~~> 
f_3(c_i, F_i) - \big( f_7(c_i, F_i, Z_i) + f_8(c_i, F_i, Z_i) \big) \\[2pt]
{}&{}~~= 
f_3(c_i, F_i) - f_7(c_i, F_i, Z_i) - \left( f_3(c_i, F_i) - f_7(c_i, F_i, Z_i) \cdot \frac{\mu(F^* \setminus F_i)}{\mu(F^*) - \mu(F_i)} \right) \\[2pt]
{}&{}~~=
f_7(c_i, F_i, Z_i) \cdot \frac{\mu(F_i \setminus F^*)}{\mu(F^*) - \mu(F_i)}\\[2pt]
{}&{}~~> 
0 .
\end{align*}
Thus we have
\[ \delta_{i+1} = f_5(c_i, d_i, D_i, F_i) = \frac{\big( f_3(c_i, F_i) - (\uin - d_i - D_i) \big) \cdot \big( \mu(F^*) - \mu(F_i) \big)}{\mu(F_i \setminus F^*)} > 0 . \]

\noindent \textbf{Case~\hyperlink{3.2.1}{3.2.1}.} By Claim~\ref{claim:inv_min_cost_span_correction_for_a_pair_v2}, similarly as in Case~\hyperlink{3.1.1}{3.1.1}, $f_7(c_i, F_i, Z_i) > 0$ holds. Thus we have
\begin{align*}
&\delta_{i+1}\\[2pt]
{}&{}~~=
f_6(c_i, D_i, F_i) \\[2pt]
{}&{}~~= 
\frac{f_8(c_i, F_i, Z_i) \cdot \big( \mu(F^*) - \mu(F_i) \big) + f_7(c_i, F_i, Z_i) \cdot \mu(F^* \setminus F_i)}{\mu(F^* \setminus F_i)}- \frac{(\uout - D_i) \cdot \big( \mu(F^*) - \mu(F_i) \big)}{\mu(F^* \setminus F_i)} \\[2pt]
{}&{}~~= 
f_7(c_i, F_i, Z_i) + \frac{\left( f_8(c_i, F_i, Z_i) - (\uout - D_i) \right) \cdot \big( \mu(F^*) - \mu(F_i) \big)}{\mu\left( F^* \setminus F_i \right)}\\[2pt]
{}&{}~~> 
0 .
\end{align*}

\noindent \textbf{Case~\hyperlink{3.2.2.1}{3.2.2.1}.} By Claim~\ref{claim:inv_min_cost_span_correction_for_a_pair_v2}, similarly as in Case~\hyperlink{3.1.1}{3.1.1}, $f_7(w_i, F_i, Z_i) > 0$ holds, so
\[ \uout - D_i < f_8(c_i, F_i, Z_i) = f_3(c_i, F_i) - f_7(c_i, F_i, Z_i) \cdot \frac{\mu(F^* \setminus F_i)}{\mu(F^*) - \mu(F_i)} < f_3(c_i, F_i) . \]
Thus, we have
\begin{align*}
&f_3(c_i, F_i) - (\uin - d_i - D_i) \\[2pt]
{}&{}~~>
f_3(c_i, F_i) - \big( f_6(c_i, D_i, F_i) + \uout - D_i \big) \\[2pt]  
{}&{}~~= 
f_3(c_i, F_i) - \frac{\big( f_3(c_i, F_i) - (\uout - D_i) \big) \cdot \big( \mu(F^*) - \mu(F_i) \big)}{\mu(F^* \setminus F_i)} - (\uout - D_i) \\[2pt]
{}&{}~~= 
f_3(c_i, F_i) \cdot \left( 1 - \frac{\mu(F^*) - \mu(F_i)}{\mu(F^* \setminus F_i)} \right)+
(\uout - D_i) \cdot \left( \frac{\mu(F^*) - \mu(F_i)}{\mu(F^* \setminus F_i)} - 1 \right) \\[2pt]
{}&{}~~= 
\big( f_3(c_i, F_i) - (\uout - D_i) \big) \cdot \frac{\mu(F_i \setminus F^*)}{\mu(F^* \setminus F_i)}\\[2pt]
{}&{}~~
> 0 .
\end{align*}
This implies
\[ \delta_{i+1} = f_5(c_i, d_i, D_i, F_i) = \frac{\big( f_3(c_i, F_i) - (\uin - d_i - D_i) \big) \cdot \big( \mu(F^*) - \mu(F_i) \big)}{\mu(F_i \setminus F^*)} > 0 . \]

\noindent \textbf{Case~\hyperlink{4.1.2.1}{4.1.2.1}.} We have
\[ \delta_{i+1} = f_9(c_i, D_i, F_i) = \frac{\big( f_3(c_i, F_i) - (\lout - D_i) \big) \cdot \big( \mu(F^*) - \mu(F_i) \big)}{\mu(F^* \setminus F_i)} > 0 . \]

\noindent \textbf{Case~\hyperlink{4.2.1}{4.2.1}.} We have
\[ \delta_{i+1} = f_{11}(c_i, d_i, D_i, F_i) = \frac{\big( f_3(c_i, F_i) - (\lin - d_i - D_i) \big) \cdot \big( \mu(F^*) - \mu(F_i) \big)}{\mu(F_i \setminus F^*)} > 0 . \]

\noindent \textbf{Case~\hyperlink{4.2.2.1}{4.2.2.1}.} Note that
\begin{align*}
&f_3(c_i, F_i) - (\lout - D_i)\\[2pt] 
{}&{}~~< 
f_3(c_i, F_i) - f_{10}(c_i, d_i, D_i, F_i) \\[2pt]
{}&{}~~= 
f_3(c_i, F_i)- \frac{- f_3(c_i, F_i) \cdot \big( \mu(F^*) - \mu(F_i) \big) + (\lin - d_i - D_i) \cdot \mu(F^* \setminus F_i)}{\mu(F_i \setminus F^*)} \\[2pt]
{}&{}~~= 
\frac{\big( f_3(c_i, F_i) - (\lin - d_i - D_i) \big) \cdot \mu(F^* \setminus F_i)}{\mu(F_i \setminus F^*)}\\[2pt]
{}&{}~~< 
0 .
\end{align*}
Thus, we have
\[ \delta_{i+1} = f_9(c_i, D_i, F_i) = \frac{ - \big( \lout - D_i - f_3(c_i, F_i) \big) \cdot \big( \mu(F^*) - \mu(F_i) \big)}{\mu(F^* \setminus F_i)} > 0 . \]

\noindent \textbf{Case~\hyperlink{5.1.1}{5.1.1}.} By Claim~\ref{claim:inv_min_cost_span_correction_for_a_pair_v2},
\[ \delta_{i+1} = f_7(c_i, X_i, F_i) = \frac{\hphantom{x} \displaystyle \frac{c_i(F^*) - c_i(X_i)}{\mu(F^*) - \mu(X_i)} - \frac{c_i(F^*) - c_i(F_i)}{\mu(F^*) - \mu(F_i)} \hphantom{x}}{\displaystyle \frac{\mu(F^* \setminus X_i)}{\mu(F^*) - \mu(X_i)} - \frac{\mu(F^* \setminus F_i)}{\mu(F^*) - \mu(F_i)}} > 0 . \]

\noindent \textbf{Case~\hyperlink{5.1.2.1}{5.1.2.1}.} By Claim~\ref{claim:inv_min_cost_span_correction_for_a_pair_v2}, similarly as in Case~\hyperlink{5.1.1}{5.1.1}, $f_7(c_i, X_i, F_i) > 0$ holds, so
\[ \lout - D_i > f_{12}(c_i, X_i, F_i) = f_3(c_i, F_i) - f_7(c_i, F_i, Z_i) \cdot \frac{\mu(F^* \setminus F_i)}{\mu(F^*) - \mu(F_i)} > f_3(c_i, F_i) . \]
Thus, we have
\[ \delta_{i+1} = f_9(c_i, D_i, F_i) = \frac{ \big( f_3(c_i, F_i) - (\lout - D_i) \big) \cdot \big( \mu(F^*) - \mu(F_i) \big)}{\mu(F^* \setminus F_i)} > 0 . \]

\noindent \textbf{Case~\hyperlink{5.2.1}{5.2.1}.} By Claim~\ref{claim:inv_min_cost_span_correction_for_a_pair_v2}, similarly as in Case~\hyperlink{5.1.1}{5.1.1}, $f_7(c_i, X_i, F_i) > 0$ holds, so
\begin{align*}
\lin - d_i - D_i 
{}&{}> 
f_7(c_i, X_i, F_i) + f_{12}(c_i, X_i, F_i) \\[2pt]
{}&{}= 
f_7(c_i, X_i, F_i) + f_3(c_i, F_i) - f_7(c_i, X_i, F_i) \cdot \frac{\mu(F^* \setminus F_i)}{\mu(F^*) - \mu(F_i)} \\[2pt]
{}&{}= 
f_3(c_i, F_i) - f_7(c_i, X_i, F_i) \cdot \frac{\mu(F_i \setminus F^*)}{\mu(F^*) - \mu(F_i)} \ge f_3(c_i, F_i) .
\end{align*}
Thus, we have
\[ \delta_{i+1} = f_{11}(c_i, d_i, D_i, F_i) = \frac{\big( f_3(c_i, F_i) - (\lin - d_i - D_i) \big) \cdot \big( \mu(F^*) - \mu(F_i) \big)}{\mu(F_i \setminus F^*)}> 0 . \]

\noindent \textbf{Case~\hyperlink{5.2.2.1}{5.2.2.1}.} By Claim~\ref{claim:inv_min_cost_span_correction_for_a_pair_v2}, similarly as in Case~\hyperlink{5.1.1}{5.1.1}, $f_7(c_i, X_i, F_i) > 0$ holds, and similarly as in Case~\hyperlink{5.2.1}{5.2.1}, $f_3(c_i, F_i) - (\lin - d_i - D_i) < 0$, so
\begin{align*}
&f_3(c_i, F_i) - (\lout - D_i)\\[2pt]
{}&{}~~< 
f_3(c_i, F_i) - f_{10}(c_i, d_i, D_i, F_i) \\[2pt]
{}&{}~~= 
f_3(c_i, F_i)- \frac{- f_3(c_i, F_i) \cdot \big( \mu(F^*) - \mu(F_i) \big) + (\lin - d_i - D_i) \cdot \mu(F^* \setminus F_i)}{\mu(F_i \setminus F^*)} \\[2pt]
{}&{}~~= 
f_3(c_i, F_i) \cdot \left( 1 + \frac{\mu(F^*) - \mu(F_i)}{\mu(F_i \setminus F^*)} \right) - (\lin - d_i - D_i) \cdot \frac{\mu(F^* \setminus F_i)}{\mu(F_i \setminus F^*)} \\[2pt]
{}&{}~~= 
\big( f_3(c_i, F_i) - (\lin - d_i - D_i) \big) \cdot \frac{\mu(F^* \setminus F_i)}{\mu(F_i \setminus F^*)}\\[2pt]
{}&{}~~< 
0 .
\end{align*}
Thus, we have
\[ \delta_{i+1} = f_9(c_i, D_i, F_i) = \frac{ \big( f_3(c_i, F_i) - (\lout - D_i) \big) \cdot \big( \mu(F^*) - \mu(F_i) \big)}{\mu(F^* \setminus F_i)} > 0 . \]
This concludes the proof of the claim.
\end{proof}

\begin{claim} \label{claim:inv_min_cost_span_forgetting}
If $\mu(F_i) < \mu(F^*)$ and $Z_i \ne \ast$, then $\delta_{i+1} \ge f_7(c_i, F_i, Z_i)$. If $\mu(F_i) > \mu(F^*)$ and $X_i \ne \ast$, then $\delta_{i+1} \ge f_7(c_i, X_i, F_i)$.
\end{claim}
\begin{proof}
We distinguish the different cases the $i$th step might belong to.
\medskip

\noindent \textbf{Cases~\hyperlink{3.1.1}{3.1.1} and~\hyperlink{5.1.1}{5.1.1}.} In this case $\delta_{i+1} = f_7(c_i, F_i, Z_i)$ and we are done.
\medskip

\noindent \textbf{Case~\hyperlink{3.1.2.1}{3.1.2.1}.} Note that
\begin{align*}
\uin - d_i - D_i 
{}&{}< 
f_7(c_i, F_i, Z_i) + f_8(c_i, F_i, Z_i) \\[2pt]
{}&{}~~= 
f_7(c_i, F_i, Z_i) \cdot \left( 1 - \frac{\mu(F^* \setminus F_i)}{\mu(F^*) - \mu(F_i)} \right) + \frac{c_i(F^*) - c_i(F_i)}{\mu(F^*) - \mu(F_i)} \\[2pt]
{}&{}~~= 
\frac{c_i(F^*) - c_i(F_i) - f_7(c_i, F_i, Z_i) \cdot \mu(F_i \setminus F^*)}{\mu(F^*) - \mu(F_i)} .
\end{align*}
Thus, we have 
\begin{align*}
\delta_{i+1} 
{}&{}= 
f_5(c_i, d_i, D_i, F_i)\\[2pt] 
{}&{}= 
\frac{c_i(F^*) - c_i(F_i) - (\uin - d_i - D_i) \cdot \big( \mu(F^*) - \mu(F_i) \big)}{\mu(F_i \setminus F^*)} \\[2pt]
{}&{}> 
\frac{c_i(F^*) - c_i(F_i)}{\mu(F_i \setminus F^*)} - \frac{c_i(F^*) - c_i(F_i) - f_7(c_i, F_i, Z_i) \cdot \mu(F_i \setminus F^*)}{\mu(F_i \setminus F^*)} \\[2pt]
{}&{}= 
f_7(c_i, F_i, Z_i).
\end{align*}

\noindent \textbf{Case~\hyperlink{3.2.1}{3.2.1}.} We have
\begin{align*}
\delta_{i+1} 
{}&{}= 
f_6(c_i, D_i, F_i)\\[2pt]
{}&{}= 
\frac{c_i(F^*) - c_i(F_i) - (\uout - D_i) \cdot \big( \mu(F^*) - \mu(F_i) \big)}{\mu(F^* \setminus F_i)} \\[2pt]
{}&{}> 
\frac{c_i(F^*) - c_i(F_i) - f_8(c_i, F_i, Z_i) \cdot \big( \mu(F^*) - \mu(F_i) \big)}{\mu(F^* \setminus F_i)}\\[2pt]
{}&{}= 
f_7(c_i, F_i, Z_i) .
\end{align*}

\noindent \textbf{Case~\hyperlink{3.2.2.1}{3.2.2.1}.} Note that
\begin{align*}
\uin - d_i - D_i
{}&{}< 
f_6(c_i, D_i, F_i) + \uout - D_i \\[2pt]
{}&{}= 
\frac{c_i(F^*) - c_i(F_i) - (\uout - D_i) \cdot \big( \mu(F^*) - \mu(F_i) \big)}{\mu(F^* \setminus F_i)} + (\uout - D_i) \\[2pt]
{}&{}= 
\frac{c_i(F^*) - c_i(F_i)}{\mu(F^* \setminus F_i)} - (\uout - D_i) \cdot \left( \frac{\mu(F^*) - \mu(F_i)}{\mu(F^* \setminus F_i)} - 1 \right) \\[2pt]
{}&{}= 
\frac{c_i(F^*) - c_i(F_i)}{\mu(F^* \setminus F_i)} + (\uout - D_i) \cdot \frac{\mu(F_i \setminus F^*)}{\mu(F^* \setminus F_i)} \\[2pt]
{}&{}< 
\frac{c_i(F^*) - c_i(F_i)}{\mu(F^* \setminus F_i)} + f_8(c_i, F_i, Z_i) \cdot \frac{\mu(F_i \setminus F^*)}{\mu(F^* \setminus F_i)} \\[2pt]
{}&{}=
\frac{c_i(F^*) - c_i(F_i)}{\mu(F^* \setminus F_i)} \cdot \left( 1 + \frac{\mu(F_i \setminus F^*)}{\mu(F^*) - \mu(F_i)} \right)- f_7(c_i, F_i, Z_i) \cdot \frac{\mu(F_i \setminus F^*)}{\mu(F^*) - \mu(F_i)} \\[2pt]
{}&{}= \frac{c_i(F^*) - c_i(F_i) - f_7(c_i, F_i, Z_i) \cdot \mu(F_i \setminus F^*)}{\mu(F^*) - \mu(F_i)} .
\end{align*}
Thus, we have
\begin{align*}
\delta_{i+1} 
{}&{}= 
f_5(c_i, d_i, D_i, F_i)\\[2pt]
{}&{}= 
\frac{c_i(F^*) - c_i(F_i) - (\uin - d_i - D_i) \cdot \big( \mu(F^*) - \mu(F_i) \big)}{\mu(F_i \setminus F^*)} \\[2pt]
{}&{}> 
f_7(c_i, F_i, Z_i) .
\end{align*}

\noindent \textbf{Case~\hyperlink{5.1.2.1}{5.1.2.1}.} We have
\begin{align*}
\delta_{i+1}
{}&{}= 
f_9(c_i, D_i, F_i)\\[2pt]
{}&{}= 
\frac{c_i(F^*) - c_i(F_i) - (\lout - D_i) \cdot \big( \mu(F^*) - \mu(F_i) \big)}{\mu(F^* \setminus F_i)} \\[2pt]
{}&{}> 
\frac{c_i(F^*) - c_i(F_i) - f_{12}(c_i, X_i, F_i) \cdot \big( \mu(F^*) - \mu(F_i) \big)}{\mu(F^* \setminus F_i)}\\[2pt]
{}&{}= f_7(c_i, F_i, Z_i) .
\end{align*}

\noindent \textbf{Case~\hyperlink{5.2.1}{5.2.1}.} Note that
\begin{align*}
\lin - d_i - D_i 
{}&{}> 
f_7(c_i, X_i, F_i) + f_{12}(c_i, X_i, F_i) \\[2pt]
{}&{}= 
f_7(c_i, X_i, F_i) + \frac{c_i(F^*) - c_i(F_i) - f_7(c_i, X_i, F_i) \cdot \mu(F^* \setminus F_i)}{\mu(F^*) - \mu(F_i)} \\[2pt]
{}&{}= 
f_7(c_i, X_i, F_i) \cdot \left( 1 - \frac{\mu(F^* \setminus F_i)}{\mu(F^*) - \mu(F_i)} \right) + \frac{c_i(F^*) - c_i(F_i)}{\mu(F^*) - \mu(F_i)} \\[2pt]
{}&{}= 
\frac{c_i(F^*) - c_i(F_i) - f_7(c_i, X_i, F_i) \cdot \mu(F_i \setminus F^*)}{\mu(F^*) - \mu(F_i)} .
\end{align*}
Thus, we have
\begin{align*}
\delta_{i+1} 
{}&{}= 
f_{11}(c_i, d_i, D_i, F_i)\\[2pt] 
{}&{}= \frac{c_i(F^*) - c_i(F_i) - (\lin - d_i - D_i) \cdot \big( \mu(F^*) - \mu(F_i) \big)}{\mu(F_i \setminus F^*)} \\[2pt]
{}&{}> 
\frac{c_i(F^*) - c_i(F_i)}{\mu(F_i \setminus F^*)}- \frac{c_i(F^*) - c_i(F_i) - f_7(c_i, X_i, F_i) \cdot \mu(F_i \setminus F^*)}{\mu(F^*) - \mu(F_i)} \cdot \frac{\mu(F^*) - \mu(F_i)}{\mu(F_i \setminus F^*)} \\[2pt]
{}&{}= 
f_7(c_i, X_i, F_i) .
\end{align*}

\noindent \textbf{Case~\hyperlink{5.2.2.1}{5.2.2.1}.} Similarly as in Case~\hyperlink{5.2.1}{5.2.1}, we obtain
\begin{align*}
\lin - d_i - D_i 
{}&{}> 
f_7(c_i, X_i, F_i) + f_{12}(c_i, X_i, F_i) \\[2pt]
{}&{}= 
\frac{c_i(F^*) - c_i(F_i) - f_7(c_i, X_i, F_i) \cdot \mu(F_i \setminus F^*)}{\mu(F^*) - \mu(F_i)} .
\end{align*}
Thus, we have
\begin{align*}
\delta_{i+1} 
{}&{}= 
f_9(c_i, D_i, F_i)\\[2pt] 
{}&{}= \frac{c_i(F^*) - c_i(F_i) - (\lout - D_i) \cdot \big( \mu(F^*) - \mu(F_i) \big)}{\mu(F^* \setminus F_i)} \\[2pt]
{}&{}> 
\frac{c_i(F^*) - c_i(F_i) - f_{10}(c_i, d_i, D_i, F_i) \cdot \big( \mu(F^*) - \mu(F_i) \big)}{\mu(F^* \setminus F_i)} \\[2pt]
{}&{}= 
\frac{c_i(F^*) - c_i(F_i)}{\mu(F^* \setminus F_i)}- \frac{c_i(F_i) - c_i(F^*) + (\lin - d_i - D_i) \cdot \mu(F^* \setminus F_i)}{\mu(F_i \setminus F^*)} \cdot \frac{\mu(F^*) - \mu(F_i)}{\mu(F^* \setminus F_i)} \\[2pt]
{}&{}= 
\frac{c_i(F^*) - c_i(F_i)}{\mu(F^* \setminus F_i)} \cdot \left( 1 + \frac{\mu(F^*) - \mu(F_i)}{\mu(F_i \setminus F^*)} \right)- (\lin - d_i - D_i) \cdot \frac{\mu(F^*) - \mu(F_i)}{\mu(F_i \setminus F^*)} \\[2pt]
{}&{}= 
\frac{c_i(F^*) - c_i(F_i)}{\mu(F_i \setminus F^*)} - (\lin - d_i - D_i) \cdot \frac{\mu( F^*) - \mu(F_i)}{\mu(F_i \setminus F^*)} \\[2pt]
{}&{}> 
\frac{c_i(F^*) - c_i(F_i)}{\mu(F_i \setminus F^*)} - \frac{c_i(F^*) - c_i(F_i) - f_7(c_i, X_i, F_i) \cdot \mu(F_i \setminus F^*)}{\mu(F_i \setminus F^*)} \\[2pt]
{}&{}= 
f_7(c_i, X_i, F_i) .
\end{align*}
This concludes the proof of the claim.
\end{proof}

\begin{claim} \label{claim:inv_min_cost_span_notsamesize}
\mbox{}
\begin{enumerate}[label=(\alph*)]\itemsep0em
\item \label{notsame:a} If $X_{i_1+1}, X_{i_2+1} \ne \ast$ and $\mu(X_{i_1+1}) = \mu(X_{i_2+1})$ for some $i_1 < i_2$, then
\begin{equation*}
\mu(X_{i_1+1} \cap F^*) \le \mu(X_{i_2+1} \cap F^*)
\end{equation*}
where equality implies $c(F_{i_1}) = c(F_{i_2})$.
\item \label{notsame:b} If $Z_{i_1+1}, Z_{i_2+1} \ne \ast$ and $\mu(Z_{i_1+1}) = \mu(Z_{i_2+1})$ for some $i_1<i_2$, then
\begin{equation*}
\mu(Z_{i_1+1} \cap F^*) \le \mu(Z_{i_2+1} \cap F^*)
\end{equation*}
where equality implies $c(F_{i_1}) = c(F_{i_2})$.
\end{enumerate}
\end{claim}
\begin{proof}
First, consider the case when $X_{i_1+1}, X_{i_2+1} \ne \ast$. Without loss of generality, we may assume that $X_{i_1+1} = F_{i_1}$ and $X_{i_2+1} = F_{i_2}$. Suppose to the contrary that $\mu(X_{i_1+1} \cap F^*) > \mu(X_{i_2+1} \cap F^*)$.
By Claim~\ref{claim:inv_min_cost_span_nonneg_of_delta},
\begin{align*}
0
{}&{}\le 
c_{i_2}(X_{i_1+1}) - c_{i_2}(F_{i_2}) \\[2pt]
{}&{}= 
c_{i_2}(X_{i_1+1}) - c_{i_2}(X_{i_2+1}) \\[2pt]
{}&{}= 
c_{i_1}(X_{i_1+1}) - c_{i_1}(X_{i_2+1})- (\delta_{i_1+1} + \ldots + \delta_{i_2}) \cdot \big( \mu(X_{i_1+1} \cap F^*) - \mu(X_{i_2+1} \cap F^*) \big) \\[2pt]
{}&{}~~ - (\Delta_{i_1+1} + \ldots + \Delta_{i_2}) \cdot \big( \mu(X_{i_1+1}) - \mu(X_{i_2+1}) \big) \\[2pt]
{}&{}= 
c_{i_1}(F_{i_1}) - c_{i_1}(X_{i_2+1})- (\delta_{i_1+1} + \ldots + \delta_{i_2}) \cdot \big( \mu(X_{i_1+1} \cap F^*) - \mu(X_{i_2+1}  \cap F^*) \big) \\[2pt]
{}&{}~~ - (\Delta_{i_1+1} + \ldots + \Delta_{i_2}) \cdot 0\\[2pt]
{}&{}\le 
0 ,
\end{align*}
which is only possible if $c_{i_2}(X_{i_2+1}) = c_{i_2}(X_{i_1+1})$, $c_{i_1}(X_{i_2+1}) = c_{i_1}(X_{i_1+1})$, and $\delta_{i_1+1} = \ldots = \delta_{i_2} = 0$. This implies that each of the $i_1$th, $(i_1+1)$th, $\dots$, $(i_2-1)$th steps belong to Case~\hyperlink{2.1.1}{2.1.1}, thus $\Delta_{i_1+1}, \ldots, \Delta_{i_2} > 0$. Note that $F^*$ is not a minimum $c_{i_2}$-cost member of $\mathcal{F}$ since $X_{i_2+1} \ne \ast$. Hence, by Lemma~\ref{lemma:inv_min_cost_span_correction},
\begin{align*}
0
{}&{}< 
c_{i_2}(F^*) - c_{i_2}(F_{i_2})\\[2pt] 
{}&{}= 
c_{i_2}(F^*) - c_{i_2}(X_{i_2+1}) \\[2pt] 
{}&{}=
c_{i_2}(F^*) - c_{i_2}(X_{i_1+1}) \\[2pt]
{}&{}= 
c_{i_1+1}(F^*) - c_{i_1+1}(X_{i_1+1}) - (\delta_{i_1+2} + \ldots + \delta_{i_2}) \cdot \big( \mu(F^*) - \mu( X_{i_1+1} \cap F^*) \big) \\[2pt]
{}&{}~~ - (\Delta_{i_1+2} + \ldots + \Delta_{i_2}) \cdot \big( \mu(F^*) - \mu(X_{i_2+1}) \big) \\[2pt]
{}&{}= 
c_{i_1+1}(F^*) - c_{i_1+1}(F_{i_1}) - 0 \cdot \big( \mu(F^*) - \mu( X_{i_1+1} \cap F^*) \big) \\[2pt]
{}&{}~~ - (\Delta_{i_1+2} + \ldots + \Delta_{i_2}) \cdot \big( \mu(F^*) - \mu(X_{i_2+1}) \big)\\[2pt]
{}&{}< 
0 ,
\end{align*}
a contradiction.

Now assume $\mu(X_{i_1+1} \cap F^*) = \mu(X_{i_2+1}  \cap F^*)$. Then
\begin{align*}
0 
{}&{}\le 
c_{i_2}(X_{i_1+1}) - c_{i_2}(F_{i_2}) \\[2pt] 
{}&{}=
c_{i_2}(X_{i_1+1}) - c_{i_2}(X_{i_2+1}) \\[2pt]
{}&{}= 
c_{i_1}(X_{i_1+1}) - c_{i_1}(X_{i_2+1})- (\delta_{i_1+1} + \ldots + \delta_{i_2}) \cdot \big( \mu(X_{i_1+1} \cap F^*) - \mu(X_{i_2+1} \cap F^*) \big) \\[2pt]
{}&{}~~ - (\Delta_{i_1+1} + \ldots + \Delta_{i_2}) \cdot \big( \mu(X_{i_1+1}) - \mu(X_{i_2+1}) \big) \\[2pt]
{}&{}= 
c_{i_1}(X_{i_1+1}) - c_{i_1}(X_{i_2+1}) - 0 \cdot \big( \mu(X_{i_1+1} \cap F^*) - \mu(X_{i_2+1} \cap F^*) \big) \\[2pt]
{}&{}~~ - (\Delta_{i_1+1} + \ldots + \Delta_{i_2}) \cdot 0 \\[2pt]
{}&{}\le 
0.
\end{align*}
Thus, we have
\begin{align*}
0 
{}&{}= 
c_{i_2}(X_{i_1+1}) - c_{i_2}(X_{i_2+1}) \\[2pt]
{}&{}= 
c(X_{i_1+1}) - c(X_{i_2+1}) - d_{i_2} \! \cdot \! \big( \mu(X_{i_1+1} \cap F^*) - \mu(X_{i_2+1}  \cap F^*) \! \big) - D_{i_2} \! \cdot \! \big( \mu(X_{i_1+1}) - \mu(X_{i_2+1}) \big)\\[2pt]
{}&{}= 
c(X_{i_1+1}) - c(X_{i_2+1}) - d_{i_2} \cdot 0 - D_{i_2} \cdot 0\\[2pt]
{}&{}= 
c(X_{i_1+1}) - c(X_{i_2+1}).
\end{align*}

The second statement of the lemma can be proved analogously.
\end{proof}

\begin{claim} \label{claim:inv_min_cost_span_pairs_corrected_once}
 Let $F', F''' \in \mathcal{F}$ be such that $\mu(F') < \mu(F^*) < \mu(F''')$, and let $\delta, \Delta \in \mathbb{R}$ be such that $\delta \ge f_7(c, F', F''')$.
Then neither
\begin{equation*}
(c - \p{\delta, \Delta}{\ell, u}{w})(F') < (c - \p{\delta, \Delta}{\ell, u}{w})(F^*) = (c - \p{\delta, \Delta}{\ell, u}{w})(F''')
\end{equation*}
nor
\begin{equation*}
(c - \p{\delta, \Delta}{\ell, u}{w})(F') = (c - \p{\delta, \Delta}{\ell, u}{w})(F^*) > (c - \p{\delta, \Delta}{\ell, u}{w})(F''')
\end{equation*}
hold.
\end{claim}
\begin{proof}
First, suppose to the contrary that 
\[ (c - \p{\delta, \Delta}{\ell, u}{w})(F') < (c - \p{\delta, \Delta}{\ell, u}{w})(F^*) = (c - \p{\delta, \Delta}{\ell, u}{w})(F'''). \]
Then
\begin{align*}
0 
{}&{}< 
(c - \p{\delta, \Delta}{\ell, u}{w})(F^*) - (c - \p{\delta, \Delta}{\ell, u}{w})(F') \\[2pt]
{}&{}= 
c(F^*) - c(F') - \delta \cdot \mu(F^*\setminus F') - \Delta \cdot \big( \mu(F^*) - \mu(F') \big),
\end{align*}
and
\begin{align*}
0 
{}&{}= 
(c - \p{\delta, \Delta}{\ell, u}{w})(F^*) - (c - \p{\delta, \Delta}{\ell, u}{w})(F''') \\[2pt]
{}&{}= 
c(F^*) - c(F''') - \delta \cdot \mu(F^*\setminus F''') - \Delta \cdot \big( \mu(F^*) - \mu(F''') \big).
\end{align*}
Therefore, we get
\begin{align*}
\frac{c(F^*) - c(F''') - \delta \cdot \mu(F^*\setminus F''')}{\mu(F^*) - \mu(F''')} &= \Delta,\ \text{and} \\[2pt]
\frac{c(F^*) - c(F') - \delta \cdot \mu(F^*\setminus F')}{\mu(F^*) - \mu(F')} &>\Delta.
\end{align*}
Using this, we get
\begin{align*}
\delta 
< 
\frac{\hphantom{x} \displaystyle \frac{c(F^*) - c(F')}{\mu(F^*) - \mu(F')} - \frac{c(F^*) - c(F''')}{\mu(F^*) - \mu(F''')} \hphantom{x}}{\displaystyle \frac{\mu(F^*\setminus F')}{\mu(F^*) - \mu(F')} - \frac{\mu(F^*\setminus F''')}{\mu(F^*) - \mu(F''')}}= 
f_7(c, F', F''') ,
\end{align*}
a contradiction.

The second half of the lemma can be proved analogously.
\end{proof}

%%%%%%%%%%%%%%%%
\subsection{Proof of Lemma~\ref{lemma:step1}}
%%%%%%%%%%%%%%%%

\begin{proof}[Proof of Lemma~\ref{lemma:step1}]
Without loss of generality, we may assume $Y_{i_1+1} = F_{i_1}$ and $Y_{i_2+1} = F_{i_2}$.

First, we show $c_{i_2}(F^*) \le c_{i_2}(Y_{i_1+1})$. By Lemma~\ref{lemma:inv_min_cost_span_correction} and Claim~\ref{claim:inv_min_cost_span_nonneg_of_delta},
\begin{align*}
c_{i_2}(F^*) - c_{i_2}(Y_{i_1+1})
{}&{}= 
c_{i_1+1}(F^*) - c_{i_1+1}(Y_{i_1+1}) - (\delta_{i_1+2} + \ldots + \delta_{i_2}) \cdot \mu(F^*\setminus Y_{i_1+1}) \\[2pt]
{}&{}~~- (\Delta_{i_1+2} + \ldots + \Delta_{i_2}) \cdot \big( \mu(F^*) - \mu(Y_{i_1+1}) \big) \\[2pt]
{}&{}= 
0 - (\delta_{i_1+2} + \ldots + \delta_{i_2}) \cdot \mu(F^* \setminus Y_{i_1+1}) - (\Delta_{i_1+2} + \ldots + \Delta_{i_2}) \cdot 0\\[2pt]
{}&{}\le
0 .
\end{align*}

Now suppose to the contrary that $\mu( Y_{i_1+1} \cap F^*) \ge \mu( Y_{i_2+1} \cap F^*)$. Then by the above observation and by Claim~\ref{claim:inv_min_cost_span_nonneg_of_delta}, we get
\begin{align*}
0 
{}&{}< 
c_{i_2}(F^*) - c_{i_2}(F_{i_2}) \\[2pt]
{}&{}= 
c_{i_2}(F^*) - c_{i_2}(Y_{i_2+1}) \\[2pt]
{}&{}\le 
c_{i_2}(Y_{i_1+1}) - c_{i_2}(Y_{i_2+1}) \\[2pt]
{}&{}= 
c_{i_1}(Y_{i_1+1}) - c_{i_1}(Y_{i_2+1})- (\delta_{i_1+1} + \ldots + \delta_{i_2}) \cdot \big( \mu( Y_{i_1+1} \cap F^*) - \mu(Y_{i_2+1} \cap F^*) \big) \\[2pt]
{}&{}~~ - (\Delta_{i_1+1} + \ldots + \Delta_{i_2}) \cdot \big( \mu(Y_{i_1+1}) - \mu(Y_{i_2+1}) \big) \\[2pt]
{}&{}= 
c_{i_1}(F_{i_1}) - c_{i_1}(Y_{i_2+1})- (\delta_{i_1+1} + \ldots + \delta_{i_2}) \cdot \big( \mu( Y_{i_1+1} \cap F^*) - \mu(Y_{i_2+1} \cap F^*) \big) \\[2pt]
{}&{}~~ - \left( \Delta_{i_1+1} + \ldots + \Delta_{i_2} \right) \cdot 0\\[2pt]
{}&{}\le 
0,
\end{align*}
a contradiction.
\end{proof}

%%%%%%%%%%%%%%%%
\subsection{Proof of Lemma~\ref{lemma:step2}}
%%%%%%%%%%%%%%%%

\begin{proof}[Proof of Lemma~\ref{lemma:step2}]
 We prove \ref{lobab:a}, the proof of \ref{lobab:b} is analogous. Suppose to the contrary that there exist two indices $i_1$ and $i_2$ with $i_1 < i_2$ satisfying the conditions of part \ref{lobab:a} of the lemma. Then, by Claim~\ref{claim:inv_min_cost_span_notsamesize}, we have $c(F_{i_1}) = c(F_{i_2})$. Moreover, since $\big (\mu(F_{i_1}),\mu(Z_{i_1}),\mu(F_{i_1}\cap F^*),\mu(Z_{i_1} \cap F^*)\big)=\big (\mu(F_{i_2}),\mu(Z_{i_2}),\mu(F_{i_2}\cap F^*),\mu(Z_{i_2} \cap F^*)\big)$, it also follows that $c_i(F_{i_1}) = c_i(F_{i_2})$ for any $i$ by the assumption that \hyperlink{spec-lu}{(SPEC-LU)} holds. This implies $f_7(c_i, F_{i_1}, Z_{i_1}) = f_7(c_i, F_{i_2}, Z_{i_2})$ for any~$i$. By Claims~\ref{claim:inv_min_cost_span_nonneg_of_delta} and~\ref{claim:inv_min_cost_span_forgetting}, we get 
 \[ d_{i_2} = \delta_0 + \ldots + \delta_{i_2} \ge \delta_{i_1+1} \ge f_7(c_{i_1},F_{i_1}, Z_{i_1}) = f_7(c_{i_1},F_{i_2}, Z_{i_2}). \]
 In addition, by Claim~\ref{claim:inv_min_cost_span_correction_for_a_pair_v2}, we have $c_{i_2}(F^*) = c_{i_2}(Z_{i_2})$. Thus, applying Claim~\ref{claim:inv_min_cost_span_pairs_corrected_once} for $F'\coloneqq F_{i_2}$, $F'''\coloneqq Z_{i_2}$, $\delta\coloneqq d_{i_2}$ and $\Delta\coloneqq D_{i_2}$, we get that $c_{i_2}(F_{i_2}) < c_{i_2}(F^*) = c_{i_2}(Z_{i_2})$ cannot occur, contradicting the fact that $F_{i_2}$ is a minimum $c_{i_2}$-cost member of $\mathcal{F}$ while $F^*$ is not.
\end{proof}

%%%%%%%%%%%%%%%%
\subsection{Proof of Lemma~\ref{lemma:step3}}
%%%%%%%%%%%%%%%%

\begin{proof}[Proof of Lemma~\ref{lemma:step3}]
Suppose to the contrary that $\mu(F_i) = \mu(F_{i+1})$ and $\mu(F_i \cap F^*) \ge \mu(F_{i+1} \cap F^*)$. Then by Claim~\ref{claim:inv_min_cost_span_nonneg_of_delta}, we get
\begin{align*}
0 
{}&{}\le 
c_{i+1}(F_i) - c_{i+1}(F_{i+1}) \\[2pt]
{}&{}= 
c_i(F_i) - c_i(F_{i+1}) - \delta_{i+1} \cdot \big( \mu(F_i \cap F^*) - \mu(F_{i+1} \cap F^*) \big) - \Delta_{i+1} \cdot \big( \mu(F_i) - \mu(F_{i+1}) \big) \\[2pt]
{}&{}= 
c_i(F_i) - c_i(F_{i+1}) - \delta_{i+1} \cdot \big( \mu(F_i \cap F^*) - \mu(F_{i+1} \cap F^*) \big) - \Delta_{i+1} \cdot 0 \\[2pt]
{}&{}< 
0,
\end{align*}
a contradiction.
\end{proof}

%%%%%%%%%%%%%%%%
\subsection{Proof of Lemma~\ref{lemma:inv_min_cost_span_opt_dev_vector_v2}}
%%%%%%%%%%%%%%%%

\begin{proof}[Proof of Lemma~\ref{lemma:inv_min_cost_span_opt_dev_vector_v2}]
Recall that we assumed \hyperlink{spec-lu}{(SPEC-LU)} to hold. By Corollary~\ref{cor:inv_min_cost_span_opt_dev_vector}, there exist $\delta', \Delta' \in \mathbb{R}$ with $\delta' + \Delta' \ge \lin$ and $\Delta' \le \uout$ for which $\p{\delta', \Delta'}{\ell, u}{w}$ is an optimal deviation vector. 

If $\delta' + \Delta' \le \uin$ and $\Delta' \ge \lout$, then set $\delta\coloneqq\delta'$ and $\Delta\coloneqq\Delta'$. If $\delta' + \Delta' \le \uin$ and $\Delta' < \lout$, then set $\delta \coloneqq  \delta' + \Delta' - \lout$ and $\Delta \coloneqq  \lout$. If $\delta' + \Delta' > \uin$ and $\Delta' \ge \lout$, then set $\delta \coloneqq  \uin - \Delta'$ and $\Delta \coloneqq  \Delta'$. Finally,  if $\delta' + \Delta' > \uin$ and $\Delta' < \lout$, then set $\delta \coloneqq  \uin - \lout$ and $\Delta \coloneqq  \lout$. It is not difficult to check that in all cases, we get that $\p{\delta, \Delta}{\ell, u}{w} = \p{\delta', \Delta'}{\ell, u}{w}$, $\lin \le \delta + \Delta \le \uin$ and $\lout \le \Delta \le \uout$ hold.
\end{proof}

%%%%%%%%%%%%%%%%
\subsection{Proof of Lemma~\ref{lemma:inv_min_cost_span_deltas_stay_in_the_intervals}}
%%%%%%%%%%%%%%%%

\begin{proof}[Proof of Lemma~\ref{lemma:inv_min_cost_span_deltas_stay_in_the_intervals}]
Recall that we assumed \hyperlink{spec-lu}{(SPEC-LU)} to hold. We prove the lemma by induction on $i$. 

First we show that the statements hold for $i=0$. If $\lin \ne - \infty$ and $\uout \ne + \infty$, then there are two cases. If $\uout \le \lin$, then $d_0 = \delta_0 = \lin - \uout$ and $D_0 = \Delta_0 = \uout$, thus $\lin = d_0 + D_0 \le \uin$ and $\lout \le D_0 = \uout$. If $\uout \ge \lin$, then $d_0 = \delta_0 = 0$ and $D_0 = \Delta_0 = \uout$, thus $\lin \le d_0 + D_0 = \uout \le \uin$ and $\lout \le D_0 = \uout$.  If $\lin = - \infty$ and $\uout \ne + \infty$, then $d_0 = \delta_0 = 0$ and $D_0 = \Delta_0 = \uout$, thus $\lin \le d_0 + D_0 = \uout \le \uin$ and $\lout \le D_0 = \uout$. If $\lin \ne - \infty$ and $\uout = + \infty$, then $d_0 = \delta_0 = 0$ and $D_0 = \Delta_0 = \lin$, thus $\lin = d_0 + D_0 \le \uin$ and $\lout \le \lin = D_0 \le \uout$. If $\lin = - \infty$ and $\uout = + \infty$, then $d_0 = \delta_0 = 0$ and $D_0 = \Delta_0 = 0$, thus $\lin \le d_0 + D_0 \le \uout \le \uin$ and $\lout \le \lin \le D_0 \le \uout$. 

Now let $i$ be arbitrary and assume that $\lin \le d_i + D_i \le \uin$ and $\lout \le D_i \le \uout$ and that Algorithm~\ref{algo:inv_min_cost_span} does not declare the problem to be infeasible in the $i$th step. We distinguish the different cases the $i$th step might belong to.
\medskip

\noindent \textbf{Case~\hyperlink{1.1}{1.1}.} We have $d_{i+1} + D_i = d_i + \delta_{i+1} + D_i = d_i + f_1(c_i, F_i) + D_i \le \uin$ and $\Delta_{i+1} = 0$, hence we are done by  Claim~\ref{claim:inv_min_cost_span_nonneg_of_delta} and by the induction hypothesis.
\medskip

\noindent \textbf{Case~\hyperlink{1.2.1}{1.2.1}.} We have $d_{i+1} + D_i = d_i + \delta_{i+1} + D_i = d_i + f_1(c_i, F_i) + D_i > \uin$ and $D_{i+1} \ge \lout$. Then $\Delta_{i+1} = f_2(c_i, d_i, D_i, F_i) = \uin - d_i - D_i - f_1(c_i, F_i) = \uin - d_{i+1} - D_i < 0$. Thus we obtain $d_{i+1} + D_{i+1} = d_{i+1} + D_i + \Delta_{i+1} = \uin$ and $D_{i+1} = D_i + \Delta_{i+1} < D_i$, hence we are done by the induction hypothesis.
\medskip

\noindent \textbf{Case~\hyperlink{2.1.1}{2.1.1}.} We have $D_{i+1} = D_i + \Delta_{i+1} = D_i + f_3(c_i, F_i) \le \uout$ and $d_i + D_{i+1} = d_i + D_i + \Delta_{i+1} = d_i + D_i + f_3(c_i, F_i) \le \uin$.
Since $\delta_{i+1} = 0$ and
\[ \Delta_{i+1} = f_3(c_i, F_i) = \frac{c_i(F^*) - c_i(F_i)}{\mu(F^*) - \mu(F_i)} > 0, \]
we are done by the induction hypothesis.
\medskip

\noindent \textbf{Case~\hyperlink{2.1.2.1}{2.1.2.1}.} We have $D_{i+1} = D_i + \Delta_{i+1} = D_i + f_4(c_i, d_i, D_i, F_i) \ge \lout$ and
\begin{align*}
d_{i+1} + D_{i+1} 
{}&{}= 
d_i + \delta_{i+1} + D_i + \Delta_{i+1}\\[2pt]
{}&{}= 
d_i + f_5(c_i, d_i, D_i, F_i) + D_i + f_4(c_i, d_i, D_i, F_i) \\[2pt]
{}&{}= 
d_i + D_i + (\uin - d_i - D_i) \cdot \frac{\mu(F^* \setminus F_i) - \big( \mu(F^*) - \mu(F_i) \big)}{\mu(F_i \setminus F^*)}\\[2pt]
{}&{}= 
\uin .
\end{align*}
In addition,
\begin{align*}
\Delta_{i+1} 
{}&{}= 
f_4(c_i, d_i, D_i, F_i) \\[2pt]
{}&{}= 
\frac{- f_3(c_i, F_i) \cdot \big( \mu(F^*) - \mu(F_i) \big) + \big( \uin - d_i - D_i \big) \cdot \mu(F^* \setminus F_i)}{\mu(F_i \setminus F^*)} \\[2pt]
{}&{}= 
f_3(c_i, F_i) - \frac{\big( f_3(c_i, F_i) - (\uin - d_i - D_i) \big) \cdot \mu(F\setminus F_i)}{\mu(F_i \setminus F)}\\[2pt]
{}&{}< 
f_3(c_i, F_i) ,
\end{align*}
so $D_{i+1} = D_i + \Delta_{i+1} < D_i + f_3(w_i, F_i) \le \uout$.
\medskip

\noindent \textbf{Case~\hyperlink{2.2.1}{2.2.1}.} We have $D_{i+1} = D_i + \Delta_{i+1} = D_i + (\uout - D_i) = \uout$, and thus $d_{i+1} + D_{i+1} = d_i + \delta_{i+1} + \uout = d_i + f_6(c_i, d_i, F_i) + \uout \le \uin$. Therefore, we are done by Claim~\ref{claim:inv_min_cost_span_nonneg_of_delta} and by the induction hypothesis.
\medskip

\noindent \textbf{Cases~\hyperlink{2.2.2.1}{2.2.2.1} and~\hyperlink{3.2.2.1}{3.2.2.1}.} We have $D_{i+1} = D_i + \Delta_{i+1} = D_i + f_4(c_i, d_i, D_i, F_i) \ge \lout$.
Also, note that
\begin{align*}
\uin - d_i - D_i 
{}&{}< f_6(w_i, d_i, F_i) + \uout - D_i \\[2pt]
{}&{}= 
\frac{c_i(F^*) - c_i(F_i) - (\uout - D_i) \cdot \big( \mu(F^*) - \mu(F_i) \big)}{\mu(F^* \setminus F_i)} + \uout - D_i \\[2pt]
{}&{}= 
\frac{c_i(F^*) - c_i(F_i)}{\mu(F^* \setminus F_i)} + (\uout - D_i) \left( - \, \frac{\mu(F^*) - \mu(F_i)}{\mu(F^* \setminus F_i)} + 1 \right) \\[2pt]
{}&{}=
\frac{c_i(F^*) - c_i(F_i) + (\uout - D_i) \cdot \mu\left( F_i \setminus F^* \right)}{\mu\left( F^* \setminus F_i \right)} .
\end{align*}
Thus, we have
\begin{align*}
D_{i+1} 
{}&{}= 
D_i + \Delta_{i+1}\\[2pt]
{}&{}= 
D_i + \frac{c_i(F_i) - c_i(F^*) + (\uin - d_i - D_i) \cdot \mu(F^* \setminus F_i)}{\mu(F_i \setminus F^*)} \\[2pt]
{}&{}< 
D_i + \frac{c_i(F_i) - c_i(F^*)}{\mu(F_i \setminus F^*)}+ \frac{c_i(F^*) - c_i(F_i) + (\uout - D_i) \cdot \mu(F_i \setminus F^*)}{\mu(F_i \setminus F^*)} \\
{}&{}= 
\uout,
\end{align*}
and
\begin{align*}
d_{i+1} + D_{i+1}
{}&{}=
d_i + \delta_{i+1} + D_i + \Delta_{i+1}\\[2pt] 
{}&{}= d_i + f_5(c_i, d_i, D_i, F_i) + D_i + f_4(c_i, d_i, D_i, F_i) \\[2pt]
{}&{}= 
d_i + \frac{c_i(F^*) - c_i(F_i) - (\uin - d_i - D_i) \cdot \big( \mu(F^*) - \mu(F_i) \big)}{\mu(F_i \setminus F^*)} \\[2pt]
{}&{}~~~~+ D_i + \frac{c_i(F_i) - c_i(F^*) + (\uin - d_i - D_i) \cdot \mu(F^* \setminus F_i)}{\mu(F_i \setminus F^*)} \\[2pt]
{}&{}= 
d_i + D_i + (\uin - d_i - D_i) \cdot \frac{- \big( \mu(F^*) - \mu(F_i) \big) + \mu(F^* \setminus F_i)}{\mu(F_i \setminus F^*)}\\[2pt]
{}&{}= 
\uin .
\end{align*}

\noindent \textbf{Case~\hyperlink{3.1.1}{3.1.1}.} We have $D_{i+1} = D_i + \Delta_{i+1} = D_i + f_8(c_i, F_i, Z_i) \le \uout$ and $d_{i+1} + D_{i+1} = d_i + \delta_{i+1} + D_i + \Delta_{i+1} = d_i + f_7(c_i, F_i, Z_i) + D_i + f_8(c_i, F_i, Z_i) \le \uin$. By Claim~\ref{claim:inv_min_cost_span_reformulating_Delta_iplusone} and Claim~\ref{claim:inv_min_cost_span_correction_for_a_pair_v2},
\[ \Delta_{i+1} = \frac{- f_7(c_i, F_i, Z_i) \cdot \mu(F^* \setminus Z_i)}{\mu(F^*) - \mu(Z_i)} \ge 0 , \]
so we are done by Claim~\ref{claim:inv_min_cost_span_nonneg_of_delta} and by the induction hypothesis.
\medskip

\noindent \textbf{Case~\hyperlink{3.1.2.1}{3.1.2.1}.} We have $D_{i+1} = D_i + \Delta_{i+1} = D_i + f_4(c_i, d_i, D_i, F_i) \ge \lout$ and
\begin{align*}
d_{i+1} + D_{i+1}
{}&{}= 
d_i + \delta_{i+1} + D_i + \Delta_{i+1}\\[2pt] 
{}&{}= d_i + f_5(c_i, d_i, D_i, F_i) + D_i + f_4(c_i, d_i, D_i, F_i) \\[2pt]
{}&{}= 
d_i + \frac{c_i(F^*) - c_i(F_i) - (\uin - d_i - D_i) \cdot \big( \mu(F^*) - \mu(F_i) \big)}{\mu(F_i \setminus F^*)}\\[2pt] 
{}&{}~~ + D_i+ \frac{c_i(F_i) - c_i(F^*) + (\uin - d_i - D_i) \cdot \mu(F^* \setminus F_i)}{\mu(F_i \setminus F^*)} \\[2pt]
{}&{}= 
d_i + D_i + (\uin - d_i - D_i) \cdot \frac{- \big( \mu(F^*) - \mu(F_i) \big) + \mu(F^* \setminus F_i)}{\mu(F_i \setminus F^*)}\\[2pt]
{}&{}= 
\uin .
\end{align*}
In addition,
\begin{align*}
\Delta_{i+1} 
{}&{}= 
f_4(c_i, d_i, D_i, F_i)\\[2pt]
{}&{}= 
\frac{c_i(F_i) - c_i(F^*) + (\uin - d_i - D_i) \cdot \mu(F^* \setminus F_i)}{\mu(F_i \setminus F^*)} \\[2pt]
{}&{}= 
\frac{- f_3(c_i, F_i) \cdot \big( \mu(F^*) - \mu(F_i) \big) + (\uin - d_i - D_i) \cdot \mu(F^* \setminus F_i)}{\mu(F_i \setminus F^*)} \\[2pt]
{}&{}< 
\frac{- f_3(c_i, F_i) \cdot \big( \mu(F^*) - \mu(F_i) \big)}{\mu(F_i \setminus F^*)} +\frac{\big( f_7(c_i, F_i, Z_i) + f_8(c_i, F_i, Z_i) \big)\cdot \mu(F^* \setminus F_i)}{\mu(F_i \setminus F^*)}\\[2pt]
{}&{}= 
f_3(c_i, F_i) + \frac{\big( f_7(c_i, F_i, Z_i) + f_8(c_i, F_i, Z_i) - f_3(c_i, F_i) \big) \cdot \mu(F^* \setminus F_i)}{\mu(F_i \setminus F^*)} \\[2pt]
{}&{}= 
%f_3(c_i, F_i) \\[2pt]
%{}&{}~~+ \left( f_7(c_i, F_i, Z_i) + \left( f_3(c_i, F_i) - f_7(c_i, F_i, Z_i) \cdot \frac{\mu\left( F^* \setminus F_i\right)}{\mu(F^*) - \mu(F_i\setminus S_0)} \right) - f_3(c_i, F_i) \right) \\[2pt]
%{}&{}~~\cdot \frac{\mu\left( F^* \setminus F_i \right)}{\mu\left( F_i \setminus F^*\right)} \\[2pt]
%{}&{}= 
f_3(c_i, F_i) + f_7(c_i, F_i, Z_i) \cdot \left( 1 - \frac{\mu(F^* \setminus F_i)}{\mu(F^*) - \mu(F_i)} \right) \cdot \frac{\mu(F^* \setminus F_i)}{\mu(F_i \setminus F^*)} \\[2pt]
{}&{}= 
f_3(c_i, F_i) + f_7(c_i, F_i, Z_i) \cdot \frac{\mu(F^* \setminus F_i)}{\mu(F^*) - \mu(F_i)}\\[2pt]
{}&{}= 
f_8(c_i, F_i, Z_i).
\end{align*}
Thus we get $D_{i+1} = D_i + \Delta_{i+1} < D_i + f_8(c_i, F_i, Z_i) \le \uout$.
\medskip

\noindent \textbf{Case~\hyperlink{3.2.1}{3.2.1}.} We have $D_{i+1} = D_i + \Delta_{i+1} = D_i + (\uout - D_i) = \uout$, and so $d_{i+1} + D_{i+1} = d_i + \delta_{i+1} + \uout = d_i + f_6(c_i, D_i, F_i) + \uout \le \uin$. In addition, by the induction hypothesis, we obtain $\Delta_{i+1} = \uout - D_i \ge 0$. Thus we are done by Claim~\ref{claim:inv_min_cost_span_nonneg_of_delta} and by the induction hypothesis.
\medskip

\noindent \textbf{Case~\hyperlink{4.1.1}{4.1.1}.} We have $d_{i+1} + D_{i+1} = d_i + \delta_{i+1} + D_i + \Delta_{i+1} = d_i + 0 + D_i + f_3(c_i, F_i) \ge \lin$ and $D_{i+1} = D_i + \Delta_{i+1} = D_i + f_3(c_i, F_i) \ge \lout$. Since $\delta_{i+1} = 0$ and 
\[ \Delta_{i+1} = f_3(c_i, F_i) = \frac{c_i(F^*) - c_i(F_i)}{\mu(F^*) - \mu(F_i)} < 0 , \]
we are done by the induction hypothesis.
\medskip

\noindent \textbf{Case~\hyperlink{4.1.2.1}{4.1.2.1}.} We have $D_{i+1} = D_i + \Delta_{i+1} = D_i + (\lout - D_i) = \lout$ and $d_{i+1} + D_{i+1} = d_i + \delta_{i+1} + \lout = d_i + f_9(c_i, D_i, F_i) + \lout \le \uin$. In addition,
\begin{align*}
d_{i+1} + D_{i+1}
{}&{}= 
d_i + \delta_{i+1} + D_i + \Delta_{i+1} = d_i + f_9(c_i, D_i, F_i) + D_i + (\lout - D_i) \\
{}&{}= 
d_i + \frac{c_i(F^*) - c_i(F_i) - (\lout - D_i) \cdot \big( \mu(F^*) - \mu(F_i) \big)}{\mu(F^* \setminus F_i)} + D_i + (\lout - D_i) \\[2pt]
{}&{}= 
d_i + D_i + f_3(c_i, F_i) \cdot \frac{\mu(F^*) - \mu(F_i)}{\mu(F^* \setminus F_i)}+ (\lout - D_i) \cdot \left( - \frac{\mu(F^*) - \mu(F_i)}{\mu(F^* \setminus F_i)} + 1 \right) \\[2pt]
{}&{}= 
d_i + D_i + f_3(c_i, F_i) \cdot \frac{\mu(F^*) - \mu(F_i)}{\mu(F^* \setminus F_i)} + (\lout - D_i) \cdot \frac{\mu(F_i \setminus F^*)}{\mu(F^* \setminus F_i)} \\[2pt]
{}&{}>
d_i + D_i + f_3(c_i, F_i) \cdot \frac{\mu(F^*) - \mu(F_i)}{\mu(F^* \setminus F_i)} + f_3(c_i, F_i) \cdot \frac{\mu(F_i \setminus F^*)}{\mu(F^* \setminus F_i)} \\[2pt]
{}&{}= 
d_i + D_i + f_3(c_i, F_i) \cdot \left( \frac{\mu(F^*) - \mu(F_i)}{\mu(F^* \setminus F_i)} + \frac{\mu(F_i \setminus F^*)}{\mu(F^* \setminus F_i)} \right) \\[2pt]
{}&{}= 
d_i + D_i + f_3(c_i, F_i)\\[2pt]
{}&{}\ge 
\lin .
\end{align*}

\noindent \textbf{Case~\hyperlink{4.2.1}{4.2.1}.} We have $D_{i+1} = D_i + \Delta_{i+1} = D_i + f_{10}(c_i, d_i, D_i, F_i) \ge \lout$ and
\begin{align*}
d_{i+1} + D_{i+1}
{}&{}= 
d_i + \delta_{i+1} + D_i + \Delta_{i+1}\\[2pt]
{}&{}= 
d_i + f_{11}(c_i, d_i, D_i, F_i) + D_i + f_{10}(c_i, d_i, D_i, F_i) \\[2pt]
{}&{}= 
d_i + \frac{c_i(F^*) - c_i(F_i) - (\lin - d_i - D_i) \cdot \big( \mu(F^*) - \mu(F_i) \big)}{\mu(F_i \setminus F^*)} \\[2pt]
{}&{}~~+ D_i+ \frac{c_i(F_i) - c_i(F^*) + (\lin - d_i - D_i) \cdot \mu(F^* \setminus F_i)}{\mu(F_i \setminus F^*)} \\[2pt]  
{}&{}= 
d_i + D_i + (\lin - d_i - D_i) \cdot \frac{ - \big( \mu(F^*) - \mu(F_i) \big) + \mu(F^* \setminus F_i)}{\mu(F_i \setminus F^*)}\\[2pt]
{}&{}= 
\lin .
\end{align*}
In addition, by the induction hypothesis,
\begin{align*}
\Delta_{i+1} 
{}&{}= 
f_{10}(c_i, d_i, D_i, F_i) \\[2pt]
{}&{}= 
\frac{- f_3(c_i, F_i) \cdot \big( \mu(F^*) - \mu(F_i) \big) + (\lin - d_i - D_i) \cdot \mu(F^* \setminus F_i)}{\mu(F_i \setminus F^*)} \\[2pt]
{}&{}< 
\frac{- (\lin - d_i - D_i) \cdot \big( \mu(F^*) - \mu(F_i) \big)}{\mu(F_i \setminus F^*)}+
\frac{(\lin - d_i - D_i) \cdot \mu(F^* \setminus F_i)}{\mu(F_i \setminus F^*)} \\[2pt]
{}&{}= 
(\lin - d_i - D_i) \cdot \left( \frac{- \big( \mu(F^*) - \mu(F_i) \big) + \mu(F^* \setminus F_i)}{\mu(F_i \setminus F^*)} \right) \\[2pt]
{}&{}= 
\lin - d_i - D_i\\[2pt]
{}&{}\le 0 ,
\end{align*}
so we are done by the induction hypothesis.
\medskip

\noindent \textbf{Case~\hyperlink{4.2.2.1}{4.2.2.1}.} We have $D_{i+1} = D_i + \Delta_{i+1} = D_i + (\lout - D_i) = \lout$
and $d_{i+1} + D_{i+1} = d_i + \delta_{i+1} + \lout = d_i + f_9(c_i, D_i, F_i) + \lout \le \uin$. Note that
\begin{align*}
\lout - D_i 
{}&{}> 
f_{10}(c_i, d_i, D_i, F_i)= 
\frac{- f_3(c_i, F_i) \cdot \big( \mu(F^*) - \mu(F_i) \big) + (\lin - d_i - D_i) \cdot \mu(F^* \setminus F_i)}{\mu(F_i \setminus F^*)}.
\end{align*}
Thus, we have
\begin{align*}
d_{i+1} + D_{i+1}
{}&{}= 
d_i + \delta_{i+1} + D_i + \Delta_{i+1} \\[2pt] 
{}&{}=
d_i + f_9(c_i, D_i, F_i) + D_i + (\lout - D_i) \\[2pt]
{}&{}= 
d_i + \frac{c_i(F^*) - c_i(F_i) - (\lout - D_i) \cdot \big( \mu(F^*) - \mu(F_i) \big)}{\mu(F^* \setminus F_i)} + D_i + (\lout - D_i) \\[2pt]
{}&{}= 
d_i + D_i + f_3(c_i, F_i) \cdot \frac{\mu(F^*) - \mu(F_i)}{\mu(F^* \setminus F_i)} + (\lout - D_i) \cdot \left( - \frac{\mu(F^*) - \mu(F_i)}{\mu(F^* \setminus F_i)} + 1 \right) \\[2pt]
{}&{}= 
d_i + D_i + f_3(c_i, F_i) \cdot \frac{\mu(F^*) - \mu(F_i)}{\mu(F^* \setminus F_i)} + (\lout - D_i) \cdot \frac{\mu(F_i \setminus F^*)}{\mu(F^* \setminus F_i)} \\[2pt]
{}&{}> 
d_i + D_i + f_3(c_i, F_i) \cdot \frac{\mu(F^*) - \mu(F_i)}{\mu(F^* \setminus F_i)} \\[2pt]
{}&{}~~+ \frac{- f_3(c_i, F_i) \cdot \big( \mu(F^*) - \mu(F_i) \big) + (\lin - d_i - D_i) \cdot \mu(F^* \setminus F_i)}{\mu(F_i \setminus F^*)}\cdot \frac{\mu( F_i \setminus F^*)}{\mu(F^* \setminus F_i)}\\[2pt]
{}&{}= d_i + D_i + (\lin - d_i - D_i)\\[2pt]
{}&{}= \lin .
\end{align*}

\noindent \textbf{Case~\hyperlink{5.1.1}{5.1.1}.} We have $d_{i+1} + D_{i+1} = d_i + \delta_{i+1} + D_i + \Delta_{i+1} = d_i + f_7(c_i, X_i, F_i) + D_i + f_{12}(c_i, X_i, F_i) \ge \lin$ and $D_{i+1} = D_i + \Delta_{i+1} = D_i + f_{12}(c_i, X_i, F_i) \ge \lout$. By Claim~\ref{claim:inv_min_cost_span_reformulating_Delta_iplusone} and Claim~\ref{claim:inv_min_cost_span_correction_for_a_pair_v2},
\begin{align*}
\Delta_{i+1} 
{}&{}= 
f_{12}(c_i, X_i, F_i)\\[2pt]
{}&{}= 
\frac{c_i(F^*) - c_i(X_i) - f_7(c_i, X_i, F_i) \cdot \mu(F^* \setminus X_i)}{\mu(F^*) - \mu(X_i)} \\[2pt]
{}&{}= 
\frac{- f_7(c_i, X_i, F_i) \cdot \mu(F^* \setminus X_i)}{\mu(F^*) - \mu(X_i)}\\[2pt]
{}&{} \le 
0.
\end{align*}
Thus, by the induction hypothesis, $D_{i+1} = D_i + \Delta_{i+1} \le D_i \le \uout$. In addition, by the above and by Claim~\ref{claim:inv_min_cost_span_nonneg_of_delta},
\begin{align*}
\delta_{i+1} + \Delta_{i+1} 
{}&{}= 
f_7(c_i, X_i, F_i) + f_{12}(c_i, X_i, F_i) \\[2pt]
{}&{}= 
f_7(c_i, X_i, F_i) + \frac{- f_7(c_i, X_i, F_i) \cdot \mu(F^* \setminus X_i)}{\mu(F^*) - \mu(X_i)} \\[2pt]
{}&{}= 
f_7(c_i, X_i, F_i) \cdot \left( 1 - \frac{\mu(F^* \setminus F_i)}{\mu(F^*) - \mu(F_i)} \right) \\[2pt]
{}&{}= 
f_7(c_i, X_i, F_i) \cdot \frac{\mu(F_i \setminus F^*)}{\mu(F^*) - \mu(F_i)}\\[2pt]
{}&{}\le 
0 ,
\end{align*}
so by the induction hypothesis, $d_{i+1} + D_{i+1} = d_i + D_i + (\delta_{i+1} + \Delta_{i+1}) \le d_i + D_i \le \uin$.
\medskip

\noindent \textbf{Case~\hyperlink{5.1.2.1}{5.1.2.1}.} We have $D_{i+1} = D_i + \Delta_{i+1} = D_i + (\lout - D_i) = \lout$ and $d_{i+1} + D_{i+1} = d_i + \delta_{i+1} + \lout = d_i + f_9(c_i, D_i, F_i) + \lout \le \uin$.
In addition,
\begin{align*}
&d_{i+1} + D_{i+1}\\
{}&{}~~= 
d_i + \delta_{i+1} + D_i + \Delta_{i+1} \\[2pt] 
{}&{}~~=
d_i + f_9(c_i, D_i, F_i) + D_i + (\lout - D_i) \\[2pt]
{}&{}~~= 
d_i + \frac{c_i(F^*) - c_i(F_i) - (\lout - D_i) \cdot \big( \mu(F^*) - \mu(F_i) \big)}{\mu(F^* \setminus F_i)} + D_i + (\lout - D_i) \\[2pt]
{}&{}~~= 
d_i + D_i + \frac{c_i(F^*) - c_i(F_i)}{\mu(F^* \setminus F_i)} + (\lout - D_i) \cdot \left( - \, \frac{\big( \mu(F^*) - \mu(F_i) \big)}{\mu(F^* \setminus F_i)} + 1 \right) \\[2pt]
{}&{}~~= 
d_i + D_i + \frac{c_i(F^*) - c_i(F_i)}{\mu(F^* \setminus F_i)} + (\lout - D_i) \cdot \frac{\mu(F_i \setminus F^*)}{\mu(F^* \setminus F_i)} \\[2pt]
{}&{}~~> 
d_i + D_i + \frac{c_i(F^*) - c_i(F_i)}{\mu(F^* \setminus F_i)} + f_{12}(c_i, X_i, F_i) \cdot \frac{\mu(F_i \setminus F^*)}{\mu(F^* \setminus F_i)} \\[2pt]
{}&{}~~= 
d_i + D_i + \frac{c_i(F^*) - c_i(F_i)}{\mu(F^* \setminus F_i)}+ \frac{c_i(F^*) - c_i(F_i) - f_7(c_i, X_i, F_i) \cdot \mu(F^* \setminus X_i)}{\mu(F^*) - \mu(F_i)} \cdot \frac{\mu(F_i \setminus F^*)}{\mu(F^* \setminus F_i)} \\[2pt]
{}&{}~~= 
d_i + D_i + \frac{c_i(F^*) - c_i(F_i)}{\mu(F^* \setminus F_i)} \cdot \left( 1 + \frac{\mu(F_i \setminus F^*)}{\mu(F^*) - \mu(F_i)} \right) - f_7(c_i, X_i, F_i) \cdot \frac{\mu(F_i \setminus F^*)}{\mu(F^*) - \mu(F_i)} \\[2pt]
{}&{}~~= 
d_i + D_i + \frac{c_i(F^*) - c_i(F_i)}{\mu(F^*) - \mu(F_i)} - f_7(c_i, X_i, F_i) \cdot \frac{\mu(F_i \setminus F^*)}{\mu(F^*) - \mu(F_i)} \\[2pt]
{}&{}~~= 
d_i + D_i + \frac{c_i(F^*) - c_i(F_i)}{\mu(F^*) - \mu(F_i)} - f_7(c_i, X_i, F_i) \cdot \left( \frac{\mu(F^* \setminus F_i)}{\mu(F^*) - \mu(F_i)} - 1 \right) \\[2pt]
{}&{}~~= 
d_i + D_i + f_{12}(c_i, X_i, F_i) + f_7(e_i, X_i, F_i)\\[2pt]
{}&{}~~\ge 
\lin .
\end{align*}

\noindent \textbf{Case~\hyperlink{5.2.1}{5.2.1}.} We have $D_{i+1} = D_i + \Delta_{i+1} = D_i + f_{10}(c_i, d_i, D_i, F_i) \ge \lout$. By the induction hypothesis,
\begin{equation*}
\Delta_{i+1} = f_{10}(c_i, d_i, D_i, F_i) = \frac{c_i(F_i) - c_i(F^*) + (\lin - d_i - D_i) \cdot \mu(F^* \setminus F_i)}{\mu(F_i \setminus F^*)} \le 0,
\end{equation*}
thus, again by the induction hypothesis, $D_{i+1} = D_i + \Delta_{i+1} \le D_i \le \uout$. In addition,
\begin{align*}
d_{i+1} + D_{i+1}
{}&{}= 
d_i + \delta_{i+1} + D_i + \Delta_{i+1} \\[2pt] 
{}&{}=
d_i + f_{11}(c_i, d_i, D_i, F_i) + D_i + f_{10}(c_i, d_i, D_i, F_i) \\[2pt]
{}&{}= 
d_i + \frac{c_i(F^*) - c_i(F_i) - (\lin - d_i - D_i) \cdot \big( \mu(F^*) - \mu(F_i) \big)}{\mu(F_i \setminus F^*)} \\[2pt]
{}&{}~~+ D_i + \frac{c_i(F_i) - c_i(F^*) + (\lin - d_i - D_i) \cdot \mu(F^* \setminus F_i)}{\mu(F_i \setminus F^*)} \\[2pt]
{}&{}= d_i + D_i + (\lin - d_i - D_i) \cdot \left( - \, \frac{\mu(F^*) - \mu(F_i)}{\mu(F_i \setminus F^*)} + \frac{\mu(F^* \setminus F_i)}{\mu(F_i \setminus F^*)} \right) \\[2pt]
{}&{}= d_i + D_i + (\lin - d_i - D_i)\\[2pt]
{}&{}= \lin .
\end{align*}

\noindent \textbf{Case~\hyperlink{5.2.2.1}{5.2.2.1}.} We have $D_{i+1} = D_i + \Delta_{i+1} = D_i + (\lout - D_i) = \lout$ and $d_{i+1} + D_{i+1} = d_i + \delta_{i+1} + \lout = d_i + f_9(c_i, D_i, F_i) + \lout \le \uin$. In addition, 
\begin{align*}
&d_{i+1} + D_{i+1}\\[2pt]
{}&{}~~= 
d_i + \delta_{i+1} + D_i + \Delta_{i+1} \\[2pt] 
{}&{}~~=
d_i + f_9(c_i, D_i, F_i) + D_i + (\lout - D_i) \\[2pt]
{}&{}~~= 
d_i + \frac{c_i(F^*) - c_i(F_i) - (\lout - D_i) \cdot \big( \mu(F^*) - \mu(F_i) \big)}{\mu(F^* \setminus F_i)} + D_i + (\lout - D_i) \\[2pt]
{}&{}~~= 
d_i + D_i + \frac{c_i(F^*) - c_i(F_i)}{\mu(F^* \setminus F_i)} + (\lout - D_i) \cdot \left( - \, \frac{\mu(F^*) - \mu(F_i)}{\mu(F^* \setminus F_i)} + 1 \right) \\[2pt]
{}&{}~~= 
d_i + D_i + \frac{c_i(F^*) - c_i(F_i)}{\mu(F^* \setminus F_i)} + (\lout - D_i) \cdot \frac{\mu(F_i \setminus F^*)}{\mu(F^* \setminus F_i)} \\[2pt]
{}&{}~~= 
d_i + D_i + \frac{c_i(F^*) - c_i(F_i)}{\mu(F^* \setminus F_i)} + f_{10}(c_i, d_i, D_i, F_i) \cdot \frac{\mu( F_i \setminus F^*)}{\mu(F^* \setminus F_i)} \\[2pt]
{}&{}~~= 
d_i + D_i + \frac{c_i(F^*) - c_i(F_i)}{\mu(F^* \setminus F_i)} + \frac{c_i(F_i) - c_i(F^*) + (\lin - d_i - D_i) \cdot \mu(F^* \setminus F_i)}{\mu(F_i \setminus F^*)} \cdot \frac{\mu(F_i \setminus F^*)}{\mu(F^* \setminus F_i)} \\[2pt]
{}&{}~~= 
d_i + D_i + (\lin - d_i - D_i)\\[2pt]
{}&{}~~= 
\lin .
\end{align*}
This concludes the proof of the lemma.
\end{proof}

\clearpage
%%%%%%%%%%%%%%%%
\section{Toy examples} 
\label{sec:appendixb}
%%%%%%%%%%%%%%%%

\begin{figure}[h!]
 \centering 
 
 \begin{subfigure}[b]{0.245\textwidth}
  \centering

  \captionsetup{width=0.9\linewidth}
  \caption{Case~\hyperlink{1.1}{1.1}: $\delta_0\!=\!0$, $\Delta_0\!=\!0$, $\delta_1\!=\!1$, $\Delta_1\!=\!0$.}
 \end{subfigure}\hfill
 \begin{subfigure}[b]{0.245\textwidth}
   \centering
  %
  \captionsetup{width=0.95\linewidth}
  \caption{Case~\hyperlink{1.2.1}{1.2.1}: $\delta_0\!=\!0$, $\Delta_0\!=\!0$, $\delta_1\!=\!1$, $\Delta_1\! =\! -1$.}
 \end{subfigure}\hfill
 \begin{subfigure}[b]{0.245\textwidth}
   \centering
  %
  \captionsetup{width=0.9\linewidth}
  \caption{Case~\hyperlink{1.2.2}{1.2.2}: $\delta_0\!=\!0$, $\Delta_0\!=\!0$.}
 \end{subfigure}\hfill
 \begin{subfigure}[b]{0.248\textwidth}
  \centering
  %
  \captionsetup{width=0.9\linewidth}
  \caption{Case~\hyperlink{2.1.1}{2.1.1}: $\delta_0\!=\!0$, $\Delta_0\!=\!0$, $\delta_1\!=\!0$, $\Delta_1\!=\!2$.}
 \end{subfigure}\hfill
 \\[15pt]
 \begin{subfigure}[b]{0.31\textwidth}
\centering
  %
  \captionsetup{width=0.95\linewidth}
  \caption{Case~\hyperlink{2.1.2.1}{2.1.2.1}: $\delta_0\!=\!0$, $\Delta_0\!=\!0$, $\delta_1\!=\!1$, $\Delta_1\!=\!-1$, $\delta_2\!=\!1, \Delta_2\!=\!-1$ (first step is in Case~\hyperlink{1.2.1}{1.2.1}).}
 \end{subfigure}\hfill
 \begin{subfigure}[b]{0.31\textwidth}
  \centering
  %
  \captionsetup{width=0.95\linewidth}
  \caption{Case~\hyperlink{2.1.2.2}{2.1.2.2}, Subcase 1: $\delta_0\!=\!0$, $\Delta_0\!=\!0$, $\delta_1\!=\!1, \Delta_1\!=\!-1$ (first step is in Case~\hyperlink{1.2.1}{1.2.1}).}
 \end{subfigure}\hfill
 \begin{subfigure}[b]{0.31\textwidth}
\centering
  %
  \captionsetup{width=0.95\linewidth}
  \caption{Case~\hyperlink{2.1.2.2}{2.1.2.2}, Subcase 2: $\delta_0\!=\!0$, $\Delta_0\!=\!0$, $\delta_1\!=\!1$, $\Delta_1\!=\!-1$ (first step is in Case~\hyperlink{1.2.1}{1.2.1}).}
 \end{subfigure}\hfill
 \\[15pt]
 \begin{subfigure}[b]{0.245\textwidth}
  \centering
  %
  \captionsetup{width=0.9\linewidth}
  \caption{Case~\hyperlink{2.2.1}{2.2.1}: $\delta_0\!=\!0$, $\Delta_0\!=\!0$, $\delta_1\!=\!1$, $\Delta_1\!=\!0$.}
 \end{subfigure}\hfill
 \begin{subfigure}[b]{0.245\textwidth}
\centering
  %
  \captionsetup{width=0.95\linewidth}
  \caption{Case~\hyperlink{2.2.2.1}{2.2.2.1}: $\delta_0\!=\!0$, $\Delta_0\!=\!0$, $\delta_1\!=\!2$, $\Delta_1\!=\!-2$.}
 \end{subfigure}\hfill
 \begin{subfigure}[b]{0.245\textwidth}
  \centering
  %
  \captionsetup{width=0.9\linewidth}
  \caption{Case~\hyperlink{2.2.2.2}{2.2.2.2}, Subcase 1: $\delta_0\!=\!0$, $\Delta_0\!=\!0$.}
 \end{subfigure}\hfill
 \begin{subfigure}[b]{0.245\textwidth}
 \centering
  %
  \captionsetup{width=0.9\linewidth}
  \caption{Case~\hyperlink{2.2.2.2}{2.2.2.2}, Subcase 2: $\delta_0\!=\!0$, $\Delta_0\!=\!0$.}
 \end{subfigure}\hfill
 \\[15pt]
 \begin{subfigure}[b]{0.3\textwidth}
  %
  \captionsetup{width=0.93\linewidth}
  \caption{Case~\hyperlink{3.1.1}{3.1.1}: $\delta_0\!=\!0$, $\Delta_0\!=\!0$, $\delta_1\!=\!0$, $\Delta_1\!=\!-2$, $\delta_2\!=\!1$, $\Delta_1\!=\!2$ (first step is in Case~\hyperlink{4.1.1}{4.1.1}).}
 \end{subfigure}\hfill
 \begin{subfigure}[b]{0.3\textwidth}
  %
  \captionsetup{width=0.965\linewidth}
  \caption{Case~\hyperlink{3.1.2.1}{3.1.2.1}: $\delta_0\!=\!0$, $\Delta_0\!=\!0$, $\delta_1\!=\!0$, $\Delta_1\!=\!-2$, $\delta_2\!=\!1$, $\Delta_2\!=\!0$ (first step is in Case~\hyperlink{4.1.1}{4.1.1}).}
 \end{subfigure}\hfill
 \begin{subfigure}[b]{0.3\textwidth}
  \centering
  %
  \captionsetup{width=0.95\linewidth}
  \caption{Case~\hyperlink{3.1.2.2}{3.1.2.2}, Subcase 1: $\delta_0\!=\!0$, $\Delta_0\!=\!0$, $\delta_1\!=\!0, \Delta_1\!=\!-2$ (first step is in Case~\hyperlink{4.1.1}{4.1.1}).}
 \end{subfigure}\hfill
\end{figure}
\begin{figure} \ContinuedFloat
 \centering
 \begin{subfigure}[b]{0.29\textwidth}
 \centering
  %
  \captionsetup{width=0.98\linewidth}
  \caption{Case~\hyperlink{3.1.2.2}{3.1.2.2}, Subcase 2: $\delta_0\!=\!0$, $\Delta_0\!=\!0$, $\delta_1\!=\!1$, $\Delta_1\!=\!-2$ (first step is in Case~\hyperlink{4.2.2.1}{4.2.2.1}).}
 \end{subfigure}\hfill
 \begin{subfigure}[b]{0.35\textwidth}
 \centering
  %
  \captionsetup{width=0.9\linewidth}
  \caption{Case~\hyperlink{3.2.1}{3.2.1}: $\delta_0\!=\!0$, $\Delta_0\!=\!0$, $\delta_1\!=\!0$, $\Delta_1\!=\!-2$, $\delta_2\!=\!1$, $\Delta_2\!=\!2$ (first step is in Case~\hyperlink{4.1.1}{4.1.1}).}
 \end{subfigure}\hfill
 \begin{subfigure}[b]{0.35\textwidth}
 \centering
  %
  \captionsetup{width=0.9\linewidth}
  \caption{Case~\hyperlink{3.2.2.1}{3.2.2.1}: $\delta_0\!=\!0$, $\Delta_0\!=\!0$, $\delta_1\!=\!0$, $\Delta_1\!=\!-1$, $\delta_2\!=\!1$, $\Delta_2\!=\!0$ (first step is in Case~\hyperlink{4.1.1}{4.1.1}).}
 \end{subfigure}\hfill
 \\[15pt]
 \begin{subfigure}[t]{0.35\textwidth}
 \centering
  %
  \captionsetup{width=0.9\linewidth}
  \caption{Case~\hyperlink{3.2.2.2}{3.2.2.2}, Subcase 1: $\delta_0\!=\!0$, $\Delta_0\!=\!0$, $\delta_1\!=\!0$, $\Delta_1\!=\!-1$ (first step is in Case~\hyperlink{4.1.1}{4.1.1}).}
 \end{subfigure}\hfill
 \begin{subfigure}[t]{0.35\textwidth}
 \centering
  %
  \captionsetup{width=0.95\linewidth}
  \caption{Case~\hyperlink{3.2.2.2}{3.2.2.2}, Subcase 2: $\delta_0\!=\!0$, $\Delta_0\!=\!0$, $\delta_1\!=\!1$, $\Delta_1\!=\!-5$ (first step is in Case~\hyperlink{4.2.1}{4.2.1}).}
 \end{subfigure}\hfill
 \begin{subfigure}[t]{0.245\textwidth}
 \centering
  %
  \captionsetup{width=0.9\linewidth}
  \caption{Case~\hyperlink{4.1.1}{4.1.1}: $\delta_0\!=\!0$, ${\Delta_0\!=\!0}$, $\delta_1\!=\!0$, $\Delta_1\!=\!-2$.}
 \end{subfigure}\hfill
 \\[15pt]
\begin{subfigure}[b]{0.32\textwidth}
 \centering
  %
  \captionsetup{width=0.9\linewidth}
  \caption{Case~\hyperlink{4.1.2.1}{4.1.2.1}: $\delta_0\!=\!0$, $\Delta_0\!=\!0$, $\delta_1\!=\!1$, $\Delta_1\!=\!0$, $\delta_2\!=\!1$, $\Delta_2\!=\!0$ (first step is in Case~\hyperlink{1.1}{1.1}).}
 \end{subfigure}\hfill
 \begin{subfigure}[b]{0.32\textwidth}
 \centering
  %
  \captionsetup{width=0.9\linewidth}
  \caption{Case~\hyperlink{4.1.2.2}{4.1.2.2}, Subcase 1: $\delta_0\!=\!0$, $\Delta_0\!=\!0$, $\delta_1\!=\!1$, $\Delta_1\!=\!0$ (first step is in Case~\hyperlink{1.1}{1.1}).}
 \end{subfigure}\hfill
 \begin{subfigure}[b]{0.32\textwidth}
 \centering
  %
  \captionsetup{width=0.9\linewidth}
  \caption{Case~\hyperlink{4.1.2.2}{4.1.2.2}, Subcase 2: $\delta_0\!=\!0$, $\Delta_0\!=\!0$, $\delta_1\!=\!1$, $\Delta_1\!=\!0$ (first step is in Case~\hyperlink{1.1}{1.1}).}
 \end{subfigure}\hfill
 \\[15pt]
 \begin{subfigure}[b]{0.245\textwidth}
\centering
  %
  \captionsetup{width=0.95\linewidth}
  \caption{Case~\hyperlink{4.2.1}{4.2.1}: $\delta_0\!=\!0$, ${\Delta_0\!=\!0}$, $\delta_1\!=\!2$, $\Delta_1\!=\!-2$.}
 \end{subfigure}\hfill
 \begin{subfigure}[b]{0.245\textwidth}
\centering
  %
  \captionsetup{width=0.95\linewidth}
  \caption{Case~\hyperlink{4.2.2.1}{4.2.2.1}: $\delta_0\!=\!0$, ${\Delta_0\!=\!0}$, $\delta_1\!=\!1$, $\Delta_1\!=\!0$.}
 \end{subfigure}\hfill
 \begin{subfigure}[b]{0.245\textwidth}
 \centering
  %
  \captionsetup{width=0.95\linewidth}
  \caption{Case~\hyperlink{4.2.2.2}{4.2.2.2}, Subcase 1: $\delta_0\!=\!0$, ${\Delta_0\!=\!0}$.}
 \end{subfigure}\hfill
 \begin{subfigure}[b]{0.245\textwidth}
 \centering
  %
  \captionsetup{width=0.95\linewidth}
  \caption{Case~\hyperlink{4.2.2.2}{4.2.2.2}, Subcase 2: $\delta_0\!=\!0$, ${\Delta_0\!=\!0}$.}
 \end{subfigure}\hfill
 \end{figure}
\begin{figure} \ContinuedFloat
 \centering
 \begin{subfigure}[b]{0.32\textwidth}
 \centering
  %
  \captionsetup{width=0.9\linewidth}
  \caption{Case~\hyperlink{5.1.1}{5.1.1}: $\delta_0\!=\!0$, $\Delta_0\!=\!0$, $\delta_1\!=\!0$, $\Delta_1\!=\!2$, $\delta_2\!=\!1$, $\Delta_1\!=\!-2$ (first step is in Case~\hyperlink{2.1.1}{2.1.1}).}
 \end{subfigure}\hfill
 \begin{subfigure}[b]{0.32\textwidth}
  \centering
  %
  \captionsetup{width=0.9\linewidth}
  \caption{Case~\hyperlink{5.1.2.1}{5.1.2.1}: $\delta_0\!=\!0$, $\Delta_0\!=\!0$, $\delta_1\!=\!0$, $\Delta_1\!=\!3$, $\delta_2\!=\!3$, $\Delta_1\!=\!-3$ (first step is in Case~\hyperlink{2.1.1}{2.1.1}).}
 \end{subfigure}\hfill
 \begin{subfigure}[b]{0.32\textwidth}
 \centering
  %
  \captionsetup{width=0.9\linewidth}
  \caption{Case~\hyperlink{5.1.2.2}{5.1.2.2}, Subcase 1: $\delta_0\!=\!0$, $\Delta_0\!=\!0$, $\delta_1\!=\!0$, $\Delta_1\!=\!2$ (first step is in Case~\hyperlink{2.1.1}{2.1.1}).}
 \end{subfigure}\hfill
 \\[15pt]
 \begin{subfigure}[b]{0.35\textwidth}
  \centering
  %
  \captionsetup{width=0.9\linewidth}
  \caption{Case~\hyperlink{5.1.2.2}{5.1.2.2}, Subcase 2: $\delta_0\!=\!0$, $\Delta_0\!=\!0$, $\delta_1\!=\!5$, $\Delta_1\!=\!0$ (first step is in Case~\hyperlink{2.2.1}{2.2.1}).}
 \end{subfigure}\hspace{1cm}
 \begin{subfigure}[b]{0.35\textwidth}
  \centering
  %
  \captionsetup{width=0.95\linewidth}
  \caption{Case~\hyperlink{5.2.1}{5.2.1}: $\delta_0\!=\!0$, $\Delta_0\!=\!0$, $\delta_1\!=\!2$, $\Delta_1\!=\!-2$, $\delta_2\!=\!1$, $\Delta_2\!=\!-1$ (first step is in Case~\hyperlink{2.2.2.1}{2.2.2.1}).}
 \end{subfigure}\hfill
 \\[15pt]
 \begin{subfigure}[b]{0.338\textwidth}
  \centering
  %
  \captionsetup{width=0.9\linewidth}
  \caption{Case~\hyperlink{5.2.2.1}{5.2.2.1}: $\delta_0\!=\!0$, $\Delta_0\!=\!0$, $\delta_1\!=\!0$, $\Delta_1\!=\!3$, $\delta_1\!=\!3$, $\Delta_1\!=\!-3$ (first step is in Case~\hyperlink{2.1.1}{2.1.1}).}
 \end{subfigure}\hfill
  \begin{subfigure}[b]{0.3\textwidth}
 \centering
  %
  \captionsetup{width=0.9\linewidth}
  \caption{Case~\hyperlink{5.2.2.2}{5.2.2.2}, Subcase 1: $\delta_0\!=\!0$, $\Delta_0\!=\!0$, $\delta_1\!=\!1$, $\Delta_1\!=\!0$ (first step is in Case~\hyperlink{2.2.1}{2.2.1}).}
 \end{subfigure}\hfill
 \begin{subfigure}[b]{0.35\textwidth}
  \centering
  %
  \captionsetup{width=0.9\linewidth}
  \caption{Case~\hyperlink{5.2.2.2}{5.2.2.2}, Subcase 2: $\delta_0\!=\!0$, $\Delta_0\!=\!0$, $\delta_1\!=\!1$, $\Delta_1\!=\!-1$ (first step is in Case~\hyperlink{2.2.2.1}{2.2.2.1}).}
 \end{subfigure}\hfill
 \caption{Toy examples illustrating the different cases occurring in Algorithm~\ref{algo:inv_min_cost_span}, where the values of the original cost function $c$ are presented in the top row. Note that $\delta_0=\Delta_0=0$ in all cases, hence $c_0\equiv c$.} 
 \label{fig:cases}
\end{figure}

\clearpage
%%%%%%%%%%%%%%%%
\section{Summary of the cases occurring in Algorithm~\ref{algo:inv_min_cost_span}}
\label{sec:appendixc}
%%%%%%%%%%%%%%%%

\begin{sideways}
%\begin{landscape}
%\begin{center}
{\small
\setlength{\tabcolsep}{3.4pt}
\renewcommand{\arraystretch}{1.05}
\begin{tabular}{!{\vrule width 1.5pt} c !{\vrule width 1.5pt} c|ccccccc !{\vrule width 1.5pt} c !{\vrule width 1.5pt} c !{\vrule width 1.5pt}} \specialrule{1.5pt}{0pt}{0pt}
 \textbf{Case}                       & \multicolumn{8}{c !{\vrule width 1.5pt}}{\textbf{Conditions}} & $\boldsymbol{\delta_{i+1}}$ & $\boldsymbol{\Delta_{i+1}}$ \\ \specialrule{1.5pt}{0pt}{0pt}
 \multirow{2}{*}{\hyperlink{1.1}{1.1}\hphantom{.0.0}} & \multirow{6}{*}{\rotatebox{90}{$\mu(F_i) = \mu(F^*)$}} & \multirow{2}{*}{\parbox{0.2\textwidth}{\centering $f_1(c_i, F_i)$ \\ $+ d_i + D_i \le \uin$}} &  &  &  &  &  &  & \multirow{2}{*}{$f_1(c_i, F_i)$} & \multirow{2}{*}{0} \\
                                     & & & & & & & & & & \\ \cline{1-1} \cline{3-11}
 \multirow{2}{*}{\hyperlink{1.2.1}{1.2.1}\hphantom{.0}} &  & \multirow{2}{*}{\parbox{0.2\textwidth}{\centering $f_1(c_i, F_i)$ \\ $+ d_i + D_i > \uin$}} &  & \multirow{2}{*}{\parbox{0.2\textwidth}{\centering $f_2(c_i, d_i, D_i, F_i)$ \\ $+ D_i \ge \lout$}} &  &  &  &  & \multirow{2}{*}{$f_1(c_i, F_i)$} & \multirow{2}{*}{$f_2(c_i, d_i, D_i, F_i)$} \\
                                     & & & & & & & & & & \\ \cline{1-1} \cline{3-11}
 \multirow{2}{*}{\hyperlink{1.2.2}{1.2.2}\hphantom{.0}} &  & \multirow{2}{*}{\parbox{0.2\textwidth}{\centering $f_1(c_i, F_i)$ \\ $+ d_i + D_i > \uin$}} &  & \multirow{2}{*}{\parbox{0.2\textwidth}{\centering $f_2(c_i, d_i, D_i, F_i)$ \\ $+ D_i < \lout$}} &  &  &  &  & \multicolumn{2}{c !{\vrule width 1.5pt}}{\multirow{2}{*}{\texttt{Infeasible}}} \\
                                     & & & & & & & & & \multicolumn{2}{c !{\vrule width 1.5pt}}{} \\ \hline
 \multirow{2}{*}{\hyperlink{2.1.1}{2.1.1}\hphantom{.0}} & \multirow{12}{*}{\rotatebox{90}{\parbox{120pt}{\centering $\mu(F_i) < \mu(F^*)$ \\ $Z_i = \ast$}}} & \multirow{2}{*}{\parbox{0.2\textwidth}{\centering $f_3(c_i, F_i)$ \\ $+ D_i \le \uout$}} &  & \multirow{2}{*}{\parbox{0.2\textwidth}{\centering $f_3(c_i, F_i)$ \\ $+ d_i + D_i \le \uin$}} &  &  &  &  & \multirow{2}{*}{0} & \multirow{2}{*}{$f_3(c_i, F_i)$} \\
                                     & & & & & & & & & & \\ \cline{1-1} \cline{3-11}
 \multirow{2}{*}{\hyperlink{2.1.2.1}{2.1.2.1}} &  & \multirow{2}{*}{\parbox{0.2\textwidth}{\centering $f_3(c_i, F_i)$ \\ $+ D_i \le \uout$}} &  & \multirow{2}{*}{\parbox{0.2\textwidth}{\centering $f_3(c_i, F_i)$ \\ $+ d_i + D_i > \uin$}} &  & \multirow{2}{*}{\parbox{0.13\textwidth}{\centering $F_i\setminus S_0$ \\ $\nsubseteq F^*\setminus S_0$}} &  & \multirow{2}{*}{\parbox{0.2\textwidth}{\centering $f_4(c_i, d_i, D_i, F_i)$ \\ $+ D_i \ge \lout$}} & \multirow{2}{*}{$f_5(c_i, d_i, D_i, F_i)$} & \multirow{2}{*}{$f_4(c_i, d_i, D_i, F_i)$} \\
                                     & & & & & & & & & & \\ \cline{1-1} \cline{3-11}
 \multirow{2}{*}{\hyperlink{2.1.2.2}{2.1.2.2}} &  & \multirow{2}{*}{\parbox{0.2\textwidth}{\centering $f_3(c_i, F_i)$ \\ $+ D_i \le \uout$}} &  & \multirow{2}{*}{\parbox{0.2\textwidth}{\centering $f_3(c_i, F_i)$ \\ $+ d_i + D_i > \uin$}} &  & \multicolumn{3}{c !{\vrule width 1.5pt}}{\multirow{2}{*}{\parbox{0.35\textwidth}{\centering either $F_i\setminus S_0 \subseteq F^*\setminus S_0$, or \\ $f_4(c_i, d_i, D_i, F_i) + D_i < \lout$}}} & \multicolumn{2}{c !{\vrule width 1.5pt}}{\multirow{2}{*}{\texttt{Infeasible}}} \\
                                     & & & & & & & & & \multicolumn{2}{c !{\vrule width 1.5pt}}{} \\ \cline{1-1} \cline{3-11}
 \multirow{2}{*}{\hyperlink{2.2.1}{2.2.1}\hphantom{.0}} &  & \multirow{2}{*}{\parbox{0.2\textwidth}{\centering $f_3(c_i, F_i)$ \\ $+ D_i > \uout$}} &  & \multirow{2}{*}{\parbox{0.2\textwidth}{\centering $f_6(c_i, d_i, F_i)$ \\ $+ d_i \le \uin - \uout$}} &  &  &  &  & \multirow{2}{*}{$f_6(c_i, D_i, F_i)$} & \multirow{2}{*}{$\uout - D_i$} \\
                                     & & & & & & & & & & \\ \cline{1-1} \cline{3-11}
 \multirow{2}{*}{\hyperlink{2.2.2.1}{2.2.2.1}} &  & \multirow{2}{*}{\parbox{0.2\textwidth}{\centering $f_3(c_i, F_i)$ \\ $+ D_i > \uout$}} &  & \multirow{2}{*}{\parbox{0.2\textwidth}{\centering $f_6(c_i, d_i, F_i)$ \\ $+ d_i > \uin - \uout$}} &  & \multirow{2}{*}{\parbox{0.13\textwidth}{\centering $F_i\setminus S_0$ \\ $\nsubseteq F^*\setminus S_0$}} &  & \multirow{2}{*}{\parbox{0.2\textwidth}{\centering $f_4(c_i, d_i, D_i, F_i)$ \\ $+ D_i \ge \lout$}} & \multirow{2}{*}{$f_5(c_i, d_i, D_i, F_i)$} & \multirow{2}{*}{$f_4(c_i, d_i, D_i, F_i)$} \\
                                     & & & & & & & & & & \\ \cline{1-1} \cline{3-11}
 \multirow{2}{*}{\hyperlink{2.2.2.2}{2.2.2.2}} &  & \multirow{2}{*}{\parbox{0.2\textwidth}{\centering $f_3(c_i, F_i)$ \\ $+ D_i > \uout$}} &  & \multirow{2}{*}{\parbox{0.2\textwidth}{\centering $f_6(c_i, d_i, F_i)$ \\ $+ d_i > \uin - \uout$}} &  & \multicolumn{3}{c !{\vrule width 1.5pt}}{\multirow{2}{*}{\parbox{0.35\textwidth}{\centering either $F_i\setminus S_0 \subseteq F^*\setminus S_0$, or \\ $f_4(c_i, d_i, D_i, F_i) + D_i < \lout$}}} & \multicolumn{2}{c !{\vrule width 1.5pt}}{\multirow{2}{*}{\texttt{Infeasible}}} \\
                                     & & & & & & & & & \multicolumn{2}{c !{\vrule width 1.5pt}}{} \\ \cline{1-1} \hline
 \multirow{3}{*}{\hyperlink{3.1.1}{3.1.1}\hphantom{.0}} & \multirow{15}{*}{\rotatebox{90}{\parbox{120pt}{\centering $\mu(F_i) < \mu(F^*)$ \\ $Z_i \ne \ast$}}} & \multirow{3}{*}{\parbox{0.2\textwidth}{\centering $f_8(c_i, F_i, Z_i)$ \\ $+ D_i \le \uout$}} &  & \multirow{3}{*}{\parbox{0.2\textwidth}{\centering $f_7(c_i, F_i, Z_i)$ \\ $+ f_8(c_i, F_i, Z_i)$ \\ $+ d_i + D_i \le \uin$}} &  &  &  &  & \multirow{3}{*}{$f_7(c_i, F_i, Z_i)$} & \multirow{3}{*}{$f_8(c_i, F_i, Z_i)$} \\
                                     & & & & & & & & & & \\
                                     & & & & & & & & & & \\ \cline{1-1} \cline{3-11}
 \multirow{3}{*}{\hyperlink{3.1.2.1}{3.1.2.1}} &  & \multirow{3}{*}{\parbox{0.2\textwidth}{\centering $f_8(c_i, F_i, Z_i)$ \\ $+ D_i \le \uout$}} &  & \multirow{3}{*}{\parbox{0.2\textwidth}{\centering $f_7(c_i, F_i, Z_i)$ \\ $+ f_8(c_i, F_i, Z_i)$ \\ $+ d_i + D_i > \uin$}} &  & \multirow{3}{*}{\parbox{0.13\textwidth}{\centering $F_i\setminus S_0$ \\ $\nsubseteq F^*\setminus S_0$}} &  & \multirow{3}{*}{\parbox{0.2\textwidth}{\centering $f_4(c_i, d_i, D_i, F_i)$ \\ $+ D_i \ge \lout$}} & \multirow{3}{*}{$f_5(c_i, d_i, D_i, F_i)$} & \multirow{3}{*}{$f_4(c_i, d_i, D_i, F_i)$} \\
                                     & & & & & & & & & & \\
                                     & & & & & & & & & & \\ \cline{1-1} \cline{3-11}
 \multirow{3}{*}{\hyperlink{3.1.2.2}{3.1.2.2}} &  & \multirow{3}{*}{\parbox{0.2\textwidth}{\centering $f_8(c_i, F_i, Z_i)$ \\ $+ D_i \le \uout$}} &  & \multirow{3}{*}{\parbox{0.2\textwidth}{\centering $f_7(c_i, F_i, Z_i)$ \\ $+ f_8(c_i, F_i, Z_i)$ \\ $+ d_i + D_i > \uin$}} &  & \multicolumn{3}{c !{\vrule width 1.5pt}}{\multirow{3}{*}{\parbox{0.35\textwidth}{\centering either $F_i\setminus S_0 \subseteq F^*\setminus S_0$, or \\ $f_4(c_i, d_i, D_i, F_i) + D_i < \lout$}}} & \multicolumn{2}{c !{\vrule width 1.5pt}}{\multirow{3}{*}{\texttt{Infeasible}}} \\
                                     & & & & & & & & & \multicolumn{2}{c !{\vrule width 1.5pt}}{} \\
                                     & & & & & & & & & \multicolumn{2}{c !{\vrule width 1.5pt}}{} \\ \cline{1-1} \cline{3-11}
 \multirow{2}{*}{\hyperlink{3.2.1}{3.2.1}\hphantom{.0}} &  & \multirow{2}{*}{\parbox{0.2\textwidth}{\centering $f_8(c_i, F_i, Z_i)$ \\ $+ D_i > \uout$}} &  & \multirow{2}{*}{\parbox{0.2\textwidth}{\centering $f_6(c_i, d_i, F_i)$ \\ $+ d_i \le \uin - \uout$}} &  &  &  &  & \multirow{2}{*}{$f_6(c_i, D_i, F_i)$} & \multirow{2}{*}{$\uout - D_i$} \\
                                     & & & & & & & & & & \\ \cline{1-1} \cline{3-11}
 \multirow{2}{*}{\hyperlink{3.2.2.1}{3.2.2.1}} &  & \multirow{2}{*}{\parbox{0.2\textwidth}{\centering $f_8(c_i, F_i, Z_i)$ \\ $+ D_i > \uout$}} &  & \multirow{2}{*}{\parbox{0.2\textwidth}{\centering $f_6(c_i, d_i, F_i)$ \\ $+ d_i > \uin - \uout$}} &  & \multirow{2}{*}{\parbox{0.13\textwidth}{\centering $F_i\setminus S_0$ \\ $\nsubseteq F^*\setminus S_0$}} &  & \multirow{2}{*}{\parbox{0.2\textwidth}{\centering $f_4(c_i, d_i, D_i, F_i)$ \\ $+ D_i \ge \lout$}} & \multirow{2}{*}{$f_5(c_i, d_i, D_i, F_i)$} & \multirow{2}{*}{$f_4(c_i, d_i, D_i, F_i)$} \\
                                     & & & & & & & & & & \\ \cline{1-1} \cline{3-11}
 \multirow{2}{*}{\hyperlink{3.2.2.2}{3.2.2.2}} &  & \multirow{2}{*}{\parbox{0.2\textwidth}{\centering $f_8(c_i, F_i, Z_i)$ \\ $+ D_i > \uout$}} &  & \multirow{2}{*}{\parbox{0.2\textwidth}{\centering $f_6(c_i, d_i, F_i)$ \\ $+ d_i > \uin - \uout$}} &  & \multicolumn{3}{c !{\vrule width 1.5pt}}{\multirow{2}{*}{\parbox{0.35\textwidth}{\centering either $F_i\setminus S_0 \subseteq F^*\setminus S_0$, or \\ $f_4(c_i, d_i, D_i, F_i) + D_i < \lout$}}} & \multicolumn{2}{c !{\vrule width 1.5pt}}{\multirow{2}{*}{\texttt{Infeasible}}} \\
                                     & & & & & & & & & \multicolumn{2}{c !{\vrule width 1.5pt}}{} \\
 \specialrule{1.5pt}{0pt}{0pt}
\end{tabular}
}
%\end{center}
%\end{landscape}
\end{sideways}

\begin{landscape}
{\small
\begin{center}
\setlength{\tabcolsep}{3.4pt}
\renewcommand{\arraystretch}{1.05}
\captionsetup{type=table}
\begin{tabular}{!{\vrule width 1.5pt} c !{\vrule width 1.5pt} c|ccccccc !{\vrule width 1.5pt} c !{\vrule width 1.5pt} c !{\vrule width 1.5pt}} \specialrule{1.5pt}{0pt}{0pt}
 \textbf{Case}                       & \multicolumn{8}{c !{\vrule width 1.5pt}}{\textbf{Conditions}} & $\boldsymbol{\delta_{i+1}}$ & $\boldsymbol{\Delta_{i+1}}$ \\ \specialrule{1.5pt}{0pt}{0pt}
 \multirow{2}{*}{\hyperlink{4.1.1}{4.1.1}\hphantom{.0}} & \multirow{12}{*}{\rotatebox{90}{\parbox{120pt}{\centering $\mu(F_i) > \mu(F^*)$ \\ $X_i = \ast$}}} & \multirow{2}{*}{\parbox{0.2\textwidth}{\centering $f_3(c_i, F_i)$ \\ $+ d_i + D_i \ge \lin$}} &  & \multirow{2}{*}{\parbox{0.2\textwidth}{\centering $f_3(c_i, F_i)$ \\ $+ D_i \ge \lout$}} &  &  &  &  & \multirow{2}{*}{0} & \multirow{2}{*}{$f_3(c_i, F_i)$} \\
                                     & & & & & & & & & & \\ \cline{1-1} \cline{3-11}
 \multirow{2}{*}{\hyperlink{4.1.2.1}{4.1.2.1}} &  & \multirow{2}{*}{\parbox{0.2\textwidth}{\centering $f_3(c_i, F_i)$ \\ $+ d_i + D_i \ge \lin$}} &  & \multirow{2}{*}{\parbox{0.2\textwidth}{\centering $f_3(c_i, F_i)$ \\ $+ D_i < \lout$}} &  & \multirow{2}{*}{\parbox{0.13\textwidth}{\centering $F^*\setminus S_0$ \\ $\nsubseteq F_i\setminus S_0$}} &  & \multirow{2}{*}{\parbox{0.2\textwidth}{\centering $f_9(c_i, D_i, F_i)$ \\ $+ d_i \le \uin - \lout$}} & \multirow{2}{*}{$f_9(c_i, D_i, F_i)$} & \multirow{2}{*}{$\lout - D_i$} \\
                                     & & & & & & & & & & \\ \cline{1-1} \cline{3-11}
 \multirow{2}{*}{\hyperlink{4.1.2.2}{4.1.2.2}} &  & \multirow{2}{*}{\parbox{0.2\textwidth}{\centering $f_3(c_i, F_i)$ \\ $+ d_i + D_i \ge \lin$}} &  & \multirow{2}{*}{\parbox{0.2\textwidth}{\centering $f_3(c_i, F_i)$ \\ $+ D_i < \lout$}} &  & \multicolumn{3}{c !{\vrule width 1.5pt}}{\multirow{2}{*}{\parbox{0.35\textwidth}{\centering either $F^*\setminus S_0 \subseteq F_i\setminus S_0$, or \\ $f_9(c_i, D_i, F_i) + d_i > \uin - \lout$}}} & \multicolumn{2}{c !{\vrule width 1.5pt}}{\multirow{2}{*}{\texttt{Infeasible}}} \\
                                     & & & & & & & & & \multicolumn{2}{c !{\vrule width 1.5pt}}{} \\ \cline{1-1} \cline{3-11}
 \multirow{2}{*}{\hyperlink{4.2.1}{4.2.1}\hphantom{.0}} &  & \multirow{2}{*}{\parbox{0.2\textwidth}{\centering $f_3(c_i, F_i)$ \\ $+ d_i + D_i < \lin$}} &  & \multirow{2}{*}{\parbox{0.2\textwidth}{\centering $f_{10}(c_i, d_i, D_i, F_i)$ \\ $+ D_i \ge \lout$}} &  &  &  &  & \multirow{2}{*}{$f_{11}(c_i, d_i, D_i, F_i)$} & \multirow{2}{*}{$f_{10}(c_i, d_i, D_i, F_i)$} \\
                                     & & & & & & & & & & \\ \cline{1-1} \cline{3-11}
 \multirow{2}{*}{\hyperlink{4.2.2.1}{4.2.2.1}} &  & \multirow{2}{*}{\parbox{0.2\textwidth}{\centering $f_3(c_i, F_i)$ \\ $+ d_i + D_i < \lin$}} &  & \multirow{2}{*}{\parbox{0.2\textwidth}{\centering $f_{10}(c_i, d_i, D_i, F_i)$ \\ $+ D_i < \lout$}} &  & \multirow{2}{*}{\parbox{0.13\textwidth}{\centering $F^*\setminus S_0$ \\ $\nsubseteq F_i\setminus S_0$}} &  & \multirow{2}{*}{\parbox{0.2\textwidth}{\centering $f_9(c_i, D_i, F_i)$ \\ $+ d_i \le \uin - \lout$}} & \multirow{2}{*}{$f_9(c_i, D_i, F_i)$} & \multirow{2}{*}{$\lout - D_i$} \\
                                     & & & & & & & & & & \\ \cline{1-1} \cline{3-11}
 \multirow{2}{*}{\hyperlink{4.2.2.2}{4.2.2.2}} &  & \multirow{2}{*}{\parbox{0.2\textwidth}{\centering $f_3(c_i, F_i)$ \\ $+ d_i + D_i < \lin$}} &  & \multirow{2}{*}{\parbox{0.2\textwidth}{\centering $f_{10}(c_i, d_i, D_i, F_i)$ \\ $+ D_i < \lout$}} &  & \multicolumn{3}{c !{\vrule width 1.5pt}}{\multirow{2}{*}{\parbox{0.35\textwidth}{\centering either $F^*\setminus S_0 \subseteq F_i\setminus S_0$, or \\ $f_9(c_i, D_i, F_i) + d_i > \uin - \lout$}}} & \multicolumn{2}{c !{\vrule width 1.5pt}}{\multirow{2}{*}{\texttt{Infeasible}}} \\
                                     & & & & & & & & & \multicolumn{2}{c !{\vrule width 1.5pt}}{} \\ \hline
 \multirow{3}{*}{\hyperlink{5.1.1}{5.1.1}\hphantom{.0}} & \multirow{18}{*}{\rotatebox{90}{\parbox{120pt}{\centering $\mu(F_i) > \mu(F^*)$ \\ $X_i \ne \ast$}}} & \multirow{3}{*}{\parbox{0.2\textwidth}{\centering $f_7(c_i, X_i, F_i)$ \\ $+ f_{12}(c_i, X_i, F_i)$ \\ $+ d_i + D_i \ge \lin$}} &  & \multirow{3}{*}{\parbox{0.2\textwidth}{\centering $f_{12}(c_i, X_i, F_i)$ \\ $+ D_i \ge \lout$}} &  &  &  &  & \multirow{3}{*}{$f_7(c_i, X_i, F_i)$} & \multirow{3}{*}{$f_{12}(c_i, X_i, F_i)$} \\
                                     & & & & & & & & & & \\
                                     & & & & & & & & & & \\ \cline{1-1} \cline{3-11}
 \multirow{3}{*}{\hyperlink{5.1.2.1}{5.1.2.1}} &  & \multirow{3}{*}{\parbox{0.2\textwidth}{\centering $f_7(c_i, X_i, F_i)$ \\ $+ f_{12}(c_i, X_i, F_i)$ \\ $+ d_i + D_i \ge \lin$}} &  & \multirow{3}{*}{\parbox{0.2\textwidth}{\centering $f_{12}(c_i, X_i, F_i)$ \\ $+ D_i < \lout$}} &  & \multirow{3}{*}{\parbox{0.13\textwidth}{\centering $F^*\setminus S_0$ \\ $\nsubseteq F_i\setminus S_0$}} &  & \multirow{3}{*}{\parbox{0.2\textwidth}{\centering $f_9(c_i, D_i, F_i)$ \\ $+ d_i \le \uin - \lout$}} & \multirow{3}{*}{$f_9(c_i, D_i, F_i)$} & \multirow{3}{*}{$\lout - D_i$} \\
                                     & & & & & & & & & & \\
                                     & & & & & & & & & & \\ \cline{1-1} \cline{3-11}
 \multirow{3}{*}{\hyperlink{5.1.2.2}{5.1.2.2}} &  & \multirow{3}{*}{\parbox{0.2\textwidth}{\centering $f_7(c_i, X_i, F_i)$ \\ $+ f_{12}(c_i, X_i, F_i)$ \\ $+ d_i + D_i \ge \lin$}} &  & \multirow{3}{*}{\parbox{0.2\textwidth}{\centering $f_{12}(c_i, X_i, F_i)$ \\ $+ D_i < \lout$}} &  & \multicolumn{3}{c !{\vrule width 1.5pt}}{\multirow{3}{*}{\parbox{0.35\textwidth}{\centering either $F^*\setminus S_0 \subseteq F_i\setminus S_0$, or \\ $f_9(c_i, D_i, F_i) + d_i > \uin - \lout$}}} & \multicolumn{2}{c !{\vrule width 1.5pt}}{\multirow{3}{*}{\texttt{Infeasible}}} \\
                                     & & & & & & & & & \multicolumn{2}{c !{\vrule width 1.5pt}}{} \\
                                     & & & & & & & & & \multicolumn{2}{c !{\vrule width 1.5pt}}{} \\ \cline{1-1} \cline{3-11}
\multirow{3}{*}{\hyperlink{5.2.1}{5.2.1}\hphantom{.0}} &  & \multirow{3}{*}{\parbox{0.2\textwidth}{\centering $f_7(c_i, X_i, F_i)$ \\ $+ f_{12}(c_i, X_i, F_i)$ \\ $+ d_i + D_i < \lin$}} &  & \multirow{3}{*}{\parbox{0.2\textwidth}{\centering $f_{10}(c_i, d_i, D_i, F_i)$ \\ $+ D_i \ge \lout$}} &  &  &  &  & \multirow{3}{*}{$f_{11}(c_i, d_i, D_i, F_i)$} & \multirow{3}{*}{$f_{10}(c_i, d_i, D_i, F_i)$} \\
                                     & & & & & & & & & & \\
                                     & & & & & & & & & & \\ \cline{1-1} \cline{3-11}
 \multirow{3}{*}{\hyperlink{5.2.2.1}{5.2.2.1}} &  & \multirow{3}{*}{\parbox{0.2\textwidth}{\centering $f_7(c_i, X_i, F_i)$ \\ $+ f_{12}(c_i, X_i, F_i)$ \\ $+ d_i + D_i < \lin$}} &  & \multirow{3}{*}{\parbox{0.2\textwidth}{\centering $f_{10}(c_i, d_i, D_i, F_i)$ \\ $+ D_i < \lout$}} &  & \multirow{3}{*}{\parbox{0.13\textwidth}{\centering $F^*\setminus S_0$ \\ $\nsubseteq F_i\setminus S_0$}} &  & \multirow{3}{*}{\parbox{0.2\textwidth}{\centering $f_9(c_i, D_i, F_i)$ \\ $+ d_i \le \uin - \lout$}} & \multirow{3}{*}{$f_9(c_i, D_i, F_i)$} & \multirow{3}{*}{$\lout - D_i$} \\
                                     & & & & & & & & & & \\
                                     & & & & & & & & & & \\ \cline{1-1} \cline{3-11}
 \multirow{3}{*}{\hyperlink{5.2.2.2}{5.2.2.2}} &  & \multirow{3}{*}{\parbox{0.2\textwidth}{\centering $f_7(c_i, X_i, F_i)$ \\ $+ f_{12}(c_i, X_i, F_i)$ \\ $+ d_i + D_i < \lin$}} &  & \multirow{3}{*}{\parbox{0.2\textwidth}{\centering $f_{10}(c_i, d_i, D_i, F_i)$ \\ $+ D_i < \lout$}} &  & \multicolumn{3}{c !{\vrule width 1.5pt}}{\multirow{3}{*}{\parbox{0.35\textwidth}{\centering either $F^*\setminus S_0 \subseteq F_i\setminus S_0$, or \\ $f_9(c_i, D_i, F_i) + d_i > \uin - \lout$}}} & \multicolumn{2}{c !{\vrule width 1.5pt}}{\multirow{3}{*}{\texttt{Infeasible}}} \\
                                     & & & & & & & & & \multicolumn{2}{c !{\vrule width 1.5pt}}{} \\
                                     & & & & & & & & & \multicolumn{2}{c !{\vrule width 1.5pt}}{} \\
 \specialrule{1.5pt}{0pt}{0pt}
\end{tabular}
\captionof{table}{Summary of the cases occurring in Algorithm~\ref{algo:inv_min_cost_span}.}
\label{table:summary}
\end{center}
}
\end{landscape}

\end{document}